\documentclass[11pt,reqno]{amsart}
\usepackage{a4wide}

\numberwithin{equation}{section}
\usepackage[english]{babel}
\usepackage[T1]{fontenc}
\usepackage[latin1]{inputenc}
\usepackage{indentfirst}
\usepackage{enumitem}
\usepackage{amsmath,amssymb, amsbsy}
\usepackage{amsfonts}
\usepackage{hyperref}
\hypersetup{colorlinks,breaklinks,
            linkcolor=[rgb]{0,0,0},
            citecolor=[rgb]{0,0,0},
            urlcolor=[rgb]{0,0,0}}

\usepackage{latexsym}
\usepackage{amsthm}
\usepackage[dvips]{graphicx}
\usepackage{xcolor}
\usepackage{tikz}
\usepackage{pgfplots}
\usetikzlibrary{intersections}
\usepackage{tikz-3dplot}
\usepackage[outline]{contour}
\DeclareGraphicsExtensions{.pdf,.png,.jpg,.eps}

\usepackage{color}
\usepackage[latin1]{inputenc}
\usepackage[active]{srcltx}
\definecolor{Arancio}{cmyk}{0,0.61,0.87,0}

\definecolor{blus}{RGB}{0,102,204}

\newcommand{\brd}[1]{\mathbb{#1}}
\newcommand{\R}{\brd{R}}
\newcommand{\C}{\brd{C}}
\newcommand{\N}{\brd{N}}
\def\O{\Omega}

\newcommand{\norm}[2]{\left\Vert {#1} \right\Vert_{#2}}
\newcommand{\eps}{\varepsilon}
\newcommand{\be}{\begin{equation}}
\newcommand{\ee}{\end{equation}}
\newcommand{\loc}{{\text{\tiny{loc}}}}

\newcommand\intn{- \hspace{-0.40cm} \int}

\newtheorem{teo}{Theorem}[section]
\newtheorem{Corollary}[teo]{Corollary}
\newtheorem{Lemma}[teo]{Lemma}
\newtheorem{Theorem}[teo]{Theorem}
\newtheorem{Proposition}[teo]{Proposition}
\theoremstyle{definition}
\newtheorem{Definition}[teo]{Definition}

\newtheorem{remark}[teo]{Remark}

\pgfplotsset{compat=1.17}
\begin{document}

\title[A priori regularity estimates for equations degenerating on nodal sets]{A priori regularity estimates\\ for equations degenerating on nodal sets}
\thanks{{\bf Acknowledgments.}
The authors are research fellows of Istituto Nazionale di Alta Matematica INDAM group GNAMPA. G.T. and S.V. are supported by the GNAMPA project E5324001950001 \emph{PDE ellittiche che degenerano su variet\`a di dimensione bassa e frontiere libere molto sottili}. S.T. is supported by the PRIN project 20227HX33Z \emph{Pattern formation in nonlinear phenomena}. S.V. is supported by the PRIN project 2022R537CS \emph{$NO^3$ - Nodal Optimization, NOnlinear elliptic equations, NOnlocal geometric problems, with a focus on regularity}.
}

\author{Susanna Terracini, Giorgio Tortone and Stefano Vita}

\address{Susanna Terracini\newline\indent
Dipartimento di Matematica "Giuseppe Peano"
\newline\indent
Universit\`a degli Studi di Torino
\newline\indent
Via Carlo Alberto 10, 10124, Torino, Italy}
\email{susanna.terracini@unito.it}
\address{Giorgio Tortone\newline\indent
Dipartimento di Matematica "Giuseppe Peano"
\newline\indent
Universit\`a degli Studi di Torino
\newline\indent
Via Carlo Alberto 10, 10124, Torino, Italy}
\email{giorgio.tortone@unito.it}
\address{Stefano Vita\newline\indent
Dipartimento di Matematica "Felice Casorati"
\newline\indent
Universit\`a di Pavia
\newline\indent
Via Ferrata 5, 27100, Pavia, Italy}
\email{stefano.vita@unipv.it}

\begin{abstract}
We prove \emph{a priori} and \emph{a posteriori} H\"older bounds and Schauder $C^{1,\alpha}$ estimates for continuous solutions of degenerate elliptic equations with variable coefficients of the form
\begin{equation*}
\mathrm{div}\left(|u|^a A\nabla w\right)=0\qquad\mathrm{in \ }\Omega\subset\R^2,\quad a\in\R,
\end{equation*}
where the weight $u$ is itself a solution to an elliptic equation of the type $\mathrm{div}(A \nabla u) = 0$, with $A$ a Lipschitz-continuous, uniformly elliptic matrix. The function $u$ is allowed to have a nontrivial, possibly singular nodal set.

The estimates are uniform with respect to $u$ within a class of normalized solutions having bounded Almgren frequency.
In the special case $a = 2$, our results apply to the ratio of two solutions to the same elliptic equation sharing a common zero set. Precisely, we prove higher-order boundary Harnack principles on nodal domains, via the derived Schauder estimates for the associated degenerate equations. The results are based upon a fine blow-up argument, a Liouville theorem, and quasiconformal maps. 
\end{abstract}
\maketitle

\section{Introduction}
The classical \emph{boundary Harnack principle} states that, under mild regularity assumptions on the boundary, the ratio of two positive harmonic functions both vanishing on $\partial\Omega\cap B_r$ is bounded in $\Omega\cap B_{r/2}$. This holds true in Lipschitz (\cite{Kem72,Dal77,Anc78}), non tangentially accessible (shortly NTA) (\cite{JerKen}) and H\"older domains (\cite{BanBasBur91,BasBur91}) (see also \cite{DesSav2} for a recent new proof). In the first two cases, the ratio is even H\"older continuous (a similar result holds true in the case of optimal shapes \cite{MaiTorVel}).

When the boundary is more regular, \emph{higher order boundary Harnack principles} are available, ensuring H\"older continuity also of the ratio's derivatives (\cite{DesSav,Kuk,TerTorVit1}). This theory has found a number of applications notably for the obstacle (e.g. \cite{DesSav,Kuk,Zhang1}) and Bernoulli type (e.g \cite{MazTerVel1,MazTerVel2, BMMTV}) free boundary problems.

Recently, these results have started to be extended to \emph{boundary Harnack principles on nodal domains}; these refer to pairs of solutions $u,v$ of the same elliptic PDE with variable coefficients
\begin{equation}\label{equv}
L_A u = \mathrm{div}\left(A\nabla u\right)=\frac{\partial}{\partial x_i}\left(a_{ij}\frac{\partial}{\partial x_j}u\right)=0\qquad\mathrm{in \ }B_1,
\end{equation}
sharing their zero set; that is, $Z(u)\subseteq Z(v)$ where $Z(u):=u^{-1}\{0\}$ lies in the interior of $B_1$. Again, one  seeks comparison principles between $u$ and $v$ that are expressed in terms of local H\"older estimates for their ratio $v/u$, and possibly for its derivatives. A breakthrough in this direction has been given by Logunov and Malinnikova  (\cite{LogMal1}), who showed real analyticity of the quotient $v/u$ in the framework of real analytic operators.

In this paper, we are concerned with the boundary Harnack principles on nodal domains under mild regularity assumptions on the coefficients of the PDE. We consider coefficients $A = (a_{ij})_{ij}$ that are symmetric $a_{ij}=a_{ji}$, Lipschitz continuous, and uniformly elliptic; that is,
\begin{equation}\label{UE}
\lambda|\xi|^2\leq A(x)\xi \cdot \xi \leq\Lambda|\xi|^2 \qquad \mbox{in }\, B_1, \ \mbox{for any } \, \xi \in \R^n,
\end{equation}
for some $0<\lambda\leq\Lambda<+\infty$. Then, by the standard  Schauder theory, any weak solution is of class $C^{1,1-}_\loc(B_1)$ - i.e. $C^{1,\alpha}_\loc(B_1)$ for any $\alpha\in(0,1)$ - and the nodal set $Z(u)$ splits into a regular part $R(u)$ and a singular part $S(u)$ defined as
$$
R(u)=\{x\in Z(u) \ : \ |\nabla u(x)|\neq0\},\qquad S(u)=\{x\in Z(u) \ : \ |\nabla u(x)|=0\},
$$
where $R(u)$ is locally the graph of $C^{1,1-}$ functions and $S(u)$ has Hausdorff dimension at most $(n-2)$ (see e.g. \cite{Han,HanLin2,GarLin} for more detail on the structure of the nodal set). When restricted to the regular part $R(u)$ of the nodal set of $u$, the theory fits the standard higher order boundary Harnack principle and the quotient $v/u$ can be shown to be in the same class $C^{1,1-}$ from both sides of $R(u)$ (\cite{DesSav, Kuk}) and even across (\cite{TerTorVit1}).

In contrast to classical boundary Harnack inequalities, the presence of a singular part in the nodal set of $u$ prevents the application of known techniques and requires the development of new tools even for reaching H\"older continuity for the ratio $v/u$. In this direction, Lin and Lin in \cite{LinLin} recently deduced H\"older continuity across the full nodal set $Z(u)$ for some implicit small exponent $\alpha^*\in(0,1)$ depending on the Almgren frequency of the solution $u$ in the denominator. Indeed, despite the presence of the singular set, the authors were able to construct a modified Harnack chain built on the natural scaling invariance of the singularity, reconnecting with the proof for the case of NTA domains. On the opposite end, as already mentioned, in \cite{LogMal1,LogMal2} it was shown that the ratio is real analytic, when the coefficients are real analytic, by exploiting expansions in power series and local division lemmas between polynomials.\\

More recently, in \cite{TerTorVit1}, we approached the problem by framing it within the theory of degenerate elliptic PDEs. Indeed, the ratio $w:=v/u$ solves an elliptic equation in divergence form whose coefficients do degenerate at the nodal set $Z(u)\subseteq Z(v)$:
\begin{equation}\label{eqw}
\mathrm{div}\left(u^2A\nabla w\right)=0\qquad\mathrm{in \ } B_1.
\end{equation}
Even in the simplest case of a linear function $u(x) := \xi\cdot x$, such equations are so degenerate as to be discarded and not used in the literature. Here, solutions must be intended in the sense of weak solutions in the weighted Sobolev space $H^1(B_1,u^2)$ (see  the detailed discussion on the notion of the Sobolev space in \cite[\S 3.1]{TerTorVit2}). Although in general weak solutions to \eqref{eqw} need not be continuous across the nodal set of $u$ (see \cite[Example 1.4]{SirTerVit1} for more details on the effects of strongly degenerate weights), one can see that any solution $w$ which is continuous at least in $B_1\setminus S(u)$ is actually continuous in the whole $B_1$ by \cite{LinLin}, since it can be written as $w=v/u$ where $v$ is a solution to \eqref{equv} satisfying $Z(u)\subseteq Z(v)$ (see \cite[Proposition 3.9]{TerTorVit2}).

As a main result of our previous work \cite{TerTorVit1}, despite the presence of the degenerate weight that spoils the uniform ellipticity, we succeeded in providing Schauder elliptic estimates for $w$, depending on the regularity of coefficients,  thus proving \emph{higher order boundary Harnack principles} on regular nodal domains $R(u)$. More precisely, in the same spirit of the higher order boundary Harnack proved by De Silva and Savin in \cite{DesSav}, we obtained Schauder $C^{k,\alpha}$ regularity for the ratio across the regular set $R(u)$ when coefficients are $C^{k-1,\alpha}$, and this is clearly an optimal result. Furthermore, in the two dimensional case we addressed the possible occurrence of singularities of the nodal set of $u$. We were able to  prove $C^{1,\overline\alpha-}$ regularity also across the singular set $S(u)$, in the case of Lipschitz coefficients, with an explicit $\overline\alpha\in(0,1)$ that inversely depends on the maximal vanishing order of $u$ in the domain. Observe that the degeneracy of equation \eqref{eqw} becomes more and more dramatic as the order of vanishing of $u$ increases. However, we did not know at that time whether this exponent is a natural threshold for the regularity, imposed by the presence of the singularity: in the analytic case, no such restriction appears, and ratios are analytic even across singularities.

We anticipate that in the present paper we improve the two dimensional regularity result across singular points in \cite{TerTorVit1}: the $C^{1,\overline\alpha-}$ threshold may be overcome up to the optimal $C^{1,1-}$-regularity.

\subsection{A priori uniform estimates for the ratio \texorpdfstring{$v/u$}{Lg}}
As the reader may expect, the constants in our regularity estimates in \cite{TerTorVit1} depend on the geometric structure of the nodal set $Z(u)$ and notably on the vanishing orders of the solution $u$. The aim of the present paper is to provide some \emph{a priori} bounds for the ratio $w$ that are uniform on varying $u$, and consequently $Z(u)$,
in a compact class $\mathcal{S}_{N_0}$ of solutions with bounded Almgren frequency (see \eqref{classS.intro}). In order to state our results, we need to introduce the uniformity classes of coefficients and weights for the model degenerate equation \eqref{eqwa} that will be considered in the sequel.

Given $0<\lambda\leq\Lambda$, $L>0$, the class of admissible coefficients $\mathcal A=\mathcal A_{\lambda,\Lambda,L}$ is given by
\begin{equation}\label{coeff.ipotesi}
\begin{aligned}
\mathcal A &=\left\{
A=(a_{ij})_{ij} \, : \, A=A^T, \, A(0)=\mathbb I,\, \mathrm{satisfies} \,\eqref{UE}\mbox{ and } [a_{ij}]_{C^{0,1}(B_1)}\leq L\right\}.
\end{aligned}
\end{equation}
Then, given $N_0>0$, the class $\mathcal S_{N_0}=\mathcal S_{N_0,\lambda,\Lambda,L}$ of solutions $u$ of the elliptic equation in divergence form \eqref{equv}, whose powers are the weight terms in \eqref{eqwa}, is given by
\begin{equation}\label{classS.intro}
\mathcal S_{N_0}:=\left\{u \in H^1(B_1) \, : \, u\mbox{ solves }\eqref{equv}
 \mbox{ with }  A\in\mathcal A, \,0 \in Z(u),\, N(0,u,1)\leq N_0, \, \|u\|_{L^2(B_{1/2})}=1\right\}
\end{equation}
where $N$ is the extended Almgren frequency associated with \eqref{equv} (see Section \ref{structure.nodal} for a precise definition of $N$ and more details). In the case of matrices $A \in \mathcal{A}$ with constant coefficients, the \emph{frequency function} has the form
\be\label{e:ell}
N(x_0,u,r) = \dfrac{r\int_{E_r(x_0)} A \nabla u\cdot \nabla u\,dx}{\int_{\partial E_r(x_0)} u^2\,d\sigma},\quad\mbox{where}\quad
E_r(x_0):=\left\{x \in B_1\colon |A^{-1/2}(x-x_0)|<r\right\}
\ee
are the Euclidean ellipsoids associated to the matrix $A$, comparable with the Euclidean balls in terms of the ellipticity constants $\lambda, \Lambda$. As known and explained later, the function $N$ is almost monotone in $r$ (\emph{the Almgren monotonicity formula}).
We would like to emphasize that the choice of this class is natural in order to achieve uniformity within it. Indeed, the validity of the Almgren monotonicity formula allows us to estimate the size of nodal and critical sets as well as the vanishing orders of solutions in terms of bounds on their frequency (see also \cite{LinLin,NabVal,Han,HanLin2}).

Our first result concerns uniform-in-$\mathcal{S}_{N_0}$ \emph{a priori} H\"older bounds for the gradient of the ratio $w=v/u$ of solutions to \eqref{equv} with $u\in\mathcal S_{N_0}$ and $Z(u)\subseteq Z(v)$ in two dimensions (uniform Schauder estimates). As already mentioned, this amounts to considering equation \eqref{eqw}. In the case $n=2$, it is possible to improve the \emph{a priori} H\"older bound in \cite[Theorem 1.1]{TerTorVit2} by exploiting the geometry of the singular set $S(u)$, which consists of a locally finite set of isolated points at which a finite number of regular curves meet with equal angles (see \cite{Han,Han2,HanLin2,HarSim}).

For the sake of simplicity, throughout the paper we will say that a constant $C>0$ \emph{depends on the class $\mathcal{S}_{N_0}$}, if it only depends on $n, N_0,\lambda,\Lambda$ and $L$.

\begin{Theorem}[\emph{A priori} uniform-in-$\mathcal S_{N_0}$ Schauder $C^{1,\alpha}$-estimates for the ratio in two-dimensions]\label{uniformgradientZ}
Let $n=2$, $A\in\mathcal A, u \in \mathcal{S}_{N_0}$ and $v$ be solutions to
$$
L_A u = L_A v = 0\quad\mbox{in }B_1,\quad\mbox{such that }Z(u)\subseteq Z(v).
$$
Then, if $v/u\in C^{1,\alpha}_{\loc}(B_1)$, there exists a positive constant $C$ depending only on $\mathcal{S}_{N_0}$ and $\alpha\in(0,1)$ such that
\begin{equation*}
\left\|\frac{v}{u}\right\|_{C^{1,\alpha}(B_{1/2})}\leq C\|v\|_{L^2(B_1)}.
\end{equation*}
\end{Theorem}

\subsection{A priori uniform estimates for solution to degenerate equations with general exponents}
An interesting feature of our approach is its natural extension to a larger class of singular/degenerate equations related to \eqref{eqw} with a general exponent $a\in\R$ (positive or negative)
\begin{equation}\label{eqwa}
\mathrm{div}\left(|u|^a A\nabla w\right)=0\qquad\mathrm{in \ }B_1.
\end{equation}

A clarification of the functional framework for the equation \eqref{eqwa} is in order and we refer to \cite[\S 3.1]{TerTorVit2} for a full and comprehensive treatment. Weak solutions of \eqref{eqwa} are elements of the weighted Sobolev space $H^1(B_1,|u|^adx)$. The latter is defined as the completion of smooth functions with respect to the natural weighted norm associated with the weak formulation of the problem. However, one can show that the space can be equivalently defined as the space of functions having weak gradient and with finite norm; that is, the $(H\equiv W)$ property holds true (see \cite[Proposition 3.5]{TerTorVit2}). Moreover, the interaction of the vanishing orders of $u$ and the power $a$ calls for an in-depth analysis. For example, we need a bound $a>-a_{\mathcal{S}}$, where the critical exponent $a_{\mathcal{S}}\in(0,1]$ is taken to ensure the local integrability of the weight $|u|^a$, uniformly-in-$\mathcal{S}_{N_0}$ (see \eqref{a_S} and \cite[Lemma 3.1]{TerTorVit2}). Moreover, many properties of the weighted Sobolev space depend on the matching of $a$ and $N_0$, notably the capacitary estimates of $Z(u)$ and $S(u)$ (see \cite[Lemma 3.1 and Lemma 3.3]{TerTorVit2}).

Like in the case $a=2$, we are seeking uniform regularity estimates for solutions to weighted equations as in \eqref{eqwa}. As already mentioned, the class $\mathcal S_{N_0}$   provides effective control over both the coefficients and the geometry of the nodal set of $u$ along with its vanishing orders. Below are our main results for the equations \eqref{eqwa}.

\begin{Theorem}[\emph{A priori} uniform-in-$\mathcal S_{N_0}$ estimates in two-dimensions and general exponent]\label{uniformwa}
Let $n=2$, $a>-a_{\mathcal{S}}$, $A\in\mathcal A$, $u \in \mathcal{S}_{N_0}$ and $w \in H^1(B_1,|u|^a)$ be a solution to \eqref{eqwa} in the sense of Definition \ref{definition.energy.a}. Then the following holds true:
\begin{itemize}
    \item[\rm{(i)}] if $w\in C^{0,\alpha}_{\loc}(B_1)$, then
\begin{equation*}
\left\|w\right\|_{C^{0,\alpha}(B_{1/2})}\leq C\|w\|_{L^\infty(B_1)},
\end{equation*}
for some constant $C>0$ depending only on $\mathcal{S}_{N_0}, a$ and $\alpha\in(0,1)$;
\item[\rm{(ii)}] if $a\geq0$ and $w\in C^{1,\alpha}_{\loc}(B_1)$, then
\begin{equation*}
\left\|w\right\|_{C^{1,\alpha}(B_{1/2})}\leq C\|w\|_{L^\infty(B_1)},
\end{equation*}
for some constant $C>0$ depending only on $\mathcal{S}_{N_0}, a$ and $\alpha\in(0,1)$.
\end{itemize}
\end{Theorem}

     \emph{A priori} estimates in PDEs are a useful tool to prove \emph{a posteriori} regularity: having at one's disposal a regularization-approximation scheme, the estimates then extend to any given solution of a target PDE. Here, it is not so clear how to build a general regularization-approximation scheme neither for problem \eqref{eqwa} or for the boundary Harnack problem.
        However, in Section \ref{sec:fixed} we build a regularization-approximation scheme in the two dimensions, which implies the sharp $C^{1,1-}$ \emph{a posteriori} regularity even across singular points, once more uniform-in-$\mathcal S_{N_0}$ (see Theorem \ref{effective2}).

Following the program defined in \cite{SirTerVit1, SirTerVit2,TerTorVit1} for degenerate equations, and adapting the classical approach in \cite{Sim}, the Schauder estimates in the theorems above are proved via a blow-up argument and the validity of Liouville-type theorems. However, several difficulties make the case of singular points very peculiar, especially regarding the gradient bounds.  As already observed in \cite{TerTorVit1} in the case of regular nodal sets, to overcome the presence of degenerations of the coefficients, the idea is to control the oscillation of the gradient at points close to the nodal set by \emph{hooking} them to their projections on $R(u)$ (see Section \ref{s:hooking}).

\subsection{Liouville theorem}
The previous analysis relies on some Liouville type theorems for entire solutions of
\begin{equation*}
    \mathrm{div}(|u|^a\nabla w)=0\qquad\mathrm{in} \ \R^n,
\end{equation*}
with $u$ a homogeneous harmonic polynomial. In our opinion, these have independent interest, also related to their conformal equivariance.

We would like to highlight that in the general exponent case the classification of entire profiles needs to be carried out through a new Liouville theorem (comparing to the Liouville theorem in \cite[Theorem 1.2]{TerTorVit2} which works for the case $a=2$), which is just two dimensional and relies also on the validity of unique continuation principles for degenerate PDEs (see Section \ref{sec:unique}).

\begin{Theorem}[Liouville theorem for general exponents]\label{liou_a>-1}
Let $n\geq2$, $u$ be a harmonic polynomial of degree $d$ and $a>-\min\{1,2/d\}$. Suppose that $0 \in Z(u)$ and let $w$ be an entire and continuous solution to
\begin{equation}\label{e:a-entire}
\mathrm{div}\left(|u|^a\nabla w\right)=0\qquad\mathrm{in \ }\R^n.
\end{equation}
Suppose that there exist $\gamma\geq0$ and $C>0$ such that
\begin{equation*}
|w(x)|\leq C(1+|x|)^{\gamma}\qquad\mathrm{in \ } \R^n.
\end{equation*}
Then the following points hold true:
\begin{itemize}
\item[\rm{(i)}] if $d=1$ then, up to a rotation, $w$ is a polynomial of degree at most $\lfloor \gamma\rfloor$ and symmetric with respect to $y$-variable;
\item[\rm{(ii)}] if $d\geq 2$ and $n=2$, then, up to a rotation, there exists a harmonic polynomial $P$ of degree at most $\lfloor \gamma/d\rfloor$ and symmetric with respect to $y$-variable, such that
$$
w(x,y)  =  P(\overline{u}(x,y),u(x,y)),
$$
where $\overline u$ is the harmonic conjugate of $u$ satisfying $\overline{u}(0)=0$.
\end{itemize}
\end{Theorem}

\subsection{A posteriori higher regularity across singular sets}
So far we have worked on \emph{a priori estimates}. As has often been the case in history, they can become a very useful tool in obtaining effective estimates. Thus, we finally propose a $2$-dimensional regularization-approximation scheme for the degenerate equation \eqref{eqwa} that leverages the \emph{a priori} estimate in Theorem \ref{uniformwa} part (ii) and leads to the optimal $C^{1,1-}$ effective regularity of solutions even across the singular set. When $a=2$, this gives the higher order $C^{1,1-}$ boundary Harnack principle on nodal domains.

The following theorem improves \cite[Theorem 1.5]{TerTorVit1} both in terms of H\"older exponents of the gradient of solutions and in terms of uniformity-in-$\mathcal S_{N_0}$ of the constants.

\begin{Theorem}[\emph{A posteriori} uniform-in-$\mathcal S_{N_0}$ Schauder $C^{1,1-}$-estimates in two-dimensions and general exponent]\label{effective2}
Let $n=2, a\geq 0, A\in \mathcal{A}$ and $u\in \mathcal{S}_{N_0}$. Given $w$ a continuous solution to \eqref{eqwa} in the sense of Definition \ref{definition.energy.a}, we have $w \in C^{1,1-}_\loc(B_1)$ and
$$
A\nabla w\cdot\nabla u=0 \quad\mathrm{on \ } R(u)\cap B_1,\qquad \left|\nabla w\right|=0 \quad\mathrm{on \ } S(u)\cap B_1.
$$
Moreover, for any $\alpha\in(0,1)$ there exists a constant $C>0$ depending on $\alpha,a$ and the class $\mathcal{S}_{N_0}$, such that
\begin{equation*}
\left\|w\right\|_{C^{1,\alpha}(B_{1/2})}\leq C\left\|w\right\|_{L^\infty(B_1)}.
\end{equation*}
\end{Theorem}

The regularization-approximation scheme is performed in five steps:
\begin{itemize}
    \item[(i)] Prove effective regularity for equations having coefficients whose determinant is constant (see Theorem \ref{teodim21}). This is done by refining the quasiconformal transformation introduced by Hartman and Wintner in \cite{HartWint} (see also \cite[Lemma 3.6]{TerTorVit1}), with a particular attention to its behaviour close to the singular set.
    \item[(ii)] Given $A\in\mathcal{A}$ and $u$ a $A$-harmonic function, let us focus on the problem around an isolated singularity $0 \in S(u)$ with vanishing order $N>1$. Then, construct an approximating pair $(u_\varepsilon,A_\varepsilon)$, where $A_\varepsilon \in \mathcal{A}$ is a sequence of matrices whose determinant is constant in $B_\varepsilon$; $u_\varepsilon$ is an $A_\varepsilon$-harmonic function with the same vanishing order of $u$ at $0$. Conclude the construction by showing the convergence of the blow-up limits of $u_\varepsilon$ at the fixed singularity to that of the original function $u$.
    \item[(iii)] Given a solution $w$ to \eqref{eqwa}, and localizing at a singular point of the weight $u$, construct a $\varepsilon$-approximaint PDE with coefficients and weight terms as in (ii). As $\varepsilon\to0$, select a sequence of $\eps$-approximating solutions $w_\varepsilon$, converging to $w$.
    \item[(iv)] Use the blow-up procedure introduced in the proof of Theorem \ref{uniformwa} part (ii) to make such estimate uniform-in-$\varepsilon$.
    The uniformity of the estimate for the sequence, together with the convergence, implies \emph{a posteriori} $C^{1,1-}$-regularity for $w$ with an estimate possibly depending on $u$ and $Z(u)$, see Theorem \ref{t:effective}.
    \item[(v)] Apply Theorem \ref{uniformwa} part (ii) to make the effective estimate for $w$ uniform-in-$\mathcal S_{N_0}$.
\end{itemize}

Let us remark that, in view of \cite[\S 3.4]{TerTorVit2}, the $L^\infty(B_1)$-control in the estimate of Theorem \ref{uniformwa} (i) can be replaced by the weaker $L^2(B_1,|u|^a)$-control whenever $a\in(-a_{\mathcal{S}},a_{\mathcal{S}})\cup(1,+\infty)$. For the same reason, the estimates in Theorem \ref{uniformwa} (ii) and Theorem \ref{effective2} can be improved accordingly whenever $a\in[0,a_{\mathcal{S}})\cup(1,+\infty)$.

\subsection{Structure of the paper}
Section \ref{sec:gradient} is devoted to the proof of Theorem \ref{uniformgradientZ}; that is, the \emph{a priori} uniform-in-$\mathcal S_{N_0}$ gradient estimates for the ratio of two solutions sharing zero sets. Then, in Section \ref{sec:equationA}, we generalize the previous result to the case of solutions to \eqref{eqwa} with general exponent in the weight term and we prove Theorem \ref{uniformwa}.

In Section \ref{sec:fixed}, we construct a two dimensional regularization-approximation scheme around isolated singularities of the weight, which lead to the sharp \emph{a posteriori} $C^{1,1-}$-regularity of solutions to \eqref{eqwa} when the coefficients are Lipschitz continuous, see Theorem \ref{effective2}. The regularization crucially exploits a quasiconformal hodograph transformation, which provides regularity when coefficients have constant determinant around singular points.

Finally, in Section \ref{sec:liouville}, we prove the Liouville Theorem \ref{liou_a>-1}. The latter relies on unique continuation principles for solutions to degenerate or singular equations and on a conformal hodograph transformation.

\section{A priori uniform-in-\texorpdfstring{$\mathcal{S}_{N_0}$}{Lg} gradient estimates for the ratio in two dimensions}\label{sec:gradient}
In this section, we consider the planar case $n=2$ and we improve \cite[Theorem 1.1]{TerTorVit2} by showing uniform-in-$\mathcal S_{N_0}$ local H\"{o}lder estimates for the gradient of the ratio of solutions sharing nodal sets.

\subsection{Almgren monotonicity formula and structure of the nodal set}\label{structure.nodal}
Let $n\geq 2$ and $u \in H^1(B_1)$ be a weak solution to \eqref{equv}
where $A \in \mathcal{A}$ is a symmetric, uniformly elliptic, Lipschitz continuous matrix (see \eqref{coeff.ipotesi}). In such case, by elliptic regularity any weak solution is of class $C^{1,1-}_\loc(B_1) \cap H^2_\loc(B_1)$.
 Thus, by \cite{Han} (see also \cite{HanLin2,GarLin}) the nodal set $Z(u)=u^{-1}\{0\}$ of $u$ splits into a regular part $R(u)$, which is locally a $(n-1)$-dimensional hyper-surface of class $C^{1,1-}$, and the singular part $S(u)$ which has Hausdorff dimension at most $(n-2)$.

For what it concerns the study of singular points, it is necessary to introduce the notion of vanishing order.

\begin{Definition}[Vanishing order]
Given $u \in H^1(B_1)$, the \emph{vanishing order} of $u$ at $x_0 \in B_1$ is defined as
$$
\mathcal{V}(x_0,u) = \sup\left\{\beta\geq 0: \ \limsup_{r \to 0^+} \frac1{r^{n-1+2\beta}} \int_{\partial B_r(x_0)} u^2 \,d\sigma <+\infty\right\}.
$$
\end{Definition}

The number $\mathcal{V}(x_0,u) \in [0,+\infty]$ is characterized by the property that
$$
\limsup_{r \to 0^+} \frac1{r^{n-1+2\beta}} \int_{\partial B_r(x_0)} u^2  \,d\sigma= \begin{cases} 0 & \text{if $0 <\beta< \mathcal{V}(x_0,u)$} \\
+\infty  & \text{if $\beta > \mathcal{V}(x_0,u)$}.
\end{cases}
$$
The Lipschitz continuity of the coefficients of $A$ allows to prove the strong unique continuation principle, see \cite{GarLin}, which consists in the fact that non-trivial solutions can not vanish with infinite order at $Z(u)$. Ultimately, it implies that non-trivial solutions can not vanish identically in any open subset of their reference domain,
which is the classical unique continuation principle.

Under the assumption of Lipschitz continuity of the leading coefficients, it is possible to prove the validity of a Almgren-type monotonicity formula, which is a key tool to the local analysis of solutions near their nodal set. For the sake of completeness, we recall the construction of the generalized Almgren frequency function in the case of variable coefficients, by following the general ideas in \cite{GarLin,HanLin2,Han} (see also the recent developments in \cite{CheNabVal,NabVal}). In these papers, the authors introduce a specific change of metrics to define the frequency function which turns out to be almost monotone at small scales.
Since these results are nowadays classical, we simply review the main notations and consequences of the monotonicity formula and we suggest to the interested reader to address \cite[Section 3]{CheNabVal} for more recent contributions in this direction.

By following the construction in \cite{CheNabVal}, fixed $x_0\in B_1$ we consider the function
$$
r^2(x_0,x):=a^{ij}(x_0)(x-x_0)_i(x-x_0)_j,
$$
where $x=x_i e_i$ is the classical decomposition with respect to the canonical basis of $\R^n$ and $a^{ij}$ are the entries of the inverse matrix $A^{-1}$. Then, let us consider the new metric $g(x)=g(x_0,x)$ centered at $x_0$ defined as
\begin{equation}\label{e:g}
g_{ij}(x_0,x)=\eta(x_0,x) a^{ij}(x),\quad\mbox{with}\quad
\eta(x_0,x)=\frac{a_{kl}(z)a^{ks}(x_0)a^{lt}(x_0)(x-x_0)_s(x-x_0)_t}{r^2(x_0,x)}.
\end{equation}
It is worth pointing out that the geodesic distance in the metric $g(x)$ centered at $x_0$ corresponds to $r(x_0,x)$, for every $x_0,x \in B_1$. Moreover, for every $x_0 \in B_1$ we can define the associated geodesic ball, which does coincide with the Euclidean ellipsoid associated to the matrix $A(x_0)$. More precisely, for every $\rho>0$
\begin{align*}
E_\rho(x_0)&=\{x\in B_1\colon r(x_0,x) <\rho\}\\
&=\left\{x \in B_1\colon |A^{-1/2}(x_0)(x-x_0)|<\rho\right\} = x_0 + \rho A^{-1/2}(x_0)(B_1),
\end{align*}
where for convenience of notation we write $E_\rho(x_0)$ for the geodesic ball of radius $\rho$ centered
at $x_0$ with respect to the metric $g_{x_0}$. Obviously, if $A$ has constant coefficients (resp. $A\equiv \mathbb I$), it is easy to see that the geodesic ball coincides with the ellipsoids introduced in \eqref{e:ell} (resp. $E_\rho(x_0)$ coincides with the Euclidean ball of radius $\rho$ centered at $x_0$). Moreover, since $A \in \mathcal{A}$, the geodesic ball $E_\rho(x_0)$ is well-defined in the Euclidean ball of radius $\lambda^{1/2}(1-|x_0|)\rho$.

Before giving the definition of the Almgren-type frequency function, we can rewrite equation \eqref{equv} in terms of the Laplace-Beltrami operator associated to the metric $g$. Indeed, by a direct computation, $u$ is a solution to \eqref{equv} if and only if it solves
\begin{equation}\label{e:NabVal}
\Delta_g u + \left(\nabla_g \log(|g|^{-1/2}\eta)\cdot \nabla_g u\right)_g =0,
\end{equation}
where $|g|$ is the determinant of the metric $g$ and
$$
\Delta_g u =\frac{1}{\sqrt{|g|}}\partial_i\left( \sqrt{|g|} g^{ij}\partial_j u\right)\quad\mbox{and}\quad \nabla_g u = g^{ij}\partial_j u e_i.
$$
Equivalently, we can rewrite \eqref{e:NabVal} in terms of a weighted divergence form equations, that is
$$
\mathrm{div}_g(\omega \nabla_g u) =0,\qquad\mbox{with }\quad\omega= |g|^{-1/2}\eta(x_0,x).
$$
Finally, we are ready to introduce the generalized frequency function associated to solutions to \eqref{equv}. For every $x_0 \in B_1, r \in (0,\Lambda^{-1/2}(1-|x_0|))$, set
$$
 H(x_0,u,r) := \int_{\partial E_r(x_0)} \omega u^2\, d\sigma_g,\qquad  D(x_0,u,r) := \int_{E_r(x_0)} \omega |\nabla_g u|^2_{g}  \,d V_{g}.
$$
Notice that, if $A\in \mathcal{A}$ has constant coefficients, then the previous energies take form
$$
H(x_0,u,r)= \int_{\partial E_r(x_0)} u^2\,d\sigma, \qquad D(x_0,u,r) = \int_{E_r(x_0)} A \nabla u\cdot \nabla u\,dx.
$$
By adapting the results of \cite{GarLin, CheNabVal,NabVal, HanLin2} to our notations, we can state the following monotonicity result.
\begin{Proposition}[\cite{GarLin,CheNabVal,HanLin2}]\label{prop.monoton}
Let $u$ be a solution to \eqref{equv}, $x_0 \in B_1, r_0:=\Lambda^{-1/2}(1-|x_0|)$. Then there exists $C>0$ depending only on $n,\lambda,\Lambda$ and $L$, such that the map
\be\label{e:Almgren}
N(x_0,u,\cdot) \colon r \mapsto e^{C r}\frac{r D(x_0,u,r)}{H(x_0,u,r)}
\ee
is monotone non-decreasing for  $r<r_0$. Moreover, for $0<s<r< r_0$ it holds
that
\begin{equation}\label{doub}
\left|\frac{H(x_0,u,s)}{H(x_0,u,r)}\exp\left(-2\int_s^r \frac{N(x_0,u,t)}{t}\,dt\right)-1\right|\leq C r.
\end{equation}
\end{Proposition}
We refer as generalized Almgren frequency the formula in \eqref{e:Almgren}. By the monotonicity, we immediately infer the existence of the limit
$$
N(x_0,u,0^+)=\lim_{r\to 0^+}N(x_0,u,r),\quad\mbox{for every }x_0\in B_1,
$$
and, in view of the existence and uniqueness of the tangent map proved in \cite{Han,HanLin2}, such limit coincides with the vanishing order of $u$ at $x_0$, that is
$
\mathcal{V}(x_0,u)= N(x_0,u,0^+).
$
Secondly, Proposition \ref{prop.monoton} yields to the following doubling-type estimate: by \cite[Theorem 1.3]{GarLin} (see also \cite[Section 3.1]{HanLin2}) there exists $C=C(n,\lambda,\Lambda,L)$ such that
\begin{equation}\label{tipo.doubling}
\intn_{B_{R}(x_0)}u^2 \,dx\leq C\left(\frac{R}{r}\right)^{2N(x_0,u,R)}\intn_{B_{r}(x_0)}u^2 \,dx,
\end{equation}
for every $x_0 \in B_{1/2}, 0<r\leq R\leq r_0$.\\
We conclude this preliminaries by introducing a  quantity that will be used through the paper several times. In light of the monotonicity formula, since there exists $\overline{C}>0$, depending only on $n,\lambda, \Lambda, L$, such that
$$
N(x_0,u,r)\leq \overline{C} N(0,u,1)\leq \overline{C} N_0, \quad\mbox{for every }x_0 \in B_{7/8}, r\leq 1/16,
$$
we can defined 
\begin{equation}\label{overlineN0}
    \overline N_0:=\sup_{u\in \mathcal{S}_{N_0}}\max_{x_0\in \overline{B_{7/8}}, r\leq 1/16} N(x_0,u,r) \leq  \overline{C}N_0.
\end{equation}
It is well known that $\overline{N}_0$ allows  to bound the vanishing orders of $u$ on $Z(u)\cap B_{7/8}$ uniformly-in-$\mathcal{S}_{N_0}$. Thus, we define
  \begin{equation}\label{a_S}
     a_{\mathcal{S}}:=\min\left\{1,\frac{2}{\overline N_0}\right\}\in(0,1].
 \end{equation}

We proceed now by recalling the notion of solutions to \eqref{eqwa} considered through the paper. We refer to \cite{TerTorVit2} for a detailed discussion on the functional setting associated to equations of the form \eqref{eqwa}.

\begin{Definition}\label{definition.energy.a}
Let $A \in \mathcal{A}$, $u \in \mathcal{S}_{N_0}$ and $a > -a_{\mathcal{S}}$ (with $a_{\mathcal{S}}$ as in \eqref{a_S}). Then 
\begin{enumerate}
    \item[\rm{(i)}] if $a\geq 1$, then we say that $w\in H^1(B_1,|u|^a)$ is a solution to \eqref{eqwa} in $B_1$ if
\begin{equation*}
\int_{B_1}|u|^a A\nabla w\cdot\nabla\phi\,dx=0,\quad \text{for every } \phi\in C^\infty_c(B_1)\,;
\end{equation*}
\item[\rm{(ii)}] if $a \in (-a_{\mathcal{S}},1)$, then we say that $w\in H^1(B_1,|u|^a)$ is a solution to \eqref{eqwa} in $B_1$ satisfying
\be\label{e:equation-exponent-a-neumann}
|u|^a  A \nabla w \cdot \nabla u = 0\quad\mbox{on }R(u)\cap B_1,
\ee
if, on every connected component $\O_u$ of $\{u\neq 0\}$, we have
$$
\int_{\O_u \cap B_1}|u|^a A\nabla w\cdot\nabla\phi\,dx=0,\quad \text{for every } \phi\in C^\infty_c(B_1).
$$
\end{enumerate}
\end{Definition}

In \cite[Section 3]{TerTorVit2}, the authors deepen the study of the weighted Sobolev spaces associated to powers of solution to elliptic PDE. Precisely, in \cite[Proposition 3.5]{TerTorVit2} they prove the following interplay between the exponent $a_{\mathcal{S}}$ and the characterization of the Sobolev spaces $W^{1,2}(B_1,|u|^a)$ and $H^1(B_1,|u|^a).$ We recall the statement for the sake of readability.
\begin{Proposition}\cite[Proposition 3.5]{TerTorVit2}\label{p:H=W}
Let $u\in\mathcal{S}_{N_0}, a>-a_{\mathcal{S}}$ and
\begin{enumerate}
    \item[\rm{(i)}] if $a\geq 1$, set  $
W^{1,2}(B_1,|u|^a) := \left\{w \in W^{1,1}_\loc(B_1\setminus Z(u))\colon \norm{w}{H^1(B_1,|u|^a)}<+\infty\right\};$\vspace{0.1cm}
\item[\rm{(ii)}] if $a \in [a_{\mathcal{S}},1)$, set
$W^{1,2}(B_1,|u|^a) := \left\{w \in W^{1,1}_\loc(B_1\setminus S(u))\colon \norm{w}{H^1(B_1,|u|^a)}<+\infty\right\};$\vspace{0.1cm}
\item[\rm{(iii)}] if $a\in (-a_{\mathcal{S}},a_{\mathcal{S}})$, set $
W^{1,2}(B_1,|u|^a) := \left\{w \in W^{1,1}_\loc(B_1)\colon \norm{w}{H^1(B_1,|u|^a)}<+\infty\right\}.$\vspace{0.05cm}
\end{enumerate}
Then, $W^{1,2}(B_1,|u|^a) \equiv H^1(B_1,|u|^a)$.
\end{Proposition}

We notice that, in light of the results in \cite[\S 3.2]{TerTorVit2}, Theorem \ref{uniformgradientZ} can be reformulated in terms of solution $w$ to degenerate PDEs as follows.
\begin{Theorem}\label{t:grad-version-w}
Let $n=2, A\in\mathcal A, u \in \mathcal{S}_{N_0}$ and $w \in H^1(B_1,u^2)$ be a solutions to
$$
\mathrm{div}(u^2 A\nabla w) = 0\quad\mbox{in }B_1,
$$
in the sense of Definition \ref{definition.energy.a}. Then, if $w\in C^{1,\alpha}_{\loc}(B_1)$, there exists a positive constant $C$ depending only on $\mathcal{S}_{N_0}$ and $\alpha\in(0,1)$ such that
\begin{equation*}
\left\|w\right\|_{C^{1,\alpha}(B_{1/2})}\leq C\|w\|_{L^2(B_1,u^2)}.
\end{equation*}
\end{Theorem}

Among the various peculiarities of the planar case, we remark that the structure of the singular set is more explicit. Indeed,
it is well known that the singular set $S(u)$ is a locally finite set of isolated points (see \cite{GarLin, Han}) at which the nodal sets consists in $2N$ intersecting curves. We would like to remark here that the proof of the following result is $n$-dimensional apart from what is contained in Sections \ref{s:hooking}, \ref{s:inter} which is purely $2$-dimensional.

\subsection{Proof of Theorem \ref{uniformgradientZ}}\label{s:section-proof-gradient2}
Let $A_k\in \mathcal A$, $u_k\in \mathcal S_{N_0}$ and $v_k$ be solutions to
\be\label{e:ipotesi}
L_{A_k}u_k= L_{A_k}v_k=0 \quad\mbox{in }B_1,\qquad \mbox{with }Z(u_k)\subseteq Z(v_k),
\ee
and $w_k:=v_k/u_k\in C^{1,\alpha}_{\loc}(B_1)$ be their ratio satisfying, up to normalization, $\norm{w_k}{L^\infty(B_1)}=1$.

Now, consider a radially decreasing cut-off function $\eta\in C^\infty_c(B_1)$ with $0\leq\eta\leq1$, $\eta\equiv1$ in $B_{1/2}$ and such that $\mathrm{supp}\eta=B_{5/8}=:B$. Then, if we show that $(\eta w_k)_k$ is uniformly bounded in $C^{1,\alpha}(B_1)$, using that $\eta \equiv 1$ in $B_{1/2}$, we infer the same bound for $(w_k)_k$ in $C^{1,\alpha}(B_{1/2})$. Notice also that, since for every $x_0 \in B_1\setminus B_{5/8}$ and $k>0$ we have $(\eta w_k)(x_0)=|(\nabla w_k) \eta|(x_0)=0$, then it is sufficient to ensure that the following seminorms
$$
\max_{i=1,\dots,n}\left[\partial_{i}(\eta w_k)\right]_{C^{0,\alpha}(B_1)} \leq C,
$$
uniformly as $k\to +\infty$. Thus, by contradiction suppose that, up to a subsequence,
$$
\max_{i=1,\dots,n}\left[\partial_{i}(\eta w_k)\right]_{C^{0,\alpha}(B_1)}\to +\infty,$$
that is, there exist two sequences of points $x_k,\zeta_k\in B$, such that
\begin{equation}\label{e:contr}
L_k = \frac{|\partial_{i_k}(\eta w_k)(x_k)-\partial_{i_k}(\eta w_k)(\zeta_k)|}{|x_k-\zeta_k|^\alpha}\to+\infty.
\end{equation}
Naturally, up to relabeling we can assume that $i_k = 1$ for every $k>0$. Now, given $r_k=|x_k-\zeta_k| \in (0,\text{diam} B)$ we set $\xi_k \in Z(u_k)$ to be the closest point of $Z(u_k)$ to $x_k$, that is
\be\label{e:delta_k}
\delta_k = \mathrm{dist}(x_k,Z(u_k))=|\xi_k - x_k|.
\ee
Clearly, $\xi_k \in R(u)$. Now, we introduce the following blow-up sequences centered at $\hat{x}_k\in B_1$ to be chosen later:
$$
 V_k(x)=\frac{\eta(\hat{x}_k + r_k x)}{L_kr_k^{1+\alpha}}\left(w_k(\hat{x}_k+r_k x)-w_k(x_k)\right),\,\, W_k(x)=\frac{\eta(\hat{x}_k)}{L_kr_k^{1+\alpha}}\left(w_k(\hat{x}_k+r_k x)-w_k(\hat{x}_k)\right),
$$
defined in $(B_1-\hat{x}_k)/r_k =: \O_k$. Through the proof we consider centers satisfying $\hat{x}_k\in B_{3/4}$, for $k$ sufficiently large, so that the limit domain (all the limits are up to subsequences) $\O_\infty = \lim_{k\to +\infty} \O_k$ is not empty and contains a ball (see Lemma \ref{l:involuto} for the precise choice of the centers).\\
In what follows, we collect some general property of the two sequences (see
\cite[Theorems 5.1 5.2]{SirTerVit1} and \cite[Theorem 2.1]{TerTorVit1} for more details on the features of these sequences).  It is worth noting that, in view of the definition of $V_k$, for every $x,y \in \O_k$ it holds that
\be\label{e:mancante}
|\partial_i V_k(x)- \partial_i V_k(y)| \leq |x-y|^{\alpha} + \left|\frac{w_k(\hat{x}_k)}{L_k}\right|r_k^{1-\alpha}[\eta]_{C^{0,1}}|x-y|,
\ee
for every $i=1,\dots,n$. Now, let us focus on the compactness of the sequence $V_k$: let $K$ be a compact set of $\O_\infty$ and $x,y\in K$. Then, for $k$ sufficiently large, we have $x,y\in \O_k$ and so \eqref{e:mancante} holds true, for every $i=1,\dots,n$ and $k$ sufficiently large. On the other hand, since $L_k\to +\infty, K$ is compact and $w_k$ is bounded in $L^\infty(B)$, there exists $\bar{k}=\bar{k}(K)$ such that, for every $k\geq \bar{k}$ we have
$$
\left|\frac{w_k(\hat{x}_k)}{L_k}\right|r_k^{1-\alpha}[\eta]_{C^{0,1}}|x-y|^{1-\alpha}\leq 1
$$
and so
\be\label{e:bound.above.H}
\max_{i=1,\dots,n} \left[\partial_i V_k\right]_{C^{0,\alpha}(K)} \leq 2,\qquad \mbox{for every }k\geq \bar{k}(K).
\ee
On the other hand, since
\be\label{e:bound.below.H}
\left|\partial_1 V_k(0)-\partial_1 V_k\left(\frac{x_k-\zeta_k}{r_k}\right)\right| = 1 + O\left(\frac{r_k^{1-\alpha}}{L_k}\right)
\ee
as $k\to +\infty$, it implies that on every compact set of $\O_\infty$ the quantities $[\partial_i V_k]_{C^{0,\alpha}(K)}$ are uniformly bounded by above and $[\partial_1 V_k]_{C^{0,\alpha}(K)}$ is also uniformly bounded by below. Let us then define the sequences
\be\label{e:blowlin}
\overline{V}_k(x)= V_k(x)- \nabla V_k(0)\cdot x,\qquad \overline{W}_k(x)= W_k(x)- \nabla W_k(0)\cdot x, \quad\mbox{for }x \in \O_k
\ee
where
\be\label{e:grad}
\nabla V_k(0) = \frac{\eta(\hat{x}_k)}{L_k r_k^\alpha}\nabla w_k(\hat{x}_k) = \nabla W_k(0).
\ee
The importance of such sequences lies on the fact that $|\nabla \overline{V}_k|(0) = |\nabla \overline{W}_k|(0)$. Indeed, with some few exceptions which will be treated later (see Section \ref{s:inter}), in general $|\nabla W_k(0)|$ needs not to be uniformly bounded with respect to $k$.\\
Notice that, by rewriting the definitions of $W_k$ and $\overline{W}_k$, we get
$$
w_k(\hat{x}_k + r_k x) = w_k(\hat{x}_k) + r_k(\nabla w_k)(\hat{x}_k)\cdot x + \frac{L_k r_k^{1+\alpha}}{\eta(\hat{x}_k)}\overline{W}_k(x).
$$
On the other hand, since $\overline{V}_k(0)=0= |\nabla \overline{V}_k|(0)$ and $$
[\partial_i \overline{V}_k]_{C^{0,\alpha}(K)} =[\partial_i V_k]_{C^{0,\alpha}(K)},\quad\mbox{for every }i=1,\dots,n,\quad K\subset \O_\infty\mbox{ compact},
$$
we can apply the Ascoli-Arzel\'{a} theorem and deduce that, up to consider a subsequence, $\overline{V}_k \to \overline{V}$ in $C^{1,\beta}_\loc(\O_\infty)$, for every $\beta \in (0,\alpha)$. Moreover, by \eqref{e:bound.above.H} and \eqref{e:bound.below.H}, we have that the limit $\overline{V}\in C^{1,\alpha}_\loc(\O_\infty)$ satisfies
\be\label{e:lim}
\max_{i=1,\dots,n} \left[\partial_i \overline{V}\right]_{C^{0,\alpha}_\loc(\O_\infty)} \leq 2,\qquad \partial_1 \overline{V}(0)-\partial_1 \overline{V}\left(\xi\right)=1, \quad\mbox{with }\xi = \lim_{k\to +\infty}\frac{x_k-\zeta_k}{r_k}\in \mathbb S^{n-1},
\ee
namely $\overline{V}$ has non constant gradient. Before considering the sequence $\overline{W}_k$, we need to show that both the sequences $r_k$ and $\delta_k$ (see \eqref{e:delta_k}) approach zero as $k\to +\infty$.\\

Ultimately, the following lemma allows to consider blow-up sequences centered either at a point $x_k$ associated to \eqref{e:contr} or at its closest projection $\xi_k$ on the nodal set $Z(u_k)$.
\begin{Lemma}\label{l:involuto}
Under the previous assumptions one has $r_k\to 0^+$ and $\delta_k\to 0^+$. Moreover, we can choose as centers of the blow-up analysis the following points
\be\label{e:z.cap}
\hat{x}_k = \begin{cases}
x_k,  &\mbox{if } \lim_{k\to +\infty}\frac{\delta_k}{r_k}= +\infty\\
\xi_k, &\mbox{if } \lim_{k\to +\infty}\frac{\delta_k}{r_k}<+\infty;
\end{cases}
\ee
that is, there exists $\overline{k}>0$ such that $\hat{x}_k\in B_{3/4}$ for both the cases in \eqref{e:z.cap} and $r_k\leq 1/8$ for any $k\geq\overline k$.
\end{Lemma}
\begin{proof}
Let us split the proof into two steps.\\

 Step 1: we prove that $r_k \to 0^+$. First, suppose by contradiction that $r_k\to \bar{r}>0$ and choose in the definition of the sequence $V_k$ the centers $\hat{x}_k = x_k$. The letter choice is always admissible since $x_k\in B_{3/4}$. Then, for $k$ sufficiently large, $\O_\infty$ contains an open ball $B_R$ for some $R>0$ and one gets
  $$
  \sup_{x\in \O_k}|V_k(x)| \leq \frac{2}{L_k r_k^{1+\alpha}} \norm{\eta w_k}{L^\infty(\text{supp}\eta)} \to 0,
  $$
  that is $V_k \to 0^+$ uniformly on every compact set in $\O_\infty$. Therefore, if we consider the sequence $\overline{V}_k$, it implies
  $$
  \overline{V}(x) = \lim_{k\to +\infty}\nabla V_k(0)\cdot x \quad \mbox{in }\O_\infty.
  $$
  Notice also that the sequence $(\nabla V_k(0))_k$ is bounded and, up to a subsequence, it converges to some non-trivial vector $\nu \in \R^n$, so that $\overline{V}(x) = \nu \cdot x$. Indeed, if one component $(\partial_i V_k(0))_k$ is unbounded, then
  $$
  |\overline{V}|(R e_i) = R \lim_{k\to +\infty} |(\partial_i V_k)(0)| = +\infty,
  $$
  in contradiction with the fact that $\overline{V} \in C^{1,\alpha}(B_R)$. Finally, the condition $\overline{V}(x) = \nu \cdot x$ contradicts the fact that the gradient of $\overline{V}$ is not constant.\\

  Step 2: we prove that $\delta_k \to 0^+$. If $\delta_k \to \overline{\delta}>0$ up to subsequences, having $|x_k-\zeta_k|=r_k \to 0^+$ by Step 1, it would exists $\overline{x}=\lim_{k\to +\infty}x_k$ such that
  $$
  B_{\overline{\delta}/2}(\overline{x})\cap Z(u_k)= \emptyset \qquad\mbox{and}\qquad x_k,\zeta_k\in   B_{\overline{\delta}/2}(\overline{x}),
  $$
  for every $k$ sufficiently large. Therefore, since the matrix $u^2A$ is uniformly elliptic in $B_{\overline{\delta}/2}(\overline{x})$, 
  by standard elliptic regularity we infer that $w \in C^{1,\alpha}(B_{\overline{\delta}/2}(\overline{x}))$ and so
  $$
  L_k = \frac{|\partial_{1}(\eta w_k)(x_k)-\partial_{1}(\eta w_k)(\zeta_k)|}{|x_k-\zeta_k|^\alpha} \leq \max_{i=1,\dots,n}\left[\partial_{i}(\eta w_k) \right]_{C^{0,\alpha}(B_{\overline{\delta}/2}(\overline{x}))}<C,
  $$
uniformly with respect to $k\to +\infty$. The condition above gives a contradiction for $k$ sufficiently large.\\

  We conclude the proof by stressing that the vanishing condition of the sequence $\delta_k$ makes the choice $\hat{x}_k = \xi_k$ admissible too. Indeed, as we have already mentioned, one is free to choose the centers of blow-ups $\hat{x}_k$ as long as they belong to $B_{3/4}$ for $k$ sufficiently large. Thus, $\xi_k\in B_{3/4}$ since $x_k \in B=B_{5/8}$ and $\delta_k\to 0^+$.
 \end{proof}
 Finally, summing up all the previous observations, we have showed that $\overline{V}_k$ converges in $C^{1,\beta}_\loc(\R^n)$ for every $\beta \in (0,\alpha)$, and then also uniformly on every compact set of $\R^n$, to a function $\overline{V}$ whose gradient is non-constant and globally $\alpha$-H\"{o}lder continuous in $\R^n$ (i.e. replace $\O_\infty=\R^n$ in \eqref{e:lim}).\\

In the following Lemma we show a crucial property of the two blow-up sequences, that is that they have the same asymptotic behavior as $k\to +\infty$.
\begin{Lemma}
Let $K \subset \R^n$ be a compact set, then $
\lim_{k\to +\infty}\norm{\overline{W}_k-\overline{V}_k}{L^\infty(K)}=0$.
\end{Lemma}
\begin{proof}
The proof follows immediately by the definition of the two sequences. Indeed, by \cite[Theorem 1.1]{TerTorVit2} we already know that $w_k \in C^{0,\gamma}(\text{supp}\eta)$, for every $\gamma \in (0,1)$ and uniformly in $k$. Then, in view of \eqref{e:grad}, for every $\gamma \in (0,1)$ we get that
$$
|\overline{W}_k(x)-\overline{V}_k(x)| =\frac{1}{L_k r_k^{1+\alpha}}|\eta(\hat{x}_k+r_k x) - \eta(\hat{x}_k)||w_k(\hat{x}_k+r_k x) - w_k(\hat{x}_k)|
\leq \frac{C}{L_k r_k^{1+\alpha}} r_k|z| (r_k |x|)^\gamma,
$$
which converges to zero as $k\to+\infty$.
\end{proof}
Finally, let us consider the equation satisfied by $W_k$ and $\overline{W}_k$. First, by translating and rescaling \eqref{equv}, we get
$$
\text{div}\left(U_k^2 \overline{A}_k \nabla W_k\right) = 0 \quad\mbox{in }\O_k,
$$
and consequently
$$
-\text{div}\left(U_k^2 \overline{A}_k \nabla \overline{W}_k\right) = 2U_k (\nabla U_k \cdot \overline{A}_k\nabla W_k(0)) \quad\mbox{in }\O_k,
$$
where $\nabla W_k(0)$ is defined in \eqref{e:grad}, and
\be\label{e:blow-up-finalproof}
U_k(x) = \frac{u_k(\hat{x}_k + r_k x)}{H(\hat{x}_k,u_k,r_k)^{1/2}}, \qquad \overline{A}_k(x) = A_k(\hat{x}_k + r_k x)
\ee
are the normalized blow-up sequences associated to $u_k$. We stress that the previous equations must be understood in the sense of Definition \ref{definition.energy.a}. Now, by a Caccioppoli-type inequality we have
\be\label{e:cacio-e-pepe}
\int_{\O_k}\phi^2 U_k^2|\nabla W_k|^2\,dx \leq C \int_{\O_k}U_k^2 |\nabla \phi|^2 W_k^2\,dx\leq \frac{C \norm{\eta}{L^\infty}^2}{L_k^2 r_k^{2+2\alpha-2\gamma}}\int_{\O_k} U_k^2 |\nabla \phi|^2 \,dx
\ee
for some $C$ depending on $\lambda,\Lambda$ and $[w_k]_{C^{0,\gamma}}$. We stress that $[w_k]_{C^{0,\gamma}}$ is uniformly bounded by \cite[Theorem 1.1]{TerTorVit2}. 
On the other hand, if we write the equation of $\overline{W}_k$ in its weak form, for any $\phi \in C^\infty_c(\R^n)$ with $k$ large enough so that
$\text{supp}\phi \subset B_{\tilde{R}}$, we have
\be\label{e:line}
\int_{B_{\tilde{R}}}U^2_k \nabla \phi \cdot \overline{A}_k \nabla \overline{W}_k \,dx = 2\int_{B_{\tilde{R}} } \phi  U_k  (\nabla U_k \cdot \overline{A}_k\nabla W_k(0))\,dx.
\ee
Clearly, by exploiting the compactness of the coefficients in $\mathcal{A}$ we know that there exists a constant $C$ depending on $\lambda,\Lambda$ and the uniform bound on $[w_k]_{C^{0,\gamma}}$, such that
$$
\left|\int_{B_{\tilde{R}} }U^2_k (\overline{A}_k -  \overline{A})\nabla \phi \cdot \nabla \overline{W}_k  \,dx\right| \leq r_k L \frac{C \norm{\eta}{L^\infty}^2}{L_k r_k^{2+2\alpha-2\gamma}} \int_{\O_k} U_k^2 |\nabla \phi|^2 \,dx \to 0^+
$$
where $\overline{A} \in \mathcal{A}$ is the limit of the variable coefficients. Notice that, being $\overline{A}_k \in C^{0,1}$, the limit $\overline{A}\in \mathcal{A}$ has constant coefficients. Let $U$ be the limit of the blow-up sequence $U_k$. Since $Z(U_k)\to Z(U)$ locally with respect to the Hausdorff convergence, we easily get
$$
\left|\int_{B_{\tilde{R}} }(U^2_k -U^2) \overline{A}\nabla \phi \cdot \nabla \overline{W}_k  \,dx\right| \to 0^+, \quad\mbox{for every }\phi \in C^\infty_c(\R^n \setminus Z(U)).
$$
Moreover, in light of \cite[Proposition 3.5]{TerTorVit2}, by applying Fatou's Lemma in \eqref{e:cacio-e-pepe} we deduce that the limit $\overline{W}$ of the sequence $\overline{W}_k$ belongs to $H^1_\loc(\R^n,U^2 )$. Indeed, once we prove that the linear term in \eqref{e:line} vanishes asymptotically as $k$ approaches infinity, we can completely characterize the blow-up limit of $\overline{W}_k$ and $\overline{V}_k$ as a solution to
\be\label{e:equation-lim}
\text{div}(U^2 \overline{A}\nabla \overline{W}) = 0 \quad\mbox{in }\R^n,
\ee
in the sense of Definition \ref{definition.energy.a}.
More precisely, suppose that for every $\phi \in C^\infty_c(\R^n)$ it holds
\be\label{e:terminelineare}
\left|\int_{B_{\tilde{R}} } \phi  U_k  (\nabla U_k \cdot \overline{A}_k\nabla W_k(0))\,dx\right| \to 0^+
\ee
as $k\to +\infty$. Then, summing up all the features we have obtained on the two blow-up sequences in \eqref{e:blowlin}, we have that they both converge to a function $\overline{W} \in H^1_\loc(\R^n,U^2)$ whose gradient is non-trivial and globally $\alpha$-H\"{o}lder continuous. Moreover, $\overline{W}$ satisfies the equation \eqref{e:equation-lim} in the sense of Definition \ref{definition.energy.a}.\\
By classing compactness (e.g. \cite{Lin} or \cite[Section 2.2]{NabVal}) we already know  that, up to a change of coordinates induced by $\overline{A}$, the blow-up limit $U$ is an harmonic polynomial.
Moreover, by \cite[Proposition 3.9]{TerTorVit2} the product $\overline{U} \, \overline{W}$ is an entire harmonic function satisfying $Z(\overline{U})\subseteq Z(\overline{U} \, \overline{W})$. Finally, since
$$
\overline{W}(0) = |\nabla \overline{W}|(0)=0,
$$
by exploiting the global $\alpha$-H\"{o}lder regularity of the gradient of $\overline{W}$, we get
$$|\overline{W}(x)|\leq C\left(1+|x|\right)^{1+\alpha}\quad\mbox{in }\R^n.$$
Hence, by the Liouville theorem \cite[Theorem 1.2]{TerTorVit2}, we know that $\overline{W}$ is a linear function, in contradiction with
the fact that $\overline{W}$ has non constant gradient.\\

In order to conclude the proof, we need to show the asymptotic estimate of the linear term \eqref{e:terminelineare}, for which a fine analysis of solutions in the class $\mathcal{S}_{N_0}$ is needed.\\

Before proceeding, let us recall some notation that will be used in this section. First, given $x_k$ to be the point associated to \eqref{e:contr}, we set the \emph{macroscopic variables} (as in \eqref{e:delta_k}, \eqref{e:z.cap})
$$
\delta_k = \mathrm{dist}(x_k,Z(u_k))=|\xi_k - x_k|,\qquad \mbox{and}\qquad
\hat{x}_k = \begin{cases}
x_k,  &\mbox{if } \lim_{k\to +\infty}\frac{\delta_k}{r_k}= +\infty\\
\xi_k, &\mbox{if } \lim_{k\to +\infty}\frac{\delta_k}{r_k}<+\infty,
\end{cases}
$$
where the last definition must be interpreted asymptotically, for $k$ sufficiently large.\\ On the other hand, in view of the blow-up sequences centered at $\hat{x}_k$ with rescaling factor $r_k$, we also define the \emph{microscopic variables}
\be\label{e:micro.variable}
\mathrm{p}^1_k = \frac{\xi_k - \hat{x}_k}{r_k} = \begin{cases}
0 &\mbox{if }\hat{x}_k=\xi_k,\\
\frac{\xi_k - \hat{x}_k}{r_k} & \mbox{if }\hat{x}_k=x_k,
\end{cases}\qquad d_k = \mathrm{dist}(\mathrm{p}^1_k,Z(U_k))=\begin{cases}
0,  &\mbox{if } \hat{x}_k = \xi_k\\
\frac{\delta_k}{r_k}, &\mbox{if } \hat{x}_k = x_k
\end{cases}
\ee
where $U_k$ is defined by \eqref{e:blow-up-finalproof}. Notice that either $d_k =0$ or it approaches infinity as $k\to +\infty$. Heuristically, the variables in \eqref{e:micro.variable} can be seen as the infinitesimal counterpart of the one in \eqref{e:micro.variable}, normalized with respect to the rescaling factor.

\subsection{The hooking Lemma}\label{s:hooking} In order to estimate the linear term in \eqref{e:terminelineare}, we begin by describing how the variation of the generalized Almgren frequency near regular points detects the presence of singularities and the oscillation of the normal vector to the nodal set. We would like to stress the fact that everything we prove in the following subsection is purely $2$-dimensional.\\

We proceed by stating the result for entire solutions of \eqref{equv} in $\R^2$, then we deduce its localized counterpart in Lemma \ref{l:notanera}.
\begin{Lemma}\label{l:nera1}
Let $A\in \mathcal{A}$ be a constant matrix and $u$ be an entire solution to
$$
L_A u=0 \quad\mbox{in }\R^2, \qquad \norm{u}{L^2(\partial B_1)}=1,\qquad 0 \in R(u)
$$
with Almgren frequency bounded at infinity by some $N_0 \in \R$; that is,
$$
\lim_{r\to +\infty}N(0,u,r)\leq \overline{N}.
$$
Let also $\eps>0$ and $N(0,u,1)\geq 1+\eps$.

Then there exist $\overline{R}>0$ and $\overline{\theta}>0$ depending only on $\eps$ and $\overline{N}$, such that there exists a new regular point $x_1 \in R(u)\cap \partial B_{\overline{R}}$ satisfying
$$
|\mathrm{angle}\left(\nabla u(0),\pm\nabla u(x_1)\right)|\geq \overline{\theta}.
$$
\end{Lemma}
\begin{proof}
Let us fix $\eps>0$ and consider $v(x):=u(A^{1/2}x)$, which is an entire harmonic function with bounded frequency at infinity. Hence, by applying a classical blow-down argument, we already know that $v$ is an harmonic polynomial of degree $N$, with $1+\eps \leq N\leq \overline{N}$.

Now, suppose by contradiction that there exists a sequence of harmonic polynomials $v_k$ of degree $N_k$, with $1+\eps \leq N_k\leq \overline{N}$, such that for every $x \in R(v_k)\cap B_k$, we have $$
|\mathrm{angle}(A^{-1/2}\nabla v_k(0),\pm A^{-1/2}\nabla v_k(x))|\leq \frac{1}{k}.
$$
Since $\norm{v_k}{L^2(\partial B_1)}=1$, by compactness there exists a non-trivial limit harmonic polynomial $v_\infty$ of degree $N_\infty$ (with $1+\eps \leq N_\infty\leq \overline{N}$), and a point $\xi \in \R^2\setminus \{0\}$, such that
$$
0 \in R(v_\infty)\quad\mbox{and}\quad \nabla v_\infty(x)=|\nabla v_\infty(x)|\xi,\,\mbox{ for every }x \in R(v_\infty).
$$
Notice that $\xi \neq 0$ since the matrix $A^{-1}$ is uniformly elliptic.\\
On the other hand, since $N_\infty\geq 1+\eps>1$, we know that regular set $R(v_\infty)$ shadows $2N_\infty$-lines which are invariant with respect to rotations of angle $\pi/N_\infty$. More precisely, there exist a direction $\nu \in \mathbb S^1$ and two sequences $R_k \to +\infty, \, \eta_k^i \in \mathbb S^1$ with $i=0,\dots,2N_\infty-1$, such that
$$
R_k\eta_k^i \in R(v_\infty)\cap \mathbb S^1 \quad\mbox{and}\quad
\lim_{R_k\to +\infty}\norm{\frac{\nabla v_\infty}{|\nabla v_\infty|}( R_k\eta_k^i)-\mathcal{O}_{i\frac{\pi}{N_\infty}}\nu}{L^\infty(B_1)} =0,
$$
where $\mathcal{O}_{\theta}$ is the counterclockwise rotation of angle $\theta \in [0,2\pi)$. Therefore, there exist $2N_\infty$ sequences of points in $x_k^i\in R(v_\infty)$, with $i=1,\dots,2N_\infty$ such that, for $k$ sufficiently large, we have
$$
\mathrm{angle}(\nabla v_\infty(x_k^i),\pm\nabla v_\infty(x_k^j))\geq \frac{10}{11}\frac{\pi}{2N_\infty}\geq  \frac{10}{11}\frac{\pi}{2\overline{N}},
$$
for every $i\neq j$, which leads to a contradiction.
\end{proof}
\begin{Lemma}\label{l:notanera}
Let $\eps>0$. Then there exist $\delta>0$ small and $ \overline{\rho}>1$ sufficiently large such that the following holds true: given $A \in \mathcal{A}$ and $u \in H^1(B_{\overline{\rho}})$ satisfying
$$
L_A u= 0 \quad\mbox{in }B_{\overline{\rho}},\quad 0 \in R(u)\quad\mbox{and}\quad N(0,u,\overline{\rho})\leq \overline{N},
$$
for some $\overline{N}>0$, and assuming that
$$
[A]_{C^{0,1}(B_{\overline{\rho}})}\leq \delta \quad\mbox{and}\quad N(0,u,1)\geq 1+\eps,
$$
then, there exist $\overline{R}\in (1,\overline{\rho})$ and $\overline{\theta}>0$, depending only on $\eps$ and $\overline{N}$, for which it exists a new regular point $\mathrm{q} \in R(u)\cap \partial B_{\overline{R}}(0)$ such that
$$
|\mathrm{angle}\left(\nabla u(0),\pm\nabla u(\mathrm{q})\right)|\geq \overline{\theta}.
$$
\end{Lemma}
\begin{proof}
By contradiction, suppose there exist $A_k \in \mathcal{A}, u_k \in H^1(B_k)$
satisfying
$$
L_{A_k} u_k= 0 \quad\mbox{in }B_{k},\qquad 0 \in R(u_k),\qquad N(0,u_k,k)\leq \overline{N},
$$
and such that
$$
[A_k]_{C^{0,1}(B_{k})}\leq \frac{1}{k} \quad\mbox{and}\quad N(0,u_k,1)\geq 1+\eps
$$
but the thesis of Lemma \ref{l:notanera} does not hold true. Namely, we have
$$
|\mathrm{angle}\left(\nabla u_k(0),\pm\nabla u_k(\mathrm{q})\right)|\leq \frac{1}{k},\quad \mbox{for every }\mathrm{q} \in Z(u_k)\cap B_{\overline{\rho}}.
$$
Fix $R>1$, then, up to normalizing the sequence $u_k$ with respect to the $L^2$-norm on $\partial B_1$, there exists $\overline{k}>0$ such that, for every $k>\overline{k}$,
$$
[A_k]_{C^{0,1}(\overline{B_R})}\leq \frac{1}{k},\quad L_{A_k} u_k=0\quad\mbox{in }B_R,\quad \norm{u_k}{L^2(\partial B_1)}=1.
$$
Moreover, $N(0,u_k,R)\leq N_0, N(0,u_k,1)\geq 1+\eps$ and
\be\label{e:ab}
|\mathrm{angle}\left(\nabla u_k(0),\pm\nabla  u_k(\mathrm{q})\right)|\leq \frac{1}{k},\quad \mbox{for every }\mathrm{q} \in Z(u_k)\cap B_{R}.
\ee
Therefore, by a classical diagonal argument, up to a subsequence it holds that:
\begin{enumerate}
  \item[\rm{(i)}] $A_k \to A \in \mathcal{A}$, for some constant matrix $A$;
  \item[\rm{(ii)}] $u_k \to u$ strongly in $C^{1,\alpha}_\loc(\R^2)\cap H^1(\R^2)$, where $u$ is a non-trivial entire solution to
      $$
      L_A u=0 \quad \mbox{in }\R^2\qquad\mbox{such that}\quad\norm{u}{L^\infty(B_1)}=1;
      $$
  \item[\rm{(iii)}] by the strong convergence and the monotonicity of the Almgren frequency, it holds
\be\label{e:invew}
\lim_{R\to +\infty}N(0,u,R)\leq \overline{N}\qquad\mbox{and}\qquad N(0,u,1)\geq 1+\eps.
\ee
\end{enumerate}
On the other hand, by the strong convergence, the absurd hypothesis \eqref{e:ab} implies that there exists $\xi \in \R^2\setminus \{0\}$ such that
\be\label{e:invew2}
\nabla u(x)= |\nabla u(x)|\xi, \qquad\mbox{for every }x \in R(u).
\ee
In view of \eqref{e:invew} and \eqref{e:invew2}, the contradiction follows by repeating the same proof of Lemma \ref{l:nera1} for entire solutions in $\R^2$ with constant leading coefficients.
\end{proof}

Let us proceed with the definition of the critical radius associated to the blow-up sequence \eqref{e:blow-up-finalproof}. As we will see, the validity of the hooking Lemma for the blow-up sequence \eqref{e:blow-up-finalproof} strictly relies on a careful asymptotic analysis of this peculiar radius.
      \begin{Definition}[The critical radius]\label{d:radius}
Let $\eps>0$ to be chosen later and $\overline{k}>0$ be the index of Lemma \ref{l:involuto}. Then, for every $k\geq \overline{k}$ we set
$$
\begin{aligned}
       R_k := R_k^\eps &= \mathrm{inf}\left\{r\in \left(0,\frac{1}{8 r_k}\right) \colon N(\xi_k,u_k,r_k r)\geq 1+\eps\right\}\\
       &= \mathrm{inf}\left\{r\in \left(0,\frac{1}{8 r_k}\right) \colon N(\mathrm{p}^1_k,U_k,r)\geq 1+\eps\right\}.
       \end{aligned}
       $$
        In view of Lemma \ref{l:involuto}, we already know that the Almgren frequency is well defined since $B_{R_k r_k}(\xi_k)\subset B_1$, for every $k \geq \overline{k}$.
       Clearly, for a fixed $k \geq \overline{k}$ the critical radius associated to $\eps>0$ does exist if and only if $1+\eps \leq N(\xi_k,u_k,1/8)$, which corresponds to the case in which the sequence $u_k$ does detect the presence of singular points at macroscopic scales.
       \end{Definition}
       \begin{remark}
       In \cite[Section 2.3]{NabVal} the authors introduced a similar quantity in order to construct a specific covering of the tubular neighborhood of the singular set in terms of critical balls (see \cite[Theorem 1.1]{NabVal}). Such similarity suggests a deep interplay between the two problems, where in our setting the drop of the Almgren frequency detects the oscillation of the gradient of $W_k$ close to $Z(U_k)$, due to the possible persistence of singularities at small scales.
        \end{remark}

At this point, we know that three scenarios can occur:
        \be\label{case0}
        \text{\rm{(i)} \quad either $\displaystyle{\lim_{k\to +\infty}N(\xi_k,u_k,1/8)= 1},\qquad \left(\text{i.e. } \lim_{k\to +\infty}N\left(\mathrm{p}^1_k,U_k,\frac{1}{8 r_k}\right)= 1\right)$}
        \ee
        or there exist $\eps \in (0,1)$ and $\overline{k}>0$ (possibly larger than the one in Definition \ref{d:radius}) such that it exists $R_k:=R_k^\eps \in (0,1/(8r_k))$ for which
        $$
        N(\mathrm{p}^1_k,U_k,R_k)=1+\eps,\quad\mbox{ for every }k\geq \overline{k}.
        $$
        In this case we have the two remaining possibilities:
        \begin{enumerate}
          \item[\rm{(ii)}] either there exists $M>0$ such that
          $$
          r_k R_k \geq M, \quad\mbox{for }k\geq \overline{k},
          $$
          which implies that
          \be\label{case1}
          N(\mathrm{p}^1_k,U_k,R)\leq 1+\eps,\quad\mbox{for every }R< \mathrm{min}\left\{\frac{M}{r_k},\frac{1}{8 r_k}\right\};
          \ee
          \item[\rm{(iii)}] or $r_k R_k \to 0^+$.
        \end{enumerate}
Those cases admit a more geometric interpretation in terms of the hooking Lemma of Section \ref{s:hooking}. More precisely:
\begin{enumerate}
\item[\rm{(i)}] either, from a microscopic point of view, there are no hooking points $\mathrm{p}^2_k$ resulting from an increment of any size $\eps \in (0,1)$ of the Almgren frequency centered at $\mathrm{p}^1_k$;
\end{enumerate}
or such a point exists, and this admits two behaviors:
\begin{enumerate}
\item[\rm{(ii)}] asymptotically, the two points $\mathrm{p}^1_k, \mathrm{p}^2_k$ \emph{stay apart} while maintaining a quantitative effect on the Almgren frequency of size $\eps \in (0,1)$;
\item[\rm{(iii)}] asymptotically, the two points \emph{collapse} more rapidly than the scaling factor, and therefore the presence of singularities is perceived both in terms of the frequency formula and the hooking point.
        \end{enumerate}
We stress that, in light of the monotonicity of the Almgren frequency formula, it is not restrictive to assume that the energy increase is of size $\eps \in (0,1-\alpha)$.\\
        We postpone the cases associated to \eqref{case0} and \eqref{case1} to Section \ref{s:almost-flat} and we address the latter one in which $R_k$ exists for every $k$ sufficiently large and satisfies $r_k R_k \to 0^+$.\\

        First of all, we observe that in this case, the normalized blow-up sequence \eqref{e:blow-up-finalproof} falls within the hypothesis of Lemma \ref{l:notanera}. Given $\eps>0$, let $\overline{\rho}>0,\delta >0$ be the constants in Lemma \ref{l:notanera}. By definition we already know that
$$
L_{\overline{A}_k} U_k = 0 \quad\mbox{in }\O_k \supseteq B_{\frac{1}{8r_k}},\quad
0 \in R(U_k),\quad \mathrm{p}^1_k \in R(U_k)\cap \partial B_{d_k}
$$
with
$$
[\overline{A}_k]_{C^{0,1}\left(B_{\frac{1}{8r_k}}(\mathrm{p}^1_k)\right)} = r_k [A_k]_{C^{0,1}\left(B_{\frac{1}{8}}(\xi_k)\right)} \leq L r_k.
$$
On the other hand, by classical results on the Almgren frequency (see also \cite[Theorem 3.2.10]{HanLin2}), there exists $C=C(\lambda,\Lambda)$ such that
$$
C N_0 \geq  N\left(\xi_k, u_k, \frac18\right) = N\left(\mathrm{p}^1_k,U_k,\frac{1}{8r_k}\right),
$$
for $k\geq \overline{k}$, where $\overline{k}>0$ is the index of Lemma \ref{l:involuto}. Indeed, such bounds follows by showing that for $\xi_k \in B_{3/4}$, for $k\geq \overline{k}$. On the other hand, since the critical radius $R_k=R_k^\eps$ exists and satisfies
$$
N(\mathrm{p}^1_k,U_k,R_k)\geq 1+\eps,\qquad \mbox{with } R_k r_k\to 0^+,
$$
up to consider $k$ sufficiently large so that
$$
\frac{1}{8R_k r_k}\geq \overline{\rho},\qquad L R_k r_k\leq \delta,
$$
the sequence $U_k$ falls under Lemma \ref{l:notanera}'s assumptions. More precisely, by applying Lemma \ref{l:notanera} to
$$
\overline{U}_k(x) = \frac{U_k(\mathrm{p}^1_k + R_k x)}{\norm{U_k}{L^2(\partial B_{R_k}(\mathrm{p}^1_k))}}
 $$
 and rephrasing the results in terms of $U_k$, we deduce that for $k\geq \overline{k}$ there exists $\mathrm{p}^2_k \in R(U_k)\cap \partial B_{R_k\overline{R}}(\mathrm{p}^1_k)$ such that
$$
 \mathrm{angle}\left(\nu_1,\pm\nu_2\right)\geq \overline{\theta},\qquad\mbox{where }\nu_i = \frac{\nabla U_k(\mathrm{p}^i_k)}{\norm{\nabla U_k(\mathrm{p}^i_k)}{}}
 $$
 and $\overline{R},\overline{\theta}>0$ depend only on $\eps$ and the class $\mathcal{S}_{N_0}$ (see Figure \ref{figure1}). Through the proof, we omit the dependence of $\nu_i$ on the index $k\geq \overline{k}$.
        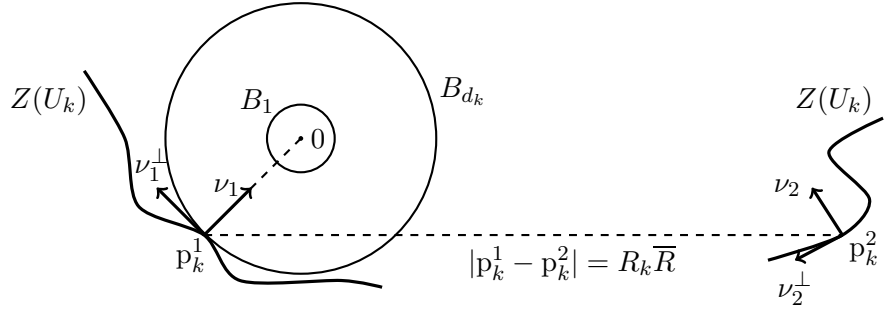
\begin{figure}[ht]
\begin{tikzpicture}[scale =0.9]
      \draw[thick,black,dashed] (0,0) -- (-1.428,-1.428);
      \draw[thick] (2,0) arc (0:360:2cm and 2cm);
      \draw[thick] (0.5,0) arc (0:360:0.5cm and 0.5cm);
      \fill circle[radius=1pt] (0,0) node [anchor=west] {$0$};
      \draw (-1.628,-1.728) node {$\mathrm{p}^1_k$};
      \draw [very thick,->] (-1.428,-1.428) -- (-0.728,-0.728) node [anchor=east] {$\nu_1$};
      \draw [very thick,->] (-1.428,-1.428) -- (-2.128,-0.728);
       \draw (-2.2,-0.78) node [anchor=south] {$\nu_1^\perp$};
      \draw (1.9,0.7) node [anchor=west] {$B_{d_k}$};
      \draw (-0.24,0.55) node [anchor=east] {$B_1$};
      \draw [black,very thick] plot [smooth] coordinates {(-3.2,1) (-2.6,0) (-2.4,-1) (-1.43,-1.428) (-0.9,-2.1) (0.6,-2.1) (1.2,-2.2)};
      \draw[thick,black,dashed] (-1.428,-1.428)--(8,-1.428);
      \draw (4,-1.428) node [anchor=north] {$|\mathrm{p}^1_k - \mathrm{p}^2_k| = R_k \overline{R}$};
      \draw [black,very thick] plot [smooth] coordinates { (6.9,-1.8) (8,-1.428) (8.4,-0.9) (7.8,-0.2) (8.6,0.3) };
      \draw (7.95,-1.62) node [anchor=west] {$\mathrm{p}^2_k$};
            \draw [very thick,->] (8,-1.428) -- (7.55,-0.728) node [anchor=east] {$\nu_2$};
            \draw [very thick,->] (8,-1.428) -- (7.3,-1.8) node [anchor=north] {$\nu_2^{\perp}$};
            \draw (8.6,0.2) node [anchor=south east] {$Z(U_k)$};
             \draw (-3,0.2) node [anchor=south east] {$Z(U_k)$};
           \end{tikzpicture} \caption{This picture describes the scenario in which the critical radius exists for a given threshold and the branches of $R(U_k)$ are connected in the sense of Lemma \ref{l:notanera}.}
\label{figure1}
       \end{figure}

Since the angle between $\nu_1$ and $\nu_2$ is not zero, it follows that they are linearly independent and, by orthogonality, also $\nu^\perp_1$ and $\nu^\perp_2$ as well.
Therefore, since $\mathrm{p}^1_k, \mathrm{p}^2_k \in R(U_k)$, by \cite[Theorem 1.3]{TerTorVit1} we have
\be\label{e:orto}
\left(\overline{A}_k \nabla W_k \right)(\mathrm{p}^i_k) \cdot \nu_i = 0,\quad\mbox{for }i=1,2,
\ee
that is $\left(\overline{A}_k \nabla W_k \right)(\mathrm{p}^i_k) \in \langle \nu_i^\perp\rangle$. On the other hand, by exploiting that $\R^2$ is spanned by $(\nu_1,\nu_2)$ and also by $ (\nu_1^\perp,\nu_2^\perp)$, we get
$$
\pi_{\nu_2^\perp}\left(\left(\overline{A}_k \nabla W_k \right)(\mathrm{p}^1_k)\right)
=0=\pi_{\nu_1^\perp}\left(\left(\overline{A}_k \nabla W_k \right)(\mathrm{p}^2_k)\right)$$
where $\pi_{V}\colon \R^2 \to \R$ is the projection along one generator $V\in \mathbb S^1$ of $\R^2$. Finally, we can show how the hooking Lemma allows us to estimate the norm of $(\overline{A}_k\nabla W_k)(0)$ in the linear term in \eqref{e:terminelineare}. First, by \eqref{e:grad} we have
$$
\left(\overline{A}_k \nabla W_k \right)(\mathrm{p}^i_k) \cdot \nu_i = 0 \quad\longleftrightarrow \quad (A_k\nabla w_k)( \hat{x}_k+ r_k \mathrm{p}^i_k)\cdot \nu_i =0.
$$
Then, since $(\overline{A}_k\nabla W_k)(0)=(\overline{A}_k\nabla V_k)(0)$ and
\be\label{e:inter}
\begin{aligned}
(\overline{A}_k\nabla W_k)(x) =&\, (\overline{A}_k\nabla V_k)(x) -\frac{1}{L_k r_k^\alpha}(w_k( \hat{x}_k+r_k x)-w_k( \hat{x}_k))(A_k\nabla \eta)( \hat{x}_k+r_k x)\\ &\,- \frac{1}{L_k r_k^\alpha}(\eta( \hat{x}_k +r_k x)-\eta( \hat{x}_k))(A_k\nabla w_k)( \hat{x}_k + r_k x),
\end{aligned}
\ee
we get, for some constant $C>0$ depending on the class $\mathcal{S}_{N_0}$, that
\be\label{e:da.W.a.V}
\begin{aligned}
|(\overline{A}_k\nabla W_k)(0)| =&\, |(\overline{A}_k\nabla V_k)(0)| \leq |\pi_{\nu_1^\perp}((\overline{A}_k\nabla V_k)(0))|+
|\pi_{\nu_2^\perp}((\overline{A}_k\nabla V_k)(0))|\\
\leq&\,  |\pi_{\nu_1^\perp}((\overline{A}_k\nabla V_k)(0)-(\overline{A}_k\nabla V_k)(\mathrm{p}^2_k))|+|\pi_{\nu_1^\perp}((\overline{A}_k\nabla V_k)(\mathrm{p}^2_k))|\\
& +|\pi_{\nu_2^\perp}((\overline{A}_k\nabla V_k)(0)-(\overline{A}_k\nabla V_k)(\mathrm{p}^1_k))|+|\pi_{\nu_2^\perp}((\overline{A}_k\nabla V_k)(\mathrm{p}^1_k))|\\
\leq &\, 2|\mathrm{p}^1_k|^\alpha + 2|\mathrm{p}^2_k|^\alpha + \frac{\Lambda}{L_k r_k^\alpha}|w_k( \hat{x}_k+r_k \mathrm{p}^1_k)-w_k( \hat{x}_k)||\nabla \eta( \hat{x}_k+r_k\mathrm{p}^1_k)|\\
& +\frac{\Lambda}{L_k r_k^\alpha}|w_k( \hat{x}_k+r_k \mathrm{p}^2_k)-w_k( \hat{x}_k)||\nabla \eta( \hat{x}_k+r_k\mathrm{p}^2_k)|\\
\leq &\, C(|\mathrm{p}^1_k|^\alpha + |\mathrm{p}^2_k|^\alpha) + \frac{1}{L_k r_k^{\alpha}}r_k^\gamma\left(|\mathrm{p}^1_k|^\gamma + |\mathrm{p}^2_k|^\gamma \right) [w_k]_{C^{0,\gamma}}\norm{\nabla \eta}{L^\infty},
\end{aligned}
\ee
where in the second inequality we used that $\nabla V_k$ are globally $\alpha$-H\"{o}lder continuous (see \eqref{e:mancante}) and in the last one that $w_k$ is uniformly $C^{0,\gamma}$ for every $\gamma \in (0,1)$. Finally, by choosing $\gamma = \alpha$ and using the definitions of $\mathrm{p}^1_k$ and $\mathrm{p}^2_k$, we infer that
\be\label{e:stimagrad}
|\nabla W_k(0)|\leq C\left(d_k+R_k\overline{R}\right)^\alpha,
\ee
for $k$ sufficiently large and with $C>0$ depending only on $\mathcal{S}_{N_0}$.\\
Moreover, in view of \eqref{e:orto} and \eqref{e:inter}, we can also control the behaviour of $\nabla V_k$ at $\mathrm{p}^1_k$. More precisely
\be \label{e:t}
\begin{aligned}
|\left(\overline{A}_k \nabla V_k \right)(\mathrm{p}^1_k)\cdot \nu_i| &\leq  \frac{1}{L_k r_k^\alpha}|w_k( \hat{x}_k+r_k \mathrm{p}^1_k)-w_k( \hat{x}_k)||(A_k\nabla \eta)( \hat{x}_k+r_k \mathrm{p}^1_k)|\\
&\leq \frac{|d_k|^\alpha }{L_k} [w_k]_{C^{0,\alpha}}\norm{\nabla \eta}{L^\infty}.
\end{aligned}
\ee

\subsection{The interplay between the hooking points \texorpdfstring{$\mathrm{p}^i_k$}{Lg}}\label{s:inter} To conclude the analysis of the cases associated with $R_k r_k \to 0^+$, we need to understand how the interplay between $d_k$ and $R_k$ allows to estimate the linear term \eqref{e:terminelineare}.\\

Case $d_k =0, R_k \leq C$ uniformly. By \eqref{e:stimagrad}, since $\overline{A}_k \in \mathcal{A}$ we deduce that $|\nabla W_k(0)|$ was \emph{a priori} uniformly bounded with respect to $k$, and so it was not necessary to introduce the normalized sequences $\overline{V}_k,\overline{W}_k$, but one can play directly with $V_k,W_k$.\\

Case $d_k=0, R_k \to +\infty$, (i.e. $\mathrm{p}^1_k=0$). Let $B_{\tilde{R}}$ be the ball containing the support of the test function $\phi$ (see \eqref{e:line}) and $\rho>0$ be such that $B_\rho \subset \O_k$, for $k$ sufficiently large.\\
If $\rho> 2\tilde{R}$, by applying a scaling invariant version of Schauder estimate, we get
$$
\begin{aligned}
\frac{1}{(2\tilde{R})^\beta}\sup_{x,y \in B_{\tilde{R}}}|\nabla U_k(x)-\nabla U_k(y)| &\leq \sup_{x,y \in B_{\tilde{R}}}\frac{|\nabla U_k(x)-\nabla U_k(y)|}{|x-y|^\beta}\\ &\leq \sup_{x,y \in B_{\rho/2}}\frac{|\nabla U_k(x)-\nabla U_k(y)|}{|x-y|^\beta} \leq \frac{C}{\rho^{1+\beta}}\sup_{B_\rho} |U_k|
\end{aligned}
$$
for any $\beta \in (0,1)$, and for some $C$ depending only on $\lambda,\Lambda,L$ and $\beta$.\\On the other hand, by combining the local boundedness of solutions via Moser's iteration and the Almgren monotonicity, we have
$$
\sup_{B_\rho} |U_k|^2 \leq
C\intn_{B_{\frac32\rho}}U_k^2\,dx \leq C H(0,U_k, 2 \rho),
$$
for some constant $C>0$ depending only on $\lambda$ and $\Lambda$.
Therefore, choosing $\rho=R_k/2$ in the inequalities above (which is admissible since $R_k/2\in (\max\{1/2,2\tilde{R}\},1/(8r_k))$ for $k$ sufficiently large), we get
$$
\frac{1}{(2\tilde{R})^\beta}\sup_{x,y \in B_{\tilde{R}}}|\nabla U_k(x)-\nabla U_k(y)|\leq \frac{C}{R_k^{1+\beta}} H(0,U_k, R_k)^{1/2}\leq
\frac{C}{R_k^{1+\beta}} H(0,U_k, 1)^{1/2} R_k^{1+\eps} = \frac{C}{R_k^{\beta-\eps}}.
$$
In the second inequality we used the doubling condition \eqref{doub} and $N(\mathrm{p}^1_k,U_k,R_k)=1+\eps$.\\
We stress that the same result can be obtained by applying the doubling condition \eqref{tipo.doubling} for the $L^2$-norm on balls. Then, by recalling \eqref{e:stimagrad}, for every $x \in B_{\tilde{R}}$ we get
$$
\begin{aligned}
|\overline{A}_k(x) \nabla W_k(0)\cdot \nabla U_k(x)|\leq &\,\,
|\overline{A}_k(x) \nabla W_k(0)||\nabla U_k(x)-\nabla U_k(0)|\\&\,+ |\overline{A}_k(x) - \overline{A}_k(0)| |\nabla W_k(0)||\nabla U_k(0)| + |(\overline{A}_k \nabla W_k)(0)\cdot \nu_i||\nabla U_k(0)|\\
\leq &\,\,\Lambda|\nabla W_k(0)||\nabla U_k(x)-\nabla U_k(0)|+ [\overline{A}_k]_{C^{0,1}(B_{\tilde{R}})} |\nabla W_k(0)||\nabla U_k(0)| \tilde{R}\\
\leq &\,\, C (R_k \overline{R})^\alpha \left( |\nabla U_k(x)-\nabla U_k(0)|+ |\nabla U_k(0)| L r_k \tilde{R}\right)\\
\leq &\,\, \frac{C}{R_k^{\beta-\eps-\alpha}} \overline{R}^\alpha + C |\nabla U_k(0)| \tilde{R} \frac{r_k R_k}{R_k^{1-\alpha}} \\
\end{aligned}
$$
where in the second inequality we used \eqref{e:orto} and in the third one \eqref{e:stimagrad}. Moreover, since $H(0,U_k,1)=1$, notice that $|\nabla U_k(0)|$ is uniformly bounded.\\

Now, since $\eps \in (0,1-\alpha)$, we can choose $\beta \in (0,1)$ so that $\beta >\eps + \alpha$. Thus, by collecting the previous estimates and since $r_kR_k \to 0^+$, it follows
$$
\left|\int_{B_{\tilde{R}} } \phi  U_k  (\overline{A}_k\nabla W_k(0)\cdot \nabla U_k)\,dx\right| \leq C \left(\frac{\overline{R}^\alpha}{R_k^{\beta-\eps-\alpha}} + |\nabla U_k(0)| \tilde{R} \frac{r_k R_k}{R_k^{1-\alpha}}\right) \int_{B_{\tilde{R}} } |\phi|  |U_k|\,dx \to 0^+,
$$
for $k\to +\infty$.\\

Case $d_k\to +\infty, R_k \to +\infty$, (i.e. $\mathrm{p}^1_k\neq 0$). First, given $x \in B_{d_k}$ and $d(x)=\mathrm{dist}(x,Z(U_k))$, by applying a scaling invariant gradient estimate we get
$$
|\nabla U_k(x)| \leq \sup_{B_{\frac{d(x)}{4}}(x)}|\nabla U_k| \leq \frac{C}{d(x)}\sup_{B_{\frac{d(x)}{2}}(x)}|U_k| \leq \frac{C}{d(x)} |U_k(x)|,\qquad\mbox{for }x \in B_{d_k}.
$$
Clearly, we can improve the previous estimate by restricting to $B_{d_k/2}$, indeed
\be\label{e:gradient.est}
|\nabla U_k(x)| \leq \frac{C}{d(x)} |U_k(x)|\leq  \frac{C}{d_k} |U_k(x)|,\qquad\mbox{for }x \in B_{d_k/2}.
\ee
Now, since both $d_k$ and $R_k$ are blowing-up, we need do proceed by taking care of the asymptotic behaviour of their ratio.\\

Suppose that there exists $M>0$ such that $R_k/d_k \leq M$ is uniformly bounded with respect to $k$. Then, if $B_{\tilde{R}}$ is a ball containing the support of the test function, by \eqref{e:stimagrad} and \eqref{e:gradient.est} we get
$$
\begin{aligned}
|\overline{A}_k(x) \nabla W_k(0)\cdot \nabla U_k(x)|&\leq C\left(d_k+R_k\overline{R}\right)^\alpha|\nabla U_k(x)|\\
&\leq C\left(1+\overline{R}\frac{R_k}{d_k}\right)^\alpha \frac{1}{d_k^{1-\alpha}}|U_k(x)|
\end{aligned}
$$
for every $x \in B_{\tilde{R}}$, where $C$ depends only on the class $\mathcal{S}_{N_0}$. Therefore, since $\alpha \in (0,1)$, we get
$$
\left|\int_{B_{\tilde{R}} } \phi  U_k  (\overline{A}_k\nabla W_k(0)\cdot \nabla U_k)\,dx\right| \leq \frac{C}{d_k^{1-\alpha}}\left(1+\overline{R}M\right)^\alpha \int_{B_{\tilde{R}} } |\phi|  U_k^2\,dx \to 0^+.\vspace{0.3cm}
$$

Finally, let us consider the case $R_k/d_k\to +\infty$. Let $\rho>0$ be such that
$\rho> 2d_k$ and $B_\rho(\mathrm{p}^1_k) \subset \O_k$, for $k$ sufficiently large (i.e. $k$ sufficiently large such that $r_k\leq 1/(8\rho)$). Then, by applying a scaling invariant version of the Schauder estimate, we get
\be\label{e:scalingSchauder2}
\begin{aligned}
\frac{1}{(2d_k)^\beta}\sup_{x,y \in B_{d_k}(\mathrm{p}^1_k)}|\nabla U_k(x)-\nabla U_k(y)| &\leq \sup_{x,y \in B_{d_k}(\mathrm{p}^1_k)}\frac{|\nabla U_k(x)-\nabla U_k(y)|}{|x-y|^\beta}\\ &\leq \sup_{x,y \in B_{\rho/2}(\mathrm{p}^1_k)}\frac{|\nabla U_k(x)-\nabla U_k(y)|}{|x-y|^\beta} \leq \frac{C}{\rho^{1+\beta}}\sup_{B_\rho(\mathrm{p}^1_k)} |U_k|.
\end{aligned}
\ee
Since for $k$ sufficiently large we have $R_k\geq 4d_k$, we are allowed to consider $\rho\geq R_k/2$ as admissible radii. On the other hand, in this latter case the interplay between $\mathrm{p}^1_k$ and $\mathrm{p}^2_k$ plays a fundamental role. Indeed, proceeding as in \eqref{e:da.W.a.V}, we get
\be\label{e:gr}
\begin{aligned}
|(\overline{A}_k\nabla V_k)(\mathrm{p}^1_k)| \leq &\,\, |\pi_{\nu_1}((\overline{A}_k\nabla V_k)(\mathrm{p}^1_k))|+
|\pi_{\nu_2}((\overline{A}_k\nabla V_k)(\mathrm{p}^1_k))|\\
\leq &\,\, |\pi_{\nu_1}((\overline{A}_k\nabla V_k)(\mathrm{p}^1_k))|+
|\pi_{\nu_2}((\overline{A}_k\nabla V_k)(\mathrm{p}^1_k)-(\overline{A}_k\nabla V_k)(\mathrm{p}^2_k))|+|\pi_{\nu_2}((\overline{A}_k\nabla V_k)(\mathrm{p}^2_k))|\\
\leq &\,\, 2|R_k\overline{R}|^\alpha + \frac{\Lambda}{L_k r_k^\alpha}|w_k( \hat{x}_k+r_k \mathrm{p}^1_k)-w_k( \hat{x}_k)||\nabla \eta( \hat{x}_k+r_k\mathrm{p}^1_k)|\\
& +\frac{\Lambda}{L_k r_k^\alpha}|w_k( \hat{x}_k+r_k \mathrm{p}^2_k)-w_k( \hat{x}_k)||\nabla \eta( \hat{x}_k+r_k\mathrm{p}^2_k)|\\
\leq &\,\, C|R_k\overline{R}|^\alpha  + \frac{1}{L_k }\left(d_k + R_k\overline{R}\right)^\alpha [w_k]_{C^{0,\alpha}}\norm{\nabla \eta}{L^\infty}\\
\leq &\,\,C(R_k \overline{R})^\alpha,
\end{aligned}
\ee
for $k$ sufficiently large and $C>0$ depending only on $\mathcal{S}_{N_0}$. Finally, by collecting all the previous estimate, we can proceed to control the linear term: for every $x \in B_{\tilde{R}}$
$$
\begin{aligned}
|\overline{A}_k(x) \nabla U_k(x)\cdot \nabla W_k(0)| =&\,\,
|\overline{A}_k(x) \nabla U_k(x)\cdot \nabla V_k(0)|\\
\leq &\,\,
|\overline{A}_k(x) \nabla U_k(x)||\nabla V_k(0)-\nabla V_k(\mathrm{p}^1_k)| + |\overline{A}_k(\mathrm{p}^1_k) \nabla U_k(\mathrm{p}^1_k)\cdot \nabla V_k(\mathrm{p}^1_k)|\\&\,+ |\overline{A}_k(x) \nabla U_k(x) - \overline{A}_k(\mathrm{p}^1_k) \nabla U_k(\mathrm{p}^1_k)||\nabla V_k(\mathrm{p}^1_k)|\\
\leq &\,\,
C |\nabla U_k(x)||\mathrm{p}^1_k|^\alpha + |\nabla U_k(\mathrm{p}^1_k)||(\overline{A}_k\nabla V_k)(\mathrm{p}^1_k)\cdot \nu_1|\\&\,+ |\overline{A}_k(x) \nabla U_k(x) - \overline{A}_k(\mathrm{p}^1_k) \nabla U_k(\mathrm{p}^1_k)||\nabla V_k(\mathrm{p}^1_k)|\\
\leq &\,\, C\left(1+\frac{1}{L_k}\right)d_k^\alpha |\nabla U_k(x)| + C(R_k\overline{R})^\alpha|\nabla U_k(x)-\nabla U_k(\mathrm{p}^1_k)|,
\end{aligned}
$$
where in the second inequality we used that $\nabla V_k$ are globally $\alpha$-H\"{o}lder continuous and in the latter one we apply \eqref{e:t} and \eqref{e:gr}. Then, the gradient estimate \eqref{e:gradient.est} and \eqref{e:scalingSchauder2} with $\rho=R_k/2$, lead to
$$
\begin{aligned}
|\overline{A}_k(x) \nabla U_k(x)\cdot \nabla W_k(0)| &\leq \frac{C}{d_k^{1-\alpha}}|U_k(x)| + C(R_k\overline{R})^\alpha \frac{d_k^\beta}{R_k^{1+\beta}}\sup_{B_{R_k/2}(\mathrm{p}^1_k)} |U_k|\\
&\leq \frac{C}{d_k^{1-\alpha}}|U_k(x)| + C(R_k\overline{R})^\alpha \frac{d_k^\beta}{R_k^{1+\beta}}\frac{R_k^{1+\eps}}{d_k^{1+\eps}}\left(\intn_{B_{d_k}(\mathrm{\mathrm{p}^1_k})}U_k^2\,dx\right)^{1/2},
\end{aligned}$$
for every $x \in B_{\tilde{R}}$ and with $C$ depending only on $\mathcal{S}_{N_0}$. Notice that in the last inequality we used the doubling condition \eqref{tipo.doubling} in $B_{R_k}(\mathrm{p}^1_k)$ and $B_{d_k}(\mathrm{p}^1_k)$ where the Almgren frequency is bounded by $1+\eps$. On the other hand, the Almgren frequency enables a connection between the behavior of solutions in $\mathcal{S}_{N_0}$ across different neighborhoods of their nodal set (e.g., \cite[Remark 2.7]{TerTorVit2}), allowing us to relate the last integral to the one of $U_k^2$ on $B_1$, yielding
\begin{align*}
&\left|\int_{B_{\tilde{R}} } \phi  U_k  (\overline{A}_k\nabla W_k(0)\cdot \nabla U_k)\,dx\right| \\
&\qquad\qquad\qquad\leq \frac{C}{d_k^{1-\alpha}}\int_{B_{\tilde{R}} } |\phi| U_k^2\,dx + C\overline{R}^\alpha \frac{d_k^{\beta-\alpha-\eps}}{R_k^{\beta-\alpha-\eps}}\frac{1}{d_k^{1-\alpha}}\int_{B_{\tilde{R}} }|\phi||U_k|\,dx\to 0^+
\end{align*}
since $\alpha \in (0,1), \eps \in (0,1-\alpha)$ and $\beta\in (0,1)$ is chosen so that $\beta > \alpha +\eps$.
\subsection{The very regular case}\label{s:almost-flat}
In the cases associated to \eqref{case0} and \eqref{case1} we show that the upper bound on the Almgren frequency gives an explicit control of the $C^{1,\alpha}$ seminorm of the parametrization of the regular set, uniformly with respect to $k$. In this specific case, we establish that the points realizing the seminorm in \eqref{e:contr} are concentrated close to the regular part $R(u_k)$. Ultimately, the main result follows by combining such result with the analysis conducted in \cite{TerTorVit1}.\\

We start by controlling the oscillation of the normal vector to the regular set, uniformly-in-$k$.
\begin{Lemma}\label{l:1}
Let $\delta>0,\alpha \in (0,1), \eps \in (0,1-\alpha), r_0>0$. Then, there exists $r_1\in (0,r_0)$, depending only on $\delta, \alpha,\eps$, such that for any $u\in \mathcal{S}_{N_0}, x_0 \in Z(u)\cap B_{3/4}$ satisfying
\be\label{e:condN}
N(x_0,u,r)\leq 1+\eps,\qquad\mbox{for every }r\in (0,r_0),
\ee
we have
$$
\sup_{x,y \in B_r(x_0)}\frac{|\nabla u(x)-\nabla u(y)|}{|x-y|^\alpha} \leq \frac{\delta}{r^\alpha} |\nabla u(x_0)|,\qquad \mbox{for every } r \in (0,r_1).$$
\end{Lemma}

\begin{proof}
By the monotonicity of the Almgren frequency, \eqref{e:condN} implies that $Z(u)\cap B_{3/4}=R(u)\cap B_{3/4}$. Therefore,
suppose by contradiction that there exist $\delta>0,\alpha \in (0,1), \eps\in (0,1-\alpha), r_0>0$ and three sequences $r_k \searrow 0^+, u_k \in \mathcal{S}_{N_0}, x_k \in R(u_k)\cap B_{3/4}$ such that
$$
N(x_k,u_k,r)\leq 1+\eps\quad\mbox{for }r<r_0,\qquad \mbox{but}\quad \left[|\nabla u_k|\right]_{C^{0,\alpha}(B_{r_k}(x_k))}> \frac{\delta}{r_k^\alpha}|\nabla u_k(x_k)|.
$$
Then, consider the normalized sequence centered at $x_k$, that is
$$
\tilde{u}_k(x) = \frac{u_k(x_k+r_k x)}{H(x_k,u_k,r_k)}, \quad \mbox{for }x \in B_{\frac{r_0}{r_k}},
$$
such that $0 \in R(u_k)$ and
$$
N\left(0,\tilde{u}_k,\frac{1}{8r_k}\right)\leq 1+\eps,\quad \mbox{and}\quad \left[|\nabla u_k|\right]_{C^{0,\alpha}(B_1)} > \delta|\nabla \tilde{u}_k(0)|.
$$
Therefore, in view of the compactness of blow-up sequences in $\mathcal{S}_{N_0}$ (e.g., \cite[Proposition 2.5]{TerTorVit2}), there exists a non-zero harmonic polynomial $P$ such that, up to a subsequence, $\tilde{u}_k \to P$ in $C^{1,\alpha}_\loc(\R^n)$ for every $\alpha \in (0,1)$, and
\be\label{e:ass}
\left[|\nabla P|\right]_{C^{0,\alpha}(B_1)}> \delta|\nabla P |(0).
\ee
Moreover, since for every $t>0$ there exists $\overline{k}>0$ such that $r_k<r_0/t$, we also deduce
$$
1\leq N(0,P,t) = \lim_{k \to +\infty} N(0,\tilde{u}_k,t )\leq N\left(0,\tilde{u}_k,\frac{r_0}{r_k}\right)\leq 1+\eps,\quad\mbox{for every }t >0,
$$
which implies by a Liouville theorem, that $P$ is a linear function such that $\norm{P}{L^2(\partial B_1)}=1$. Finally, the linearity of $P$ implies
$$
\left[|\nabla P|\right]_{C^{0,\alpha}(B_1)}=0\qquad \mbox{and}\qquad |\nabla P(0)|=\sqrt{\frac{2}{n\omega_n}},
$$
which contradicts \eqref{e:ass}.
\end{proof}
\begin{Lemma}\label{l:2}
There exists $\delta_0 >0$ universal, such that the following holds true. Let $u \in \mathcal{S}_{N_0}$ be such that
$$
Z(u)\cap B_{\rho}(x_0) = R(u)\cap B_{\rho}(x_0)\qquad\mbox{and}\qquad
[|\nabla u|]_{C^{0,\alpha}(B_{\rho}(x_0))}\leq \frac{\delta_0}{\rho^\alpha} |\nabla u(x_0)|,
$$
for some $\alpha \in (0,1), x_0 \in \R^n$ and $\rho>0$.
Then, up to a rotation, there exists a parametrization $\eta\colon B_{\frac{\rho}{2}}'(x_0')\to\left(-\frac{\rho}{2},\frac{\rho}{2}\right)$ of $R(u)\cap B_{\frac{\rho}{2}}(x_0)$, that is
$$
R(u)\cap B_{\frac{\rho}{2}}(x_0)=
\Big\{(x',x_n) \in B_{\frac{\rho}{2}}(x_0) \colon x_n=\eta(x')\Big\},
$$
with $\eta(x_0')=x_{0,n},|\nabla_{x'}\eta(x_0')|=0$ and such that
$$
[|\nabla_{x'}\eta|]_{C^{0,\alpha}(B'_{\frac{\rho}{2}}(x_0))}\leq C\delta_0,
$$
for some $C>0$ universal.
\end{Lemma}
\begin{proof}
Up to consider the functions $$
v(x)=\frac{1}{\rho}u(x_0+\rho x)\qquad\mbox{and}\quad \phi(x')=\frac{\eta(x_0'+\rho x')-x_{0,n}}{\rho}
$$
we can restrict to the case $x_0=0, \rho=1$. Thus, by contradiction, suppose there exist two sequences $\delta_k\searrow 0^+$, and $u_k \in \mathcal{S}_{N_0}$ such that $Z(u_k)\cap B_{1} =  R(u_k)\cap B_{1},$ and for which
$$
\left[|\nabla u_k|\right]_{C^{0,\alpha}(B_1)}\leq \delta_k |\nabla u_k(0)|,
\qquad\mbox{but}\quad
\frac{1}{\delta_k}\left[|\nabla_{x'} \eta_k|\right]_{C^{0,\alpha}(B'_{1/2})}\geq k.
$$
Clearly,
$\nabla u_k(0)=e_n|\nabla u_k|(0)$ and that the parametrization $\eta_k \in C^{1,\alpha}(B'_{1/2})$ satisfies
$$
R(u_k)\cap B_{\frac12}=\Big\{(x',x_n)\in B_{\frac12} \colon x_n=\eta_k(x')\Big\}, \quad\,\mbox{with }\quad\,\eta_k(0')=|\nabla_{x'}\eta_k|(0')=0.$$
By the differentiability of both $u_k$ and $\eta_k$, it is possible to differentiate the implicit condition
$$
u_k(x',\eta_k(x'))=0\quad\mbox{for }x'\in B'_{1/2},
$$
and deduce that for $i=1,\dots,n-1$ it holds
$$
\partial_i \eta_k(x') = -\frac{\partial_i u_k(x',\eta(x'))}{\partial_n u_k(x',\eta(x'))}\qquad\mbox{on }R(u_k)\cap B_{1/2}.
$$
Notice that, since for every $x \in B_{1/2}$
$$
\partial_n u_k(x) \geq \partial_n u_k(0) - [|\nabla  u_k|]_{C^{0,\alpha}(B_1)}|x|^\alpha\geq |\nabla u_k(0)|\left(1 - \delta_k|x|^\alpha\right)\geq |\nabla u_k(0)|\left(1 - \delta_k \frac{1}{2^\alpha}\right),
$$
by choosing $k$ sufficiently large so that $\delta_k< 2^{\alpha-1}$, we get
$$
\frac12 |\nabla u_k(0)| \leq \partial_n u_k(x) \leq |\nabla u_k|(x)\leq \frac32|\nabla u_k(0)| \qquad\mbox{in }B_{1/2}.
$$
Therefore,
$$
\begin{aligned}
\left[\frac{\partial_i u_k}{\partial_n u_k}\right]_{C^{0,\alpha}(B_{1/2})} &= \sup_{x,y \in B_{1/2}} \frac{|\partial_i u_k(x)\partial_n u_k(y)-\partial_i u_k(y)\partial_n u_k(x)|}{\partial_n u_k(x)\partial_n u_k(y)|x-y|^\alpha}\\
&\leq \frac{6}{|\nabla u_k(0)|}\sup_{x,y \in B_{1/2}} \frac{|\partial_i u_k(x)-\partial_i u_k(y)|+|\partial_n u_k(x)-\partial_n u_k(y)|}{|x-y|^\alpha}\\
&\leq \frac{6}{|\nabla u_k(0)|}\left([\partial_i u_k]_{C^{0,\alpha}(B_{1/2})}+
[\partial_n u_k]_{C^{0,\alpha}(B_{1/2})}\right)\leq 12\delta_k
\end{aligned}$$
and
$$
k\delta_k \leq [|\nabla_{x'}\eta_k|]_{C^{0,\alpha}(B'_{1/2})}\leq \left[\frac{\partial_i u_k}{\partial_n u_k}\right]_{C^{0,\alpha}(B_{1/2})} \leq 12\delta_k,
$$
which leads to a contradiction for $k\geq 13$.
\end{proof}
Before proceeding further, let us recall the conditions (i)-(ii) associated to the cases in \eqref{case0} and \eqref{case1}. By \eqref{case1}, we know there exists $\eps\in (0,1-\alpha), M>0$ and $\overline{k}>0$ such that
          \be\label{e:rzero}
          N(\xi_k,u_k,r)\leq 1+\eps,\quad\mbox{for } r<\mathrm{min}\left\{M,\frac{1}{8}\right\},
          \ee
          for every $k\geq \overline{k}$. Clearly, condition \eqref{case0} is way stronger since it implies that for every small $\eps>0$, there exists $\overline{k}>0$ such that the previous estimate holds true for $k\geq \overline{k}$. In the next statement we simply wrihte $\eps$ and $\alpha$ for identify the two parameters in \eqref{e:rzero}.
\begin{Proposition}\label{p:improve}
There exists a radius $\overline{\rho}\in (0,1/8)$ depending only on $\eps$ and $\alpha$, such that if either \eqref{case0} or \eqref{case1} occurs (see \eqref{e:rzero}), then
$$
R(u_k)\cap B_{\frac{\overline{\rho}}{2}}(\xi_k) =\Big\{(x',x_n)\in B_{\frac{\overline{\rho}}{2}}(\xi_k)
:\ x_n=\eta_k(x')\Big\},$$
with $\eta_k$ a parametrization satisfying $\eta(\xi_k')=\xi_{k,n},|\nabla_{x'}\eta(\xi_k')|=0$ and
\be\label{e:estimate.para}
[|\nabla_{x'}\eta_k|]_{C^{0,\alpha}(B'_{\overline{\rho}/2}(\xi_k'))}\leq C,
\ee
with $C>0$ a universal constant.
\end{Proposition}
\begin{proof}
  First, let $\delta_0>0$ be the quantity arising by Lemma \ref{l:2} and $\alpha\in (0,1), \eps\in (0,1-\alpha)$ be the constants in \eqref{e:rzero}. Then, in view of
  \eqref{e:rzero}, we can apply Lemma \ref{l:1} to $u_k$ with the choices
  $$
  \xi_k\in Z(u_k)\cap B_{3/4}, \qquad\mbox{with }r_0=\mathrm{min}\left\{M,\frac{1}{8}\right\}.
  $$
  Therefore, we deduce the existence of $r_1 \in (0,r_0)$ such that
  $$
Z(u_k)\cap B_{r}(\xi_k) = R(u_k)\cap B_{r}(\xi_k)\qquad\mbox{and}\qquad
[|\nabla u_k|]_{C^{0,\alpha}(B_{r}(\xi_k))}\leq \frac{\delta_0}{r^\alpha} |\nabla u_k(\xi_k)|,
$$
for every $r \in (0,r_1).$ Therefore, if we set $\overline{\rho}:=r_1/2$, by applying Lemma \ref{l:2} on the set $B_{\overline{\rho}}(\xi_k)$, we infer the existence of a parametrization $\eta_k$ of $R(u_k)\cap B_{\overline{\rho}/2}(\xi_k)$ satisfying
$$
[|\nabla_{x'}\eta_k|]_{C^{0,\alpha}\left(B'_{{\overline{\rho}}/2}(\xi_k')\right)}\leq C\delta_0,
$$
for some $C>0$ universal. We stress that the radius $\overline{\rho}$ depends only on $\delta_0, \alpha$ and $\eps$.
  \end{proof}
  Finally, we can conclude by showing that both in \eqref{case0} or \eqref{case1} we can apply the previous result \cite[Theorem 1.2]{TerTorVit1} on the higher order boundary Harnack principle, which contradicts \eqref{e:contr}. More precisely, since by Lemma \ref{l:involuto} both $\delta_k$ and $r_k$ are approaching zero as $k\to +\infty$, it is not restrictive to consider $k\geq \overline{k}$ sufficiently large such that
  $$
 x_k,\zeta_k\in \overline{B_{r_k}(x_k)} \subset B_{\overline{\rho}/7}(\xi_k),
  $$
  where $\overline{\rho}>0$ is the radius arising from Proposition \ref{p:improve}. Then, since $u_k$ and $v_k$ are solutions to \eqref{e:ipotesi}, we have
  $$
  L_{A_k} u_k = L_{A_k} v_k = 0 \qquad \mbox{in }\O_{\eta_k}^+ \cap B_{\overline{\rho}/2}(\xi_k),
  $$
where $\O_\eta^+$ is the connected component of $B_1\setminus Z(u_k)$ to which $x_k$ belongs, which is parameterized by $\eta_k$ in $B_{\overline{\rho}/2}(\xi_k)$. Then, respecting the notations of the statement of \cite[Theorem 1.2]{TerTorVit1}, since $A_k \in \mathcal{A}_k, u_k\in \mathcal S_{N_0}$ and \eqref{e:estimate.para} holds true, for some constant universal, we have
$$
\|A\|_{C^{0,1}(\O_{\eta_k}^+\cap B_{\overline{\rho}/2}(\xi_k))}+\|\eta_k\|_{C^{1,\alpha}(B'_{\overline{\rho}/2}(\xi_k'))}\leq L_1, \qquad \inf_{B_{\overline{\rho}/2}(\xi_k)\cap R(u_k)}|\nabla u_k|\geq L_3
$$
and $L_2=1$, where $L_1$ and $L_3$ are two constant depending only on the class $\lambda,\Lambda,L, \alpha$ and $\eps>0$. Therefore, by applying the higher order Boundary Harnack of \cite{TerTorVit1}, it holds that
$$
L_k = \frac{|\partial_{i_k}(\eta w_k)(x_k)-\partial_{i_k}(\eta w_k)(\zeta_k)|}{|x_k-\zeta_k|^\alpha}\leq \max_{i=1,\dots,n}\left[\partial_i\left(\frac{v_k}{u_k}\right)\right]_{C^{0,\alpha}(\O_{\eta_k}^+\cap B_{\overline{\rho}/4}(\xi_k))}\leq C,
$$
which $C$ independent on $k$. Finally, the contradiction follows by taking $k$ sufficiently large and noticing that $L_k \to +\infty$.

\section{A priori uniform-in-\texorpdfstring{$\mathcal{S}_{N_0}$}{Lg} estimates for solutions with general exponents}\label{sec:equationA}

In this section, we are concerned with the extension of the \emph{a priori} gradient estimates in Theorem \ref{t:grad-version-w} to the case of degenerate PDEs of the form
\begin{equation*}
\mathrm{div}\left(|u|^aA\nabla w\right)=0\quad\mathrm{in \ } B_1,\qquad \mbox{with }u \in \mathcal{S}_{N_0}, a \in \R.
\end{equation*}
Since the analysis of the case of generic exponent follows the demonstrative path of \cite[Theorem 1.1]{TerTorVit2} and Theorem \ref{uniformgradientZ}, in this section we do not retrace all the steps of the two proofs but we focus on the differences due to the exponent $a\in \R$.

\subsection{The proof of Theorem \ref{uniformwa}: the H\"{o}lder estimates of \texorpdfstring{$w$}{Lg}}\label{s:proof-thm15}
The proof deeply resembles the one of \cite[Theorem 1.1]{TerTorVit2}, therefore we proceed by sketching the main ideas with a major attention on the differences due to the exponent $a\neq 2$.\\

Let $A_k\in \mathcal A$, $u_k\in \mathcal S_{N_0}$ and $w_k\in C^{0,\alpha}_{\loc}(B_1)$ be solutions in the sense of Definition \ref{definition.energy.a} satisfying $\norm{w_k}{L^\infty(B_1)}=1$.
Now, consider a radially decreasing cut-off function $\eta\in C^\infty_c(B_1)$ with $0\leq\eta\leq1$, $\eta\equiv1$ in $B_{1/2}$ and such that $\mathrm{supp}\eta=B_{5/8}=:B$. Then, if we show that $(\eta w_k)_k$ is uniformly bounded in $C^{0,\alpha}(\overline{B_1})$, using that $\eta \equiv 1$ in $B_{1/2}$, we infer the same bound for $(w_k)_k$ in $C^{0,\alpha}(\overline{B_{1/2}})$. Notice also, since for every $x_0 \in B_1\setminus B_{5/8}$ we have $(\eta w_k)(x_0)=0$ for every $k>0$, then it is sufficient to ensure the following seminorms
$$
\max_{i=1,\dots,n}\left[\eta w_k\right]_{C^{0,\alpha}(\overline{B_1})} \leq C,
$$
uniformly as $k\to +\infty$. Thus, by contradiction suppose that, up to subsequences,
$$\max_{i=1,\dots,n}\norm{\eta w_k}{L^\infty(\overline{B_1})}\leq C,\qquad\mbox{and}\qquad
\max_{i=1,\dots,n}\left[\eta w_k\right]_{C^{0,\alpha}(\overline{B_1})}\to +\infty,$$
that is, there exist two sequences of points $x_k,\zeta_k\in B$, such that
\be\label{e:Lk-blowing-ua-holder}
L_k = \frac{|(\eta w_k)(x_k)-(\eta w_k)(\zeta_k)|}{|x_k-\zeta_k|^\alpha}\to+\infty.
\ee
Now, given that $|u_k|^a \in L^1_\loc(B_1)$ for $a>-a_{\mathcal{S}}$, then the existence of a sequence $L_k$ blowing up as in \eqref{e:Lk-blowing-ua-holder} implies the existence of $x_k\in B, r_k \searrow 0^+$ and of a sequence
$$
{W}_k(x) =\frac{\eta(x_k)}{L_kr_k^{1+\alpha}}\left(w_k(x_k+r_k A_k^{-1/2}(x_k) x)-w_k(x_k)\right),
$$
of rescaled solution in the sense of Definition \ref{definition.energy.a}. More precisely, given $U_k, \tilde{A}_k$ as in \eqref{e:blow-up-finalproof}, let us set $U$ to be the blow-up limit of $U_k$. Therefore:
\begin{enumerate}
    \item[\rm{(i)}] if $a\geq 1$, then
\begin{equation*}
\int_{\R^n}|U_k|^a \tilde{A}_k \nabla W_k\cdot\nabla\phi\,dx= 0,\quad\mbox{for every }\phi \in C^\infty_c(\R^n\setminus Z(U_k));
\end{equation*}
\item[\rm{(ii)}] if $a \in (-a_{\mathcal{S}},1)$, then on every connected component $\O$ of $\{U_k \neq 0\}$
\begin{equation*}
\int_{\O}|U_k|^a \tilde{A}_k\nabla W_k\cdot\nabla\phi\,dx = 0,\quad\mbox{for every }\phi \in C^\infty_c(\R^n).
\end{equation*}
\end{enumerate}
We recall the validity of the Caccioppoli inequality associated to the equation
$$
\int_{\R^n} |U_k|^a |\nabla W_k|^2\,dx \leq C \int_{\R^n} |U_k|^a |\nabla \phi|^2\,dx,
$$
which holds for every $\phi \in C^\infty_c(\R^n)$ with finite $H^1_\loc(\R^n;|U_k|^a dx)$-norm.
First, it is straightforward to prove that, up to a subsequence, $W_k \to W$ strongly in $C^{0,\alpha}_\loc(\R^n) \cap H^1_\loc(\R^n,|U|^a)$. Indeed, since $\mathcal{L}^n(Z(U)\cap K)=0$ on every compact set $K\subset \R^n$ and $|U|^a \in L^1_\loc(\R^n)$ for every $a>-a_{\mathcal{S}}$, by Fatou's Lemma the limit function $W$ has finite norm with respect to the $H^1_\loc(\R^n,|u|^a)$-norm. Therefore, in view of Proposition \ref{p:H=W} it follows that $W$ is a solution of the limit equation in the sense of Definition \ref{definition.energy.a}. More precisely,
\begin{enumerate}
    \item[\rm{(i)}] if $a\geq 1$, then $W$ is a solution to
$$
\mathrm{div}\left(|U|^a \nabla W\right)=0\quad\mbox{in } \R^n;
$$
\item[\rm{(ii)}] if $a \in (-a_{\mathcal{S}},1)$, then $W$ is a solution to
$$
\mathrm{div}\left(|U|^a \nabla W\right)=0\quad\mbox{in }\R^n, \qquad
|U|^a \nabla W \cdot \nabla U = 0\quad\mbox{on }R(U);
$$
\end{enumerate}
both in the sense of Definition \ref{definition.energy.a}.
Finally, since
$$
|W(x)|\leq C\left(1+|x|\right)^\alpha
\quad\mbox{in }\R^n,
$$
the contradiction follows from the Liouville Theorem \ref{liou_a>-1}.

\subsection{The proof of Theorem \ref{uniformwa}: the H\"{o}lder estimates of \texorpdfstring{$\nabla w$}{Lg}}
The proof deeply resembles the one of Theorem \ref{uniformgradientZ}, therefore we proceed by sketching the main ideas with a major attention on the differences due to the exponent $a \neq 2$.\\

Let $A_k\in \mathcal A$, $u_k\in \mathcal S_{N_0}$ and $w_k\in C^{1,\alpha}_{\loc}(B_1)$ be solutions to \eqref{eqwa} in the sense of Definition \ref{definition.energy.a} satisfying $\norm{w_k}{L^\infty(B_1)}=1$.
Set $\eta\in C^\infty_c(B_1;[0,1])$ to be a radially decreasing cut-off function such that $\eta\equiv1$ in $B_{1/2}$ and  $\mathrm{supp}\eta=B_{5/8}=:B$. Then, if we show that $(\eta w_k)_k$ is uniformly bounded in $C^{1,\alpha}(\overline{B_1})$, using that $\eta \equiv 1$ in $B_{1/2}$, we infer the same bound for $(w_k)_k$ in $C^{1,\alpha}(\overline{B_{1/2}})$. Notice also, since for every $x_0 \in B_1\setminus B_{5/8}$ we have $(\eta w_k)(x_0)=|(\nabla w_k) \eta|(x_0)=0$ for every $k>0$, then it is sufficient to ensure the following seminorms
$$
\max_{i=1,\dots,n}\left[\partial_{i}(\eta w_k)\right]_{C^{0,\alpha}(\overline{B_1})} \leq C,
$$
uniformly as $k\to +\infty$. Thus, by contradiction suppose that, up to subsequences,
$$
\max_{i=1,\dots,n}\left[\partial_{i}(\eta w_k)\right]_{C^{0,\alpha}(\overline{B_1})}\to +\infty,$$
that is, there exist two sequences of points $x_k,\zeta_k\in B$, such that
\be\label{e:Lk-blowing-ua-gradient}
L_k = \frac{|\partial_{i_k}(\eta w_k)(x_k)-\partial_{i_k}(\eta w_k)(\zeta_k)|}{|x_k-\zeta_k|^\alpha}\to+\infty.
\ee
Naturally, up to relabeling we can assume that $i_k = 1$ for every $k>0$ and we can proceed as in Section \ref{s:section-proof-gradient2}. Indeed, given that $|u_k|^a \in L^1_\loc(B_1)$ for $a>-a_{\mathcal{S}}$, then the existence of a sequence $L_k$ blowing up as in \eqref{e:Lk-blowing-ua-gradient} implies the existence of $ \hat{x}_k$ as in Lemma \ref{l:involuto} and of a sequence
$$
\overline{W}_k(x) =\frac{\eta( \hat{x}_k)}{L_kr_k^{1+\alpha}}\left(w_k( \hat{x}_k+r_k x)-w_k( \hat{x}_k) - r_k (\nabla w_k)( \hat{x}_k)\cdot x\right),
$$
of rescaled solution, in the sense of Definition \ref{definition.energy.a}. More precisely
\begin{enumerate}
    \item[\rm{(i)}] if $a\geq 1$, then
\be\label{e:a1}
\int_{\R^n}|U_k|^a \overline{A}_k \nabla \overline{W}_k\cdot\nabla\phi\,dx= a\int_{\R^n} \phi  |U_k|^{a-1}  (\nabla U_k \cdot \overline{A}_k\nabla W_k(0))\,dx\,
\ee
for every $\phi \in C^\infty_c(\R^n\setminus Z(U_k))$;
\item[\rm{(ii)}] if $a \in [0,1)$, then for every connected component $\O$ of $\{U_k \neq 0\}$
\be\label{e:a2}
\int_{\O}|U_k|^a \overline{A}_k\nabla \overline{W}_k\cdot\nabla\phi\,dx = a\int_{\O} \phi  |U_k|^{a-1}  (\nabla U_k \cdot \overline{A}_k\nabla W_k(0))\,dx,
\ee
for every $\phi \in C^\infty_c(\R^n)$.
\end{enumerate}
We recall that $U_k, \overline{A}_k$ are defined by \eqref{e:blow-up-finalproof} and $\nabla W_k(0)$ as in \eqref{e:grad}. By introducing the auxiliary sequence $\overline{V}_k,$ is it possible to prove compactness of the sequence $\overline{W}_k$ and to show that, up to a subsequence, it converges in $C^{1,\beta}_\loc(\R^n)$, for every $\beta \in (0,\alpha)$, and uniformly on every compact set of $\R^n$, to a function $\overline{W}$ whose gradient is non-constant and globally $\alpha$-H\"{o}lder continuous in $\R^n$.\\

Now, let us direct our attention to the equation satisfied by the limit $\overline{W}$. Let $\phi \in C^\infty_c(\R^n)$ be a test function such that $\text{supp}\phi \subset B_{\tilde{R}}$, for some $\tilde{R}>0$. By exploiting the compactness of $\mathcal{A}$, we infer the existence of $C>0$, depending on $\lambda,\Lambda$ and the uniform bound on $[w_k]_{C^{0,\gamma}}$, such that, for $k$ sufficiently large,
$$
\left|\int_{B_{\tilde{R}}}|U_k|^a (\overline{A}_k -  \overline{A})\nabla \phi \cdot \nabla \overline{W}_k  \,dx\right| \leq r_k L \frac{C \norm{\eta}{L^\infty}^2}{L_k r_k^{2+2\alpha-2\gamma}}\int_{B_{\tilde{R}} } |U_k|^a |\nabla \phi|^2 \,dx \to 0^+,
$$
for some limiting matrix $\overline{A}\in \mathcal{A}$.
Similarly, by the compactness of sequences in $\mathcal{S}_{N_0}$ (see, e.g., \cite[Proposition 2.5]{TerTorVit2}), and since, up to a subsequence, we have $U_k \to U$ uniformly on every compact subset of $\mathbb{R}^n$, together with the uniform local boundedness of $\nabla \overline{W}*k$ in $L^\infty*\text{loc}(\mathbb{R}^n)$ with respect to $k > 0$, we deduce
$$
\begin{aligned}
&\left|\int_{B_{\tilde{R}} }(|U_k|^a -|U|^a) \overline{A}\nabla \phi \cdot \nabla \overline{W}_k  \,dx\right|\\
&\qquad\qquad\qquad\leq C\norm{|U_k|^a-|U|^a}{L^\infty(B_{\tilde{R}})} \norm{|\nabla \overline{W}_k|}{L^\infty(B_{\tilde{R}})} \norm{|\nabla \phi|}{L^\infty}\to 0^+\\
\end{aligned}
$$
which implies, by applying Fatou's Lemma in \eqref{e:cacio-e-pepe}, that $\overline{W} \in H^1_\loc(\R^n,|u|^a)$ (see also Proposition \ref{p:H=W}). As in the proof of Theorem \ref{uniformgradientZ} (see Section \ref{s:section-proof-gradient2}), once we show that the right hand side in the weak formulations \eqref{e:a1} and \eqref{e:a2} vanish asymptotically as $k$ approaches infinity, we deduce that the blow-up limit $\overline{W}$ is a solution of the limit degenerate equation associated to the weight $U$. More precisely, up to absorb the constant valued matrix $\overline{A}$ with a composition of $\overline{W}$ with a linear transformation, we have the following possibilities:
\begin{enumerate}
    \item[\rm{(ii)}] if $a\geq 1$, then $\overline{W}$ satisfies $\mathrm{div}\left(|U|^a \nabla \overline{W}\right)=0$ in $\R^n$;
\item[\rm{(ii)}] if $a \in [0,1)$, then $\overline{W}$ satisfies
$$
\mathrm{div}\left(|U|^a \nabla \overline{W}\right)=0\quad\mbox{in }\R^n, \qquad
|U|^a \nabla \overline{W} \cdot \nabla U = 0\quad\mbox{on }R(U),
$$
\end{enumerate}
where both the equations must be understood in the sense of Definition \ref{definition.energy.a}. In this scenario, the absurd follows by applying a Liouville type theorem: indeed, since by construction
$$
\overline{W}(0) = |\nabla \overline{W}|(0)=0,
$$
we deduce, in light of global $\alpha$-H\"{o}lder regularity of the gradient of $\overline{W}$, that
$$|\overline{W}(x)|\leq C\left(1+|x|\right)^{1+\alpha}\quad\mbox{in }\R^n.$$
Finally, the Liouville Theorem \ref{liou_a>-1} implies that  $\overline{W}$ must be a linear function, in contradiction with the fact that $\overline{W}$ has non constant gradient.
As for Section \ref{s:section-proof-gradient2}, the conclusion of the proof follows by exploiting the hooking Lemma of Section \ref{s:hooking} and by applying the fine analysis of Section \ref{s:inter} and  \ref{s:almost-flat}: indeed, despite the difference in the exponent of the weight, as long as $a \geq 0$ all the computation can be carried out exactly as for $a=2$.

\begin{remark}
By reviewing the passages of the proof, the reader can easily convince himself of the validity of the Theorem \ref{uniformwa} also on single nodal components of the weight function $u$ \emph{provided all possible blow-ups end up being defined on unbounded domains with a connected boundary}, in view of the application of Lemma \ref{c:conformal-polynomial} to obtain the proper Liouville type property on nodal domains as a variant of  Theorem \ref{liou_a>-1} (see Remark \ref{rem:conn_bdry}). This remark will turn useful in the proof of the a posteriori Schauder estimates, notably in proving Theorem \ref{t:effective}.
\end{remark}

\section{A posteriori higher regularity across singular sets in two dimensions}\label{sec:fixed}
The ultimate aim of this section, is the proof of Theorem \ref{effective2}; that is, uniform-in-$\mathcal S_{N_0}$ local $C^{1,1-}$-regularity of solutions of \eqref{eqwa} in two dimensions even across singular points. In order to prove the result, we develop a regularization-approximation scheme: first, in Section \ref{sec:hodo} we prove regularity for solutions around a given singular point, when the variable coefficients have constant determinant. Then, in Section \ref{s:appr}, given a solution to \eqref{eqwa} and around a given singular point of $u$, we construct a sequence of functions which solve equations with constant coefficients and do converge to $w$, and we prove that the regularity estimates are uniform along the sequence, finally leading to Theorem \ref{t:effective}. Ultimately, in Section \ref{s:fin-effective} we conclude the proof of Theorem \ref{effective2}.

\subsection{Quasiconformal hodograph transformation and regularity when the determinant is constant}\label{sec:hodo}

Along this section, let us define the following class of coefficients
$$
\begin{aligned}
\mathcal A_E &=\left\{
A=(a_{ij})_{ij} \, : \, A=A^T, \, A(0)=\mathbb I,\, \mathrm{and \ satisfies} \,\eqref{UE}\right\},
\end{aligned}
$$
namely the coefficients are assumed to be only measurable and uniformly bounded away from zero and infinity.

Now, we start by studying the behaviour of solutions to \eqref{eqwa} near an isolated singular point of $u$. Indeed, this is not restrictive since in the two dimensional case the singular set $S(u)$ is locally a collection of isolated points (actually the critical set is too). Then, aim of this first section, is the proof of the following result.

\begin{Theorem}\label{teodim21}
Let $n=2, A\in \mathcal{A}_E$ and
and $u$ be a $A$-harmonic function in $B_1$ such that $S(u)\cap B_1 = \{0\}$ and $\mathcal{V}(0,u)=N \geq 2$. Let $a>-2/N$ and $w$ be a continuous solution to \eqref{eqwa} in the sense of Definition \ref{definition.energy.a} and assume that the determinant of $A$ is constant and
\be\label{e:requirement.A}
a_{ij}\in C^{k-1,\alpha}(B_1),
\ee
for some $k\in\mathbb N\setminus\{0\}, \alpha\in(0,1)$. Then $w \in C^{k,\alpha}_\loc(B_1)$ and it satisfies
\be\label{e:condizioni-bordo}
A\nabla w\cdot\nabla u=0 \quad\mathrm{on \ } R(u)\cap B_1,\qquad \left|\nabla w\right|=0 \quad\mathrm{on \ } S(u)\cap B_1.
\ee
\end{Theorem}

As previously noted in \cite{HartWint}, the techniques of complex analysis and quasiconformal mappings offer powerful tools, particularly in two dimensions, for solving PDEs in divergence form (see \cite[Chapter 16]{quasibook} for a comprehensive discussion of their potential applications). In this section, we introduce a quasiconformal hodograph map for each connected component adjacent to a singular point, within the framework of $C^{k-1,\alpha}$ elliptic coefficients with constant determinant. This analysis extends the one conducted in \cite[Section 2]{viennesi} and \cite[Section 3.4.1]{TerTorVit1} (see also \cite[Corollary 1]{HartWint}). It is worth mentioning that a similar construction is employed in \cite[Section 4.2]{DePSV} within the context of two-phase free boundary problems.\\

Before stating the main lemma, we start by introducing the notion of $A$-harmonic conjugate, which is fundamental in order to construct a quasiconformal hodograph transformation. First, let $k\in \N\setminus \{0\}, A \in \mathcal{A}_E$ satisfying \eqref{e:requirement.A} and $u \in H^1(B_1)\cap C^{k,\alpha}_\loc(B_1)$ be a solution to \eqref{equv}, that is
\be\label{e:u-qc}
L_A u = \mathrm{div}(A\nabla u)=0 \qquad \mbox{and}\qquad \mathrm{div}(J\nabla u)=0 \qquad \mbox{in}\quad B_1,
\ee
where
$$
J := \begin{pmatrix}
0 & 1 \\
-1 & 0
\end{pmatrix}, \qquad \mbox{with}\quad J^T=J^{-1}=-J.
$$
Notice that both the equations in \eqref{e:u-qc} hold in a weak sense. It is worth noting that the second equation represents a weak formulation of the symmetry of the second derivatives.

Then, by the Poincar\'e lemma for closed forms on contractible domains, there exists a function $\overline{u} \in H^1(B_1)\cap C^{k,\alpha}_\loc(B_1)$ such that
\be\label{e:A-ham-conjugate}
\overline{u}(0)=0\qquad\mathrm{and}\qquad \nabla \overline{u} = JA \nabla u;\quad\mbox{that is, }
\begin{cases}
\partial_x \overline{u} = -a_{12}\partial_x u -a_{22}\partial_y u \\
\partial_y \overline{u} = a_{11}\partial_x u + a_{12}\partial_y u
\end{cases}.
\ee
In the literature, such a function is usually referred to as the $A$-harmonic conjugate of $u$, indeed, by rephrasing \eqref{e:u-qc} in terms of $\overline{u}$ we get that $\mathrm{div}(J\nabla \overline{u})=0$ in $B_1$ and
$$
L_{\overline{A}} \overline{u} = 0 \quad\mbox{in}\quad \O,\qquad\mbox{with}\qquad \overline{A}  := \frac{A}{\mathrm{det}A} \in \mathcal{A}_E\cap C^{k-1,\alpha}(B_1).
$$
Assuming that the determinant of $A\in \mathcal A_E$ is constant, one has that that $\overline{A}=A$.
Finally, let us consider the $C^{k,\alpha}$-hodograph type map
\begin{equation}\label{HodoTheta}
\Theta(x,y) := \left(\overline{u}(x,y), u(x,y)\right).
\end{equation}
Notice that, by direct computations, we have
\be\label{e:gradient-theta}
|\mathrm{det} D\Theta| = (A\nabla u \cdot \nabla u) = (\overline{A}\nabla \overline{u} \cdot \nabla \overline{u})\quad \mbox{in }B_1.
\ee
Then, defining
\be\label{e:matrix-theta}
B(x',y'):=\left(\frac{1}{|\mathrm{det} D\Theta|} D\Theta A (D\Theta)^T\right)(\Theta^{-1}(x',y')),
\ee
by direct computations, one has that
$$
B(x',y')=
\begin{pmatrix}
(\mathrm{det}A)(\Theta^{-1}(x',y')) & 0 \\
0 & 1
\end{pmatrix}.
$$
Assuming that $\mathrm{det}A$ is constant, we deduce that $B$ is the identity matrix.\\

In \cite[Section 4.2]{DePSV}, maps of this kind are referred to as \emph{quasiconformal hodograph transformations} since they straighten the nodal domains of $u$ through a mapping with bounded distortion (for a discussion on quasiconformal mappings, see \cite[Section 2]{quasibook}).\\
Nevertheless, it is worth noting that in our context the presence of singular points (where the map $\Theta$ has null partial derivatives) requires a more delicate analysis. Indeed, although the map $\Theta$ is a priori only quasiregular in $B_1$ (see \cite[Section 2]{quasibook} for a precise definition), we prove that given a connected component $\O\cap B_1$ of $\{u\neq 0\}\cap B_1$ adjacent to a singular point, it is possible to show that $\Theta \colon \overline{\O}\cap B_1 \to \{y\geq 0\}$ is actually quasiconformal.

\begin{Lemma}\label{l:quasiconformal-solution}
Let $n=2, k \in \N\setminus \{0\}, A\in \mathcal{A}_E$ satisfying \eqref{e:requirement.A} and $u$ be a $A$-harmonic function in $B_1$ such that
$$
 S(u) \cap B_1 = \{0\}\quad\mbox{and}\quad\mathcal V(0,u)=N\geq2.
$$
Let $\Omega$ be a nodal component of $u$ in $B_1$, $a>-2/N$ and let $w\in H^1(B_1,|u|^a)$ satisfy
\be\label{e:weak-vecchia-qc}
\int_{\O} |u|^a A \nabla w \cdot \nabla \phi \,dx dy= 0,\qquad \mbox{for every }\phi \in C^\infty_c(B_1).
\ee
Then $w(x,y) = \overline{w}\left(\overline{u}(x,y),u(x,y)\right)$ in $\overline{\O}\cap B_1$,
where:
\begin{enumerate}
    \item[\rm{(i)}] $\overline{u}$ is the $A$-harmonic conjugate of $u$ in $B_1$;
    \item[\rm{(ii)}] $\overline{w} \in H^1(B_1^+, |y|^a)$ satisfies\be\label{e:weak-nuova-qc}
\int_{\{y>0\}} |y|^a B\nabla \overline{w} \cdot \nabla \varphi \, dx dy= 0
\ee
for every $\varphi \in C^\infty_c(B_1)$ with the matrix $B$ defined in \eqref{e:matrix-theta}.
\end{enumerate}
\end{Lemma}
\begin{proof}
Let $\overline{u}$ be the $A$-harmonic conjugate of $u$ and consider the map $\Theta$ defined as in \eqref{HodoTheta}. Clearly, being both $u$ and $\overline{u}$ solutions to equations of the form \eqref{equv} with $A\in\mathcal{A}_E$, we already know that  $\Theta \in C^{k,\alpha}_\loc(B_1;\R^2)$ is bounded in $B_1$ with some constant depending only on $\lambda,\Lambda$ and $L$, and such that
$$
\Theta\colon \O \cap B_1 \to \{y>0\}\qquad\mbox{and}\qquad
\Theta\colon \partial\O \cap B_1 \to \{y=0\}.
$$
In particular, the set $\Theta(\overline{\O}\cap B_1)$ is an open neighborhood of the origin in the upper half-space $\{y\geq 0\}$. Notice also that by setting $\Theta^{-1}(0)=0$, the inverse function $\Theta^{-1}$ is well-defined in a neighborhood of the origin in the upper half-space $\{y\geq 0\}$ and therefore $\Theta$ is quasiconformal on $\overline{\Omega}\cap B_1$. For the sake of simplicity in notation, it is not restrictive to suppose that $\Theta(\overline{\O}\cap B_1) = B_1^+$. Thus, let $\phi \in H^1(\Omega\cap B_1,|u|^a)$ be fixed. Then, if we set $\varphi := \phi \circ \Theta^{-1}$, by applying the change of coordinates $(x',y')=\Theta(x,y)$, we obtain
$$
\begin{aligned}
\norm{ \varphi}{H^1(B_1^+,|y|^a )}^2 &=
\int_{B_1^+} |y'|^a \varphi^2 \,dx'dy' + \int_{B_1^+} |y'|^a |\nabla \varphi|^2 \,dx'dy'\\
&\leq C\left(
\int_{B_1^+} \frac{|y'|^a}{|\mathrm{det} D\Theta|} \varphi^2 +  |y'|^a B \nabla \varphi\cdot \nabla \varphi \,dx'dy'\right)\\
&\leq  C\left(
\int_{\O\cap B_1} |u|^a\phi^2 + |u|^a A \nabla \phi\cdot \nabla \phi  \,dxdy\right)\\
&\leq  C\norm{\phi}{H^1(\Omega\cap B_1,|u|^a)}^2,
\end{aligned}
$$
where in the first inequality the existence of a constant $C>0$ follows by the boundedness of the gradient of $\overline{u}$ (see \eqref{e:gradient-theta}).
Therefore, in light of Proposition \ref{p:H=W}, the pull-back $\varphi$ of a function $\phi\in H^1(\O\cap B_1,|u|^a)$ has finite norm with respect to the new weight $|y|^a$, thus belonging to the space $H^1(B_1^+,|y|^a)$.
Thus, by Proposition \ref{p:H=W}, we conclude that $\overline{w}:=w \circ \Theta^{-1} \in H^1(B_1^+,|y|^a)$.
On the other hand, by directly using the mapping $\Theta$, it is possible to deduce the validity of the weak formulation \eqref{e:weak-vecchia-qc}: indeed given $\varphi \in H^1(B_1^+,|y|^a)$ (alternatively the same argument can be applied to the classes of test functions of \cite[Proposition 3.4]{TerTorVit2}, being dense in the Sobolev space), if we set $\phi := \varphi \circ \Theta \in H^1(\O\cap B_1,|u|^a)$, the condition
$$
\int_{B_1^+} |y'|^a \left(\frac{1}{|\mathrm{det} D\Theta|} D\Theta A (D\Theta)^T\right)(\Theta^{-1}(x',y')) \nabla \overline{w}(x',y') \cdot \nabla \varphi(x',y')\,dx' dy' = 0
$$
follows immediately by \eqref{e:weak-vecchia-qc}. Finally, in accordance with \eqref{e:matrix-theta}, it completes the proof.
\end{proof}

\begin{proof}[Proof of Theorem \ref{teodim21}]
It is not restrictive to suppose that, up to further dilations, $0\in S(u)$ is the only critical point of $u$ in $B_1$ and that $B_1\setminus Z(u)$ is composed by $2N$ connected components (nodal domains), where $N=\mathcal{V}(0,u)\geq2$. Then, up to a rescaling and translation, by applying Lemma \ref{l:quasiconformal-solution} in every connected component $\O_i$ of $B_1\setminus Z(u)$, with $i=1,\dots,2N$, we already know that
\be\label{e:boundary-singular}
w(x,y) = \overline{w}_i\left(\overline{u}(x,y), u(x,y)\right) \quad \mbox{in }\overline{\O}_i\cap B_1,
\ee
where $\overline{w}_i$ satisfies \eqref{e:weak-nuova-qc}. Therefore, by classical Schauder regularity we know that $u,\overline{u} \in C^{k,\alpha}_\loc(B_1)$. Then, by applying \cite[Theorem 1.1]{SirTerVit1}, the solutions $\overline{w}_i$ belong to $C^{\infty}(\overline{B_{1/2}^+})$. By composition, we deduce the validity of the claimed regularity on every connected component $\O_i\cap B_{1/2}$, i.e. $w\in C^{k,\alpha}(\overline{\O}_i\cap B_{1/2})$. Finally, being $w$ continuous in $B_1$, the regularity can be extended to the whole ball $B_{1/2}$ by applying the gluing lemma \cite[Lemma 2.12]{TerTorVit1}.\\

Ultimately, the boundary conditions \eqref{e:condizioni-bordo} follow directly by \cite[Theorem 1.1]{TerTorVit1} and \eqref{e:boundary-singular}. Indeed, in light of \eqref{e:matrix-theta} we have
$$
(A\nabla w\cdot \nabla w)(x,y) = (A\nabla u\cdot \nabla u)(x,y) (B\nabla \overline{w}_i\cdot \nabla \overline{w}_i)(\overline{u}(x,y),u(x,y)),\quad \mbox{for }(x,y)\in \overline{\O}_i\cap B_1,
$$
which implies that $w$ has vanishing gradient at the origin.
\end{proof}

\begin{Corollary}
    Let $n=2$, $k\in\mathbb N\setminus\{0\}$, $\alpha\in (0,1)$, $A\in\mathcal A_E$ with $A\in C^{k-1,\alpha}(B_1)$ and constant determinant. Let $u$ be $A$-harmonic in $B_1$ with $\mathcal V(x,u)\leq N$ for $x \in B_1$. Let $a>-\min\{1,2/N\}$ and $w$ be a continuous solution to \eqref{eqwa} in $B_1$. Then $w\in C^{k,\alpha}_\loc(B_1)$ with
    $$
A\nabla w\cdot\nabla u=0 \quad\mathrm{on \ } R(u)\cap B_1,\qquad \left|\nabla w\right|=0 \quad\mathrm{on \ } S(u)\cap B_1.
$$
Moreover, there exists $C>0$ depending only on $\lambda,\Lambda,k,\alpha,a,\|A\|_{C^{k-1,\alpha}(B_1)},u$ and $Z(u)$ such that
$$
\|w\|_{C^{k,\alpha}(B_{1/2})}\leq C\|w\|_{L^\infty(B_1)}.
$$
\end{Corollary}
\begin{proof}
    Any point $x$ in the ball $B_{1}$ either
$$
\mathrm{(i)}\,\,\, x \in B_1\setminus Z(u),\qquad
\mathrm{(ii)}\,\,\, x \in B_1\cap R(u),\qquad
\mathrm{(iii)}\,\,\, x \in B_1\cap S(u).
$$
Then, there exists a small radius $r_x>0$ such that either $B_{r_x}(x)\cap Z(u)=\emptyset$ in case (i), $B_{r_x}(x)\cap S(u)=\emptyset$ in case (ii), or $B_{r_x}(x)\cap S(u)=\{x\}$ in case (iii). We claim that in any case $w\in C^{k,\alpha}_\loc(B_{r_x}(x))$. The latter regularity relies in classical Schauder theory for second order uniformly elliptic equations in case (i), \cite[Theorem 1.1 and Lemma 2.12]{TerTorVit1} in case (ii) (being $a>-1$ in particular), and Theorem \ref{teodim21} in case (iii) (being $\mathrm{det}A$ constant). Hence, by compactness one can extract a finite covering of $\overline{B_{r}}$ (for any $0<r<1$) with the same property, obtaining the desired regularity together with an estimate depending on $u,Z(u)$.
\end{proof}
We remark that the result above holds true having the determinant of the coefficiets constant locally around the singularities of $u$. However, this requirement is a bit recursive since the behaviour of the $A$-harmonic function $u$ - and the structure of its nodal set - does depend on the values of the coefficients $A$.

\subsection{The regularization-approximation scheme around singular points}\label{s:appr}
The aim of this section, is the construction of the regularization-approximation scheme which will ultimately lead to establishing the following result.

\begin{Theorem}\label{t:effective}
Let $n=2, a\geq 0, A\in \mathcal{A}$ and $ u\in H^1(B_1)$ be a solution to \eqref{equv} such that
$$
S(u)\cap B_1 = \{0\},\quad \mbox{with}\quad\mathcal{V}(0,u)=N\geq2.
$$
Given $w$ a continuous solution to \eqref{eqwa} in the sense of Definition \ref{definition.energy.a}, then $w \in C^{1,1-}_\loc(B_1)$ and
$$
A\nabla w\cdot\nabla u=0 \quad\mathrm{on \ } R(u)\cap B_1,\qquad \left|\nabla w\right|=0 \quad\mathrm{on \ } S(u)\cap B_1.
$$
\end{Theorem}

Let $n=2, A\in \mathcal{A}$ and $u \in H^1(B_1)$ be a solution to \eqref{equv} satisfying
$$
S(u)\cap B_1= \{0\}
$$
with $\mathcal{V}(0,u)=N$. We proceed by constructing the approximating sequence for $u$. As we will see, the sequence is built in a way to have a prescribed vanishing at the singularity $0$, to solve equations with locally constant coefficients around $0$, and to have a particular convergence to the original $u$ (convergence of the sequence together with its blow-ups).

To begin with, given $\eps>0$ sufficiently small, we consider a cut-off function $\eta_\eps\in C^\infty_c(B_{2\eps},[0,1])$ such that
$$
\eta_\eps \equiv 1\quad\mbox{in }B_\eps,\quad \eta_\eps \equiv 0\quad\mbox{in }B_1\setminus B_{2\eps},\quad\mbox{and}\quad |\nabla \eta_\eps|\leq \frac{2}{\eps}\quad\mbox{in }B_{2\eps}\setminus B_\eps.
$$
Therefore, the approximating sequence of variable coefficients
\be\label{e:approximating-A}
A_\eps:= A + (\mathbb{Id}-A)\eta_\eps, \qquad
\mbox{is such that}\qquad A_\eps \equiv \mathbb{Id}\quad\mbox{in }B_\eps
\ee
and
\be\label{e:matrix-esti}
\begin{aligned}
\left[A_\eps\right]_{C^{0,1}(B_1)} &\leq C\left( \left[A\right]_{C^{0,1}(B_1)} +  \norm{\mathbb{Id}-A}{L^\infty(B_{2\eps})}\norm{\nabla \eta_\eps}{L^\infty(B_{2\eps})} +  \left[\mathbb{Id}-A\right]_{C^{0,1}(B_{2\eps})}\right) \leq C L,
\end{aligned}
\ee
for some $C>0$ universal (recall that $\left[A\right]_{C^{0,1}(B_1)}\leq L$). We stress that in the second inequality we exploited the Lipschitz continuity of the coefficients $A$ and the explicit gradient estimate of the cut-off function $\eta_\eps$. Finally, given $R_0>0,\eps_0>0$ to be chosen later, for $\eps \in (0,\eps_0)$ we aim to construct a solution $\varphi_\eps \in H^1(B_{R_0})\cap C^{1,1-}(B_{R_0})$ to the equation
\be\label{e:approx.varphi}
-\mathrm{div}(A_\eps \nabla \varphi_\eps ) = \mathrm{div}F_\eps \quad\mbox{in }B_{R_0},\quad\mbox{with}\quad F_\eps(z):=\eta_\eps(z) \left(\mathbb{Id}-A(z)\right)\nabla u(z),
\ee
and such that $\mathcal{V}(0,\varphi_\eps)=N$ and $\varphi_\eps\to0$ in $C^{1,\alpha}_\loc(B_{R_0})$, for any $\alpha\in(0,1)$. Indeed, by choosing as approximating pair
\be\label{e:coppia}
u_\eps :=u + \varphi_\eps,\qquad A_\eps := A + (\mathbb{I}-A)\eta_\eps,
\ee
we deduce the existence of a sequence of functions satisfying
\be\label{e:numerarla}
\mathrm{div}(A_\eps \nabla u_\eps) =0 \quad \mbox{in }B_{R_0},\qquad
\mathcal{V}(0,u_\eps)=N,\qquad \mbox{for }\eps \in (0,\eps_0),
\ee
and converging to $u$ uniformly in $C^{1,\alpha}_\loc(B_{R_0})$, for every $\alpha \in (0,1)$, and strongly in $H^1(B_{R_0})$. Ultimately, in order to prove Theorem \ref{t:effective} we need a stronger convergence of the approximating sequence $u_\varepsilon$ to $u$: more precisely, we need to show that the blow-up limit of $u_\varepsilon$ at $0$ converges in $C^{0,\alpha}_\loc$ to the blow-up limit of $u$ at $0$, for every $\alpha \in (0,1)$ (see Lemma \ref{lem:xi_eps}).\\

In order to achieve such a result, we start by introducing the following lemma that allows to construct $A_\eps$-harmonic functions with prescribed blow-up limit at singular points.
\begin{Lemma}\label{l:killing-poly}
Given $k\geq 2$, let $P_k$ be a $k$-homogeneous harmonic polynomial. Then, there exists $R_0>0, \eps_0>0$ such that the following holds true: for every $R \in (0,R_0),\eps\in (0,\eps_0)$ there exists $a_\eps^{k} \in \R$ such that the unique solution $\psi_\eps^{k}\in H^1(B_R)$ to
\be\label{e:homo}
\begin{cases}
\mathrm{div}(A_\eps \nabla \psi_\eps^{k}) =0 & \mbox{in }B_R,\\
\psi_\eps^{k} = a_\eps^{k} + P_k & \mbox{on }\partial B_R,
\end{cases}
\qquad \mbox{satisfies}\quad\mathcal{V}(0,\psi_\eps^{k})\geq k,
\ee
where $A_\eps$ is defined by \eqref{e:approximating-A}. Moreover, for every $r<R$ there exists $\Gamma_\eps$ such that
\be\label{e:taylor-exp}
        \psi_\eps^{k}(z) = P_k(z) + \Gamma_{\eps}^{k}(z),\qquad \mbox{in }B_r
\ee
with
        \be\label{e:taylor-exp-remainder}
        |\Gamma_{\eps}^{k}(z)| \leq C|z|^{k+\delta},\quad|\nabla\Gamma_\eps^{k}(z)| \leq C |z|^{k-1+\delta}
        \qquad\quad\mbox{in }B_r,
        \ee
        and with $C,\delta>0$ depending only on $\lambda, \Lambda,L, \eps_0, R_0$.
\end{Lemma}
\begin{proof}
For simplicity of notation, through the proof we omit the dependence on the degree $k\geq2$. Then, $P$ is a $k$-homogeneous harmonic polynomial. We split the proof in two steps.\\

Step 1: existence of solution. We start by constructing $\psi_\eps$ and $a_\eps$ as above. Let $\tilde{\varphi}_\eps$ be the unique solution to
$$
\mathrm{div}(A_\eps \nabla \tilde{\varphi}_\eps) = 0 \quad\mbox{in }B_R,\qquad
\tilde{\varphi}_\eps = P  \quad\mbox{on }\partial B_R,
$$
and set $\psi_\eps(z) := \tilde{\varphi}_\eps(z) -  \tilde{\varphi}_\eps(0)$ and $a_\eps:= -\tilde{\varphi}_\eps(0)$. Then, $0 \in Z(\psi_\eps)$, and
$$
|a_\eps| \leq C R^k,
$$
where $C$ depends only on $\lambda, \Lambda$ and $L$. More precisely, by scaling invariant boundary regularity, there exists $C>0$ depending only on $\lambda, \Lambda, \alpha$ and $L$, such that
$$
\norm{\tilde{\varphi}_\eps}{L^\infty(B_R)} +
R\norm{|\nabla \tilde{\varphi}_\eps|}{L^\infty(B_R)} + R^{1+\alpha}[|\nabla \tilde{\varphi}_\eps|]_{C^{0,\alpha}(B_R)}\leq C R^k,
$$
for every $\alpha \in (0,1)$.
Naturally, this estimate implies
\be\label{schau.bou}
\norm{\psi_\eps}{L^\infty(B_R)} +
R\norm{|\nabla \psi_\eps|}{L^\infty(B_R)} + R^{1+\alpha}[|\nabla \psi_\eps|]_{C^{0,\alpha}(B_R)}\leq C R^k,
\ee
for some constant  $C>0$ depending only on $\lambda, \Lambda, \alpha$ and $L$.\\

Step 2: the validity of the condition on the vanishing order. The proof is concluded once we show that the function $\psi_\eps$, previously constructed, satisfies $\mathcal{V}(0,\psi_\eps)\geq k$ for $\eps$ and $R$ sufficiently small. Thus, we proceed by a contradiction argument: suppose that there exist two sequences $\eps_i, R_i$ such that
$$
\mathrm{div}(A_{\eps_i} \nabla \psi_{\eps_i}) = 0 \quad\mbox{in }B_{R_i},\qquad
\psi_{\eps_i} = a_{\eps_i} + P  \quad \mbox{on }\partial B_{R_i},
$$
and
\begin{equation}\label{e:absurdvanish}
    \mathcal{V}(0,\psi_{\eps_i})\leq k-1.
\end{equation}
Thus, if we set
$$
\psi_i(z) := \frac{1}{R_i^k}\psi_{\eps_i}(R_i z),\qquad A_i(z) := A_{\eps_i}(R_i z) = A(R_i z) + (\mathbb{Id}-A(R_i z))\eta_{\eps_i}(R_i z) ,
$$
we get that
$$
\mathrm{div}(A_{i} \nabla \psi_{i}) = 0\quad \mbox{in }B_1,\qquad
\psi_{i} = \frac{a_{i}}{R_i^k} + P \quad\mbox{on }\partial B_1.
$$
In light of \eqref{e:matrix-esti}, we already know that $A_i$ is equi-Lipschitz continuous in $B_1$ and it converges to the identity matrix uniformly on $B_1$. On the other hand, by rescaling \eqref{schau.bou}, we get that for every $\alpha \in (0,1),$ the sequence $\psi_i$ is uniformly bounded in $C^{1,\alpha}(\overline{B_1})$ and satisfies $\psi_i(0)=0$. Hence, it converges, up to a subsequences, to some $\psi_\infty \in C^{1,\alpha}(\overline{B_1})\cap H^1(B_1)$ solution to
$$
\Delta \psi_\infty = 0 \quad\mbox{in }B_1,\qquad \psi_\infty= a_\infty + P\quad\mbox{on }\partial B_1,
$$
where $a_\infty$ is the limit of the ratio $a_i/R_i^k$ along the converging subsequence. Clearly, by combining the uniqueness of solution for the Poisson problem with the fact that $\psi_\infty(0)=0$, we deduce that $a_\infty=0$ and $\psi_\infty \equiv P$. Finally, by uniform convergence, for every $\delta>0$ there exists $i_0>0$ such that
$$
|\psi_{\eps_i}(z)-P(z)|\leq \delta R_i^{k}\quad  \mbox{for }z\in B_{\frac{R_i}{2}}
$$
and so
$$
\left|\frac{2^{2k}}{R_i^{2k}}\intn_{B_{R_i/2}}\psi_{\eps_i}^2\,dz - \intn_{B_1}P^2\,dz \right|\leq   C\delta^2,
$$
for every $i\geq i_0$, where $C>0$ is universal. Hence, by choosing $\delta>0$ sufficiently small, we infer that $\mathcal{V}(0,\psi_{\eps_i})\geq  k$, in contradiction with the \eqref{e:absurdvanish}.\\
In light of the uniform convergence of the previous sequence, it is not restrictive to assume that $R_0,\eps_0$ are chosen in such way that
\be\label{e:cosa-al-limite}
|\psi_{\eps}(z)-P(z)|+ R|\nabla \psi_{\eps}(z)-\nabla P(z)|\leq \delta R^{k}\quad  \mbox{for }z \in B_{\frac{R}{2}}
\ee
for every $\eps<\eps_0$ and $R<R_0$. The proof is concluded once we notice that the validity of the Taylor expansion \eqref{e:taylor-exp}-\eqref{e:taylor-exp-remainder} is granted by \cite{Aless1} (see also \cite[Subsection 2.2]{TerTorVit2}). On the other hand, the fact that the leading term coincides with $P$ follows by \eqref{e:cosa-al-limite}.
\end{proof}
By iterating the previous lemma, we can construct solutions to \eqref{e:approx.varphi} with prescribed vanishing order at the origin.
\begin{Lemma}\label{l:sus}
For every $N\geq 2$, there exist $R_0>0, \eps_0>0$ such that the following holds true. For every $R \in (0,R_0),\eps\in (0,\eps_0)$ there exists a solution $\varphi_\eps\in H^1(B_R)$ to \eqref{e:approx.varphi} such that for every $\alpha \in (0,1)$
\be\label{e:sei-tu}
\norm{\varphi_\eps}{C^{1,\alpha}(B_{R})}\leq C \eps^{1-\alpha}\quad\mbox{and}\quad\mathcal{V}(0,\varphi_\eps)=N,
\ee
for some constant $C>0$ depending only on $\lambda,\Lambda, L, \eps_0, \alpha$ and $R_0$.
\end{Lemma}
\begin{proof}
Let $R_0$ and $\eps_0$ be two fixed positive constants, which will be specified later. We proceed by  constructing the claimed function with an iterative argument.\\

Let $R<R_0, \eps < \eps_0$ and $\phi^0_\eps$ be the unique solution to
\be\label{e:sol.zero}
-\mathrm{div}(A_\eps \nabla \phi^0_\eps) = \mathrm{div}F_\eps\quad\mbox{in }B_R,\qquad
\phi^0_\eps = 0  \quad\mbox{on }\partial B_R.
\ee
Then, set $\phi_\eps^1(z):= \phi^0_\eps(z)- \phi^0_\eps(0)$, which vanishes at $z=0$ and still satisfies \eqref{e:approx.varphi}.
Notice that $$
\begin{aligned}
|\phi_\eps^1(0)| &
\leq C R\norm{\mathbb{Id}-A(z)}{L^\infty(B_{2\eps})}\norm{|\nabla u|}{L^\infty(B_{2\eps})}\\
&\leq C R\eps  \norm{|\nabla u|}{L^\infty(B_{2\eps})},
\end{aligned}
$$
where $C$ depends only on $\lambda, \Lambda$ and $L$. Let $k:=\mathcal{V}(0,\phi_\eps^1)$, then $k\leq N$. Indeed, by adapting a classical blow-up argument, one can show that not only the solutions of \eqref{e:sol.zero} satisfy the strong unique continuation principle, but also that the presence of $\nabla u$ in the right hand side $F_\eps$ (see \eqref{e:approx.varphi}) implies that the vanishing order of $\phi_\eps^1$ cannot exceed that of $u$. Then, we have two possibilities:
\begin{enumerate}
    \item either $k=N,$ and the thesis follows by taking $\varphi_\eps :=\phi_\eps^1;$
    \item or $k\leq N-1$ and there exists a $k$-homogeneous harmonic polynomial $P_k$ such that a Taylor expansion of the form \eqref{e:taylor-exp} holds true for $\phi^1_\eps$, with $P_k$ as leading term.
\end{enumerate}
We proceed then by considering the latter case. By combining the local boundedness of solution via Moser iteration and the doubling estimate \eqref{doub}, we get
\be\label{e:moser+doub}
\frac{|\phi^1_\eps(z)|}{|z|^k}\leq
C \eps \frac{\norm{|\nabla u|}{L^\infty(B_{R})}}{R^{k-1}}, \qquad \mbox{for }z \in B_{R/2},
\ee
with $C>0$ depending only on $\lambda, \Lambda,\eps_0$ and $L$. Therefore, up to perform a blow-up analysis, \eqref{e:moser+doub} implies a bound in $L^\infty_\loc(\R^2)$, in which the dependence from the approximating parameter $\eps\in (0,\eps_0)$ is explicit.\\
Now, by Lemma \ref{l:killing-poly}, we can construct a $A_\eps$-harmonic function whose leading term in the Taylor expansion coincides with the one of $\phi^1_\eps$. Indeed, there exist $R_0,\eps_0>0$, depending on $k$, such that for every $R<R_0$ and $\eps<\eps_0$ we can construct $\psi^k_\eps \in H^1(B_R)$ solving \eqref{e:homo} and satisfying an expansion in which the leading term coincides with the one of $\phi^1_\eps$. Namely, for every $r<R$, there exists $\Gamma_\eps$ such that
$$
   \psi_\eps^{k}(z) = P_k(z) + \Gamma_{\eps}(z),\qquad \mbox{in }B_r
$$
with $\Gamma_{\eps}$ satisfying \eqref{e:taylor-exp-remainder}. Then, if we consider $\phi^2_\eps(z):= \phi_\eps^1(z)-\psi^k_\eps(z)$, we have constructed a new solution to \eqref{e:approx.varphi} vanishing at the origin and such that
$$
|\phi_\eps^2|\leq C|z|^{N+\delta}\quad \mbox{in }B_{R/2},
$$
that is $\mathcal{V}(0,\varphi_\eps^2)\geq k+1$. Therefore, up to apply Lemma \ref{l:killing-poly} a finite number of times (at most $N$ iterations), we can always increase the vanishing order of $\phi^i_\eps$ by subtracting a $A_\eps$-harmonic function with a prescribed blow-up limit at the origin, up to obtain a  function with vanishing order greater or equal to $N$. We stress that although each iteration of the argument may lead to a change in the values of $R_0$ and $\eps_0$ required to apply Lemma \ref{l:killing-poly}, the statement of Lemma \ref{l:sus} remains true by choosing the minimum of these two finite collections of values as parameters. Therefore, after a finite number of iterations, we have constructed a solution $\varphi_\eps^{N}\in H^1(B_R)$ to \eqref{e:approx.varphi}, satisfying
$$
\varphi_\eps = \sum_{k_i\leq N} \left(a^{k_i}_\eps + P_{k_i}\right) =: P_\eps\quad \mbox{on }\partial B_{R_0},
$$
where $a_\varepsilon^{k_i}$ are the constants of Lemma \ref{l:killing-poly}, and therefore $P_\varepsilon$ is a harmonic polynomial of degree at most $N$, such that $\norm{P_\varepsilon}{L^\infty(B_{R_0})}\leq C \eps$, with $C>0$ depending only on $\lambda,\Lambda, L$ and $\eps_0$. Thus, by boundary Schauder estimates, we get that for $\eps<\eps_0$ and $\alpha \in (0,1)$
$$
\norm{\varphi}{C^{1,\alpha}(B_{R_0})}\leq C\left(\norm{F_\eps}{C^{0,\alpha}(B_{R_0})} + \norm{P_\eps}{C^{1,\alpha}(B_{R_0})}\right)\leq C\eps^{1-\alpha}
$$
with $C>0$ depending on $\lambda,\Lambda,L,\eps_0,\alpha$ and $R_0$. Ultimately, it implies the validity of the $C^{1,\alpha}$-estimate in \eqref{e:sei-tu}.
\end{proof}
We would like to emphasize that, in order to apply the blow-up procedure introduced in the proof of Theorem \ref{uniformwa} part (ii), it is necessary to establish also uniform convergence of the blow-up limits of the approximations $u_\eps$ at the fixed singularity. This result is purely two dimensional and relies on a refined version of the strategy first introduced by Hartman and Wintner in \cite{HartWint} and already used in the previous paper \cite{TerTorVit1}.
\begin{Lemma}\label{lem:xi_eps}
Let $R_0>0, \eps_0>0$ be the constants of Lemma \ref{l:sus}
and $u_\eps, A_\eps$ be the approximating pair defined in \eqref{e:coppia}. Then, for $\eps\in (0,\eps_0)$, if we set
         \be\label{e:hopeless}
         {\xi}_\eps(z):=\frac{i\overline{\nabla u_\eps(z)}}{z^{N-1}},\qquad \mbox{we have }\quad|{\xi}_\eps(0)|\geq C_1>0,\quad \norm{{\xi}_\eps}{C^{0,\omega}(B_{R_1})}\leq C_2,
         \ee
         where $\omega(t):=t |\log t|$ and $C_1,C_2, R_1>0$ depend only on $\lambda, \Lambda, L, \eps_0$ and $R_0$.
\end{Lemma}
\begin{proof}
We stress that, in light of Steps 1-2 of the proof of \cite[Lemma 3.6]{TerTorVit1}, we already know that for every $\eps\in (0,\eps_0)$ the statement in \eqref{e:hopeless} holds true with $C_1,C_2>0$, possibly depending on $\eps$. Therefore, in the proof, we limit to showing that this is not the case and that the constants can be chosen uniformly in $\eps$.\\
Before stating the result we set few notations: let $z = r e^{i\theta}$, where $r=|z|$ and $\theta=\mathrm{Arg}(z)\in[0,2\pi)$. Given $u_\eps$ and $A_\eps$ as before, it is possible to rewrite \eqref{e:numerarla} as
$$
-\mathrm{div}\left( \frac{A_\eps}{\sqrt{\mathrm{det}A_\eps}} \nabla u_\eps\right) = G_\eps \cdot \nabla u_\eps,\quad\mbox{with}\quad G_\eps:=A_\eps \nabla \left(\frac{\sqrt{\mathrm{det}A_\eps}-1}{\sqrt{\mathrm{det}A_\eps}}\right).
$$
Moreover, by Lemma \ref{l:sus}, we already know that for every $\alpha \in (0,1)$, there exists $C>0$ depending only on $u, \lambda,\Lambda,L, \eps_0$ and $R_0$ such that
\be\label{e:unif-2326}
\norm{u_\eps}{C^{1,\alpha}(B_{R})}\leq C
\ee
for every $R\in (0,R_0)$ and $\eps\in (0,\eps_0)$. Notice also that by \eqref{e:approximating-A} and \eqref{e:matrix-esti}, we get $G_\eps \in L^\infty(B_1)$ with a norm uniformly bounded for $\eps \in (0,\eps_0).$
Then, consider $\Theta_\eps$ the quasiconformal $C^{1,1-}$-diffeomorphism associated to $A_\eps$ already introduced by the authors in Step 1 of \cite[Lemma 3.6]{TerTorVit1}. Moreover, by the uniform bound \eqref{e:matrix-esti}, we infer that for every $\alpha \in (0,1)$ the $C^{1,\alpha}$-norm of $\Theta_\eps$ is uniformly bounded for $\eps \in (0,\eps_0)$ and, having a nonzero determinant of the Jacobian matrix, it is invertible in a neighborhood of the origin. Thus, the function
$$
\tilde{u}_\eps := u \circ \Theta_\eps^{-1}\colon B_{C R_0}\to \R
$$
is well defined, where $C$ is a constant depending only on $\lambda, \Lambda, L, \eps_0$ and $R_0$. Moreover, by \eqref{e:unif-2326}, for every $\alpha \in (0,1)$, there exists $C>0$ depending only on $u, \lambda,\Lambda,L, \eps_0$ and $R_0$ such that
\be\label{e:2329}
\norm{\tilde{u}_\eps}{C^{1,\alpha}(B_{R})}\leq C
\ee
for every $R\in (0,C R_0),\eps \in (0,\eps_0)$. First, by exploiting the definition of $\Theta_\eps$ and the fact that is defined in such way that $\Theta_\eps(0)=0$, then it is straightforward to deduce that for every $\phi \in C^\infty_c(B_{C R_0})$
$$
\int_{B_{C R_0}}\nabla \tilde{u}_\eps \cdot \nabla \phi \,dxdy=
\int_{B_{C R_0}}\phi ( \tilde{G}_\eps \cdot \nabla \tilde{u}_\eps) \,dxdy,\qquad \mbox{with }\tilde{G}_\eps := G_\eps \circ \Theta_\eps^{-1}.
$$
Clearly, we have $\tilde{G}_\eps \in L^\infty(B_{C R_0})$ with norm uniformly bounded for $\eps\in (0,\eps_0)$. Now, being $\Theta_\eps^{-1}$ a diffeomorphism, we get $\mathcal{V}(0,\tilde{u}_\eps)= N\geq 2$, and so
$$
\tilde{u}_\eps(z)= O(|z|^N),\,\,\,\quad |\nabla \tilde{u}_\eps(z)|=O(|z|^{N-1})
\qquad\mbox{as }|z|\to 0^+.
$$
Now, set $\tilde{\xi}_\eps(z):=i\overline{\nabla \tilde{u}_\eps(z)}/z^{N-1} \in L^\infty(B_{C R_0})$, and proceed with two steps.\\

Step 1: $\tilde{\xi}_\eps$ is bounded uniformly-in-$\eps$. By \cite[Section 7]{HartWint}, the following Cauchy formula holds true
\begin{equation}\label{e:cauchy}
\begin{aligned}
2\pi\tilde{\xi}_\eps(\zeta) &=i\int_{\partial B_R}\frac{\overline{\nabla \tilde{u}_\eps(z)}}{z^{N-1}(z-\zeta)}\, dz \ - \ \int_{B_R}\frac{\Delta \tilde{u}_\eps(z)}{z^{N-1}(z-\zeta)}\, dz\\ &=
i\int_{\partial B_R}\frac{\overline{\nabla \tilde{u}_\eps(z)}}{z^{N-1}(z-\zeta)}\, dz \ + \ \int_{B_R}\frac{\tilde{G}_\eps(z) \cdot \nabla \tilde{u}_\eps(z)}{z^{N-1}(z-\zeta)}\, dz,
\end{aligned}
\end{equation}
where $R\in (0,C R_0)$ is fixed and $\zeta$ belongs to a small neighbourhood of the origin. In particular, let $\rho>0$ be such that
$$
\lVert\tilde{G}_ \eps\lVert_{L^\infty(B_{C R_0})}\int_{B_\rho}\frac{dz}{|z-\zeta|}\leq \pi,
$$
and set $R:=\min\{\rho, C R_0/2\}$. Then, by the definition of $\rho>0$, for every $\zeta \in B_{R/2}$ we have
\begin{align*}
\int_{B_{R/2}}\frac{|\tilde{G}_\eps(z)||\nabla \tilde{u}_\eps(z)|}{|z-\zeta||z|^{N-1}}\, dz &\leq \pi \sup_{z \in B_{R/2}}|\tilde{\xi}_\eps(z)|,\\
\int_{B_R\setminus B_{R/2}}\frac{|\tilde{G}_\eps(z)||\nabla \tilde{u}_\eps(z)|}{|z|^{N-1}|z-\zeta|}\, dz &\leq \frac{C}{R^{N-1}}\norm{|\nabla \tilde{u}_\eps|}{L^\infty(B_{CR_0})} \lVert|\tilde{G}_\eps|\lVert_{L^\infty(B_{CR_0})}\int_{B_R\setminus B_{R/2}}\frac{1}{|z-\zeta|}\, dz\\
&\leq \frac{C}{R^{N-1}}\int_{B_R\setminus B_{R/2}}\frac{1}{|z-\zeta|}\, dz
\end{align*}
where in the last inequality we exploited the uniform-in-$\eps$ bound of $\tilde{G}_\eps$ and \eqref{e:2329}. Thus, by collecting the previous estiamtes, for every $\zeta \in B_{R/2}$ we get
$$
2\pi |\tilde{\xi}_\eps(\zeta)| \leq \frac{2\pi}{R^{N-1}|R-|\zeta||}\norm{|\nabla \tilde{u}_\eps|}{L^\infty(\partial B_{R})} +  \frac{C}{R^{N-1}}\int_{B_R\setminus B_{R/2}}\frac{dz}{|z-\zeta|} + \pi \sup_{z \in B_{R/2}}|\tilde{\xi}_\eps(z)|,
$$
where $C$ is a constant depending only on $u,\lambda,\Lambda,L,\eps_0$ and $R_0$. Therefore, taking the supremum over $\zeta \in B_{R/2}$ on both the sides of the previous inequality, we deduce
$$
2\pi \sup_{\zeta \in B_{R/2}}|\tilde{\xi}_\eps(\zeta)| \leq \frac{C}{R^{N}}\norm{|\nabla  \tilde{u}_\eps|}{L^\infty(\partial B_{R})} +  \frac{C}{R^{N-1}}\int_{B_R\setminus B_{R/2}}\frac{dz}{|z-\zeta|} + \pi \sup_{z \in B_{R/2}}|\tilde{\xi}_\eps(z)|,
$$
which implies the claimed result once we subtract the last term of the right hand side.\\

Step 2: $\tilde{\xi}_\eps$ is bounded in $C^{0,\omega}$ uniformly-in-$\eps$. By \eqref{e:cauchy}, for every $\zeta_1,\zeta_2 \in B_{\rho/4}$ we have
$$
\begin{aligned}
|\tilde{\xi}_\eps(\zeta_1) - \tilde{\xi}_\eps(\zeta_2)| \leq & \,\,\frac{1}{2\pi}\int_{\partial B_{\rho/2}}|\tilde{\xi}_\eps(z)|\frac{|\zeta_1-\zeta_2|}{|z-\zeta_1||z-\zeta_2|}\,d\sigma\\
&\, + \frac{1}{2\pi}\int_{B_{\rho/2}} |\tilde{G}_\eps(z)||\tilde{\xi}_\eps(z)| \frac{|\zeta_1-\zeta_2|}{|z-\zeta_1||z-\zeta_2|}\, dz \\
\leq &\,\, \frac{1}{2\pi}\left(\lVert\tilde{\xi}_\eps\lVert_{L^\infty(B_{\rho/2})}\int_{\partial B_{\rho/2}}\frac{1}{|z-\zeta_1||z-\zeta_2|}\,d\sigma\right) |\zeta_1-\zeta_2| \\
&+ \frac{1}{2\pi}\left(\lVert\tilde{\xi}_\eps\lVert_{L^\infty(B_{\rho/2})}
\lVert|\tilde{G}_\eps|\lVert_{L^\infty(B_{\rho/2})}\right) |\zeta_1-\zeta_2||\log|\zeta_1-\zeta_2||.
\end{aligned}
$$
Therefore, there complex-valued log-Lipschitz continuous function $\tilde{\xi}_\eps$ satisfies
$$
\nabla \tilde{u}_\eps(z) =\overline z^{N-1}\tilde{\xi}_\eps(z)=r^{N-1}e^{-i(N-1)\theta}\tilde{\xi}_\eps(z),
$$
and the claimed estimate for $\xi_\eps$ follows by exploiting the diffeomorphism $\Theta_\eps$.
\end{proof}

Next we prove a-priori estimate in nodal components, uniform-in-the approximating scheme.

\begin{Theorem}[\emph{A priori} uniform-in-$\eps$ Schauder estimates in two-dimensions in  nodal domains]\label{uniformwa_sector}
In the same setting of Theorem \ref{t:effective}, let $R_0>0, \eps_0>0$ be the constants of Lemma \ref{l:sus} and consider $A_\eps, u_\eps := u +  \varphi_\eps$ to be the approximating pair defined in \eqref{e:coppia},
with $S(u_\eps) \cap B_{R_0} = \{0\}$. Let $w_\eps \in H^1(B_1,|u|^a)$ be a (possibly discontinuous, if $a\geq 1$) solution to
$$
\mathrm{div}\left(|u_\eps|^a A_\eps\nabla w_\eps\right)=0 \qquad \mbox{in }B_{R_0},
$$
 in the sense of Definition \ref{definition.energy.a}. Then, for each nodal component $\Omega\subset B_{R_0}\setminus u_\eps^{-1}(0)$ the following holds true:
if $a\geq0$ and $w_\eps\in C^{1,\alpha}(\Omega)$, then
\begin{equation}\label{eq:nodal}
\left\|w_\eps\right\|_{C^{1,\alpha}(B_{R_0/2}\cap \Omega)}\leq C\|w_\eps\|_{L^\infty(\Omega)},
\end{equation}
for some constant $C>0$ independent of $\eps$.
\end{Theorem}
\begin{proof}
We know that $\mathcal{V}(0,u_\eps)=N$, for $\eps \in (0,\eps_0)$. By choosing, possibly a smaller radius, we may assume that the nodal lines of $u$ and $u_\eps$ hit the boundary $\partial B_{R_0}$ transversally in exactly $2N$ points, so that there are exactly $2N$ nodal components of $u$ and $u_\eps$. At first, we need a uniform-in-$\eps$ expansion of $u_\eps$ at the origin. In light of \eqref{e:hopeless}, let us define
        \be\label{e:xi-k}
        \begin{aligned}
        \xi_{\eps}(z) &:= \frac{i\overline{\nabla u_\eps(z)}}{z^{N-1}}
        =\xi(z) + i\frac{\overline{\nabla \varphi_\eps( z)}}{z^{N-1}},
        \end{aligned}\qquad \mbox{for }z \in B_{R_0/2}.
        \ee
Therefore, for $\varepsilon$ sufficiently small, according to \eqref{e:hopeless}, the sequence $\xi_{\varepsilon}$ is uniformly bounded both from above and away from zero at $z=0$, and uniformly $\alpha$-H\"{o}lder continuous for every $\alpha \in (0,1)$. Then, the result follows once we show that the bound in \eqref{eq:nodal} holds uniformly as $\eps \to 0^+$.

Since the proof mimics the contradiction argument in the proof of Theorem \ref{uniformwa} (see Section \ref{s:proof-thm15}), we just sketch the strategy with a major attention on the differences.

Thus, fixed $\alpha \in (0,1)$, as $k\to+\infty$, let $\eps_k < \eps_0$ and set $A_k:=A_{\eps_k}, \Omega_k$ to be a sequence of nodal domains of $u_{\eps_k}$. For the sake of simplicity, we denote
$$
u_k:=u_{\eps_k} \in C^{1,\alpha}(B_{R_0}),\qquad \xi_k:=\xi_{\eps_k}\in C^{0,\alpha}(B_{R_0})\quad\mbox{and}\quad
w_k:=w_{\eps_k} \in C^{1,\alpha}(B_{R_0}\cap \Omega_k)
$$
is the solution to \eqref{eqwa} in the sense of Definition \ref{definition.energy.a}, satisfying $\norm{w_k}{L^\infty(\Omega_k)}=1$.
Set $\eta\in C^\infty_c(B_1;[0,1])$ to be a radially decreasing cut-off function such that $\eta\equiv1$ in $B_{R_0/2}$ and  $\mathrm{supp}\eta=\overline{B_{5R_0/8}}=:\overline{B}$.\\ Then, if we show that $(\eta w_k)_k$ is uniformly bounded in $C^{1,\alpha}(\overline{B_{R_0}\cap\Omega_k})$, using that $\eta \equiv 1$ in $B_{R_0/2}$, we infer the same bound for $(w_k)_k$ in $C^{1,\alpha}(\overline{B_{R_0/2}\cap\Omega_k)}$. Notice also that, since for every $z_0 \in B_{R_0}\setminus B$ we have $(\eta w_k)(z_0)=|(\nabla w_k) \eta|(z_0)=0$ for every $k>0$, then it is sufficient to ensure that the following seminorms
$$
\max_{i=1,\dots,n}\left[\partial_{i}(\eta w_k)\right]_{C^{0,\alpha}(\overline{B_{R_0}\cap\Omega_k})} \leq C,
$$
uniformly as $k\to +\infty$. Thus, by contradiction suppose that, up to subsequences,
$$
L_k:=\max_{i=1,\dots,n}\left[\partial_{i}(\eta w_k)\right]_{C^{0,\alpha}(\overline{B_{R_0}\cap\Omega_k})}\to +\infty,$$
that is, there exist two sequences of points $z_k,\zeta_k\in B\cap\overline{\Omega_k}$, such that
$$
L_k = \frac{|\partial_{i_k}(\eta w_k)(z_k)-\partial_{i_k}(\eta w_k)(\zeta_k)|}{|z_k-\zeta_k|^\alpha}\to+\infty.
$$
Naturally, up to relabeling we can assume that $i_k = 1$ for every $k>0$ and we can proceed as in Section \ref{s:section-proof-gradient2}. Indeed, the existence of a diverging sequence $L_k$ implies the existence of $ \hat{z}_k \in B_{3R_0/4}\cap\overline{\Omega_k}, r_k \searrow 0^+$  (as in Lemma \ref{l:involuto}) and of a sequence
$$
\overline{W}_k(z) =\frac{\eta( \hat{z}_k)}{L_kr_k^{1+\alpha}}\left(w_k( \hat{z}_k+r_k z)-w_k( \hat{z}_k) - r_k (\nabla w_k)( \hat{z}_k)\cdot z\right),
$$
of rescaled solutions, in the sense of Definition \ref{definition.energy.a}. More precisely let  $\overline{A}_k$ be defined as in \eqref{e:blow-up-finalproof}, $\nabla W_k(0)$ as in \eqref{e:grad} and
$$
U_k(z):=\frac{u_k(\hat{z}_k+r_k z)}{\rho_k}\qquad\mbox{with }\rho_k:=H(\hat{z}_k,u_k,r_k)^{1/2}= H(\hat{z}_k,u+\varphi_k,r_k)^{1/2}.
$$
All functions are defined in the rescaled nodal component of $U_k$: $\tilde\Omega_k=:(\Omega-\hat {z_k})/r_k$.
Then the following holds true:
$$
\int_{\tilde \O_k}|U_k|^a \overline{A}_k\nabla \overline{W}_k\cdot\nabla\phi\,dxdy = a\int_{\tilde\O_k} \phi  |U_k|^{a-1}  (\nabla U_k \cdot \overline{A}_k\nabla W_k(0))\,dxdy,
$$
for every $\phi \in C^\infty_c(\R^n)$. By introducing the auxiliary sequence $\overline{V}_k,$ is it possible to prove compactness of the sequence $\overline{W}_k$ and to show that, up to subsequences, it converges in $C^{1,\beta}_\loc(\Omega)$, for every $\beta \in (0,\alpha)$, and uniformly on every compact set of $\O$, to a function $\overline{W}$ whose gradient is non-constant and globally $\alpha$-H\"{o}lder continuous in $\O$. Here $\O$ is the limit set of the $\tilde\O_k$.\\

Now, by following the last part of Section \ref{s:proof-thm15}, we can directly infer that the limit $\overline{W}$ is a solution of the limit degenerate equation associated to the blow-up limit $U$ of the sequence $U_k$. More precisely, we infer that
\begin{enumerate}
    \item[\rm{(i)}] if $a\geq 1$, then $\overline{W}$ satisfies $\mathrm{div}\left(|U|^a \nabla \overline{W}\right)=0$ in $\O$;
\item[\rm{(ii)}] if $a \in [0,1)$, then $\overline{W}$ satisfies
$$
\mathrm{div}\left(|U|^a \nabla \overline{W}\right)=0\quad\mbox{in }\O, \qquad
|U|^a \nabla \overline{W} \cdot \nabla U = 0\quad\mbox{on }R(U),
$$
\end{enumerate}
where both the cases must be understood in the sense of Definition \ref{definition.energy.a}. At this, we finally make use of the construction developed in Section \ref{s:appr} in order to characterize the limit $U$ and its limiting domain $\O$. First, being $\mathcal{V}(0,u)=\mathcal{V}(0,\varphi_k)=N$ for every $k$ sufficiently large, we infer that
$$
U_k(z)= O(|z|^N),\quad |\nabla U_k(z)|=O(|z|^{N-1})
$$
and, by rescaling \eqref{e:xi-k}, we get
$$
i\overline{\nabla U_k(z)}= \frac{r_k}{\rho_k}\xi_k(\hat{z}_k+ r_k z)(\hat{z}_k+ r_k z)^{N-1}\qquad \mbox{for }z \in \frac{B_{R_0/2}-\hat{z}_k}{r_k}.
$$
By exploiting the compactness of $U_k$ in $C^{1,\alpha}_\loc(\R^2)$ and the uniform $\alpha$-H\"{o}lder estimate for the sequence $\xi_k$, it is possible to give a complete characterization of the limit $U$. Indeed, we distinguish two cases:
        \begin{enumerate}
        \item[(\rm{i})] if $r_k/|\hat{z}_k| \to 0^+$, then we get
           $$
           i\overline{\nabla U_k(z)}= r_k\frac{|\hat{z}_k|^{N-1}}{\rho_k}\xi_k(\hat{z}_k+ r_k z)\left(\frac{\hat{z}_k}{|\hat{z}_k|}+ \frac{ r_k}{|\hat{z}_k|} z\right)^{N-1} \qquad \mbox{in }\frac{B_{R_0/2}-\hat{z}_k}{r_k},
           $$
           where the last two factors converge to a non-zero complex number. On the other hand, since we already know that $U_k$ is converging in $C^{1,\alpha}_\loc(\R^2)$, we deduce the existence of $k_0>0$ sufficiently large, such that
           $$
            \frac{r_k}{|\hat{z}_k|}\frac{|\hat{z}_k|^N}{\rho_k}\leq C,\quad
           $$
           for every $k>k_0$. Thus, in light of the normalization $\rho_k$, there exists $\zeta_1 \in \C\setminus \{0\},$ such that $i\overline{\nabla U(z)} = \zeta_1$ for every $z \in \C$, namely that $U$ is a linear function and $\O$ is a half plane;
            \item[(\rm{ii})] if $r_k/|\hat{z}_k| \geq C$, for some $C>0$ universal, we proceeding by collecting the scaling factor $r_k$. Thus,
            $$
           i\overline{\nabla U_k(z)}= \frac{r_k^{N}}{\rho_k}\xi_k(\hat{z}_k+ r_k z)\left(\frac{\hat{z}_k}{r_k}+ z\right)^{N-1} \qquad \mbox{in }\frac{B_{R_0/2}-\hat{z}_k}{r_k}.
           $$
           By exploiting the compactness of $U_k$, we deduce that for $\rho_k \geq C r_k^N$ for $k$ sufficiently large and, being  $\hat{z}_k/r_k \in B_{1/C},$ by compactness we finally deduce that
           $$
           i\overline{\nabla U(z)} = \lambda(\zeta_2 + z)^{N-1}\qquad\mbox{for every }z \in \R^2,
           $$
            for some $\lambda \in \R\setminus\{0\}$ and $\zeta_2 \in \C$. Being $U$ a harmonic polynomial, we immediately deduce that it coincides (up to a translation) with a $N$-homogeneous harmonic polynomial. In such a case, $\O$ is an angular sector.
        \end{enumerate}
       Finally, the contradiction follows by applying a Liouville type theorem: indeed, since by construction $\overline{W}(0) = |\nabla \overline{W}|(0)=0$, we deduce, in light of global $\alpha$-H\"{o}lder regularity of the gradient of $\overline{W}$, that
$$|\overline{W}(x)|\leq C\left(1+|x|\right)^{1+\alpha}\quad\mbox{in }\R^2.$$
Finally, regardless of whether the limit blow-up $U$ is a linear function or a $N$-homogeneous harmonic polynomial, Liouville's Theorem \ref{liou_a>-1} implies that $\overline{W}$ must be a linear function, since the boundary $\partial\O$ is connected (see Remark \ref{rem:conn_bdry}). This contradicts the fact that $\overline{W}$ has a non-constant gradient.
\end{proof}

\subsection{A posteriori \texorpdfstring{$C^{1,1-}$}{Lg} regularity around a given singular point}
The aim of this section is the proof of Theorem \ref{t:effective}, by exploiting the approximation procedure and the  a priori bounds on nodal components given in the previous subsection.

\begin{proof}[Proof of Theorem \ref{t:effective}]

Let $u \in C^{1,1-}(B_1)$ be a solution to \eqref{equv} with
$$
S(u)\cap B_1 = \{0\},\quad \mbox{and}\quad\mathcal{V}(0,u)=N\geq2.
$$
Through the proof, we set $R_0>0,\eps_0>0$ to be the constant introduced in Lemma \ref{l:sus}. Now, let $P_u$ be the unique $N$-homogeneous blow-up limit of $u$ at the origin and $\xi \in C^{0,\omega}(B_{R_0/2})$ where $\omega(t):=t|\log t|$ and
$$
\xi(z) := \frac{i\overline{\nabla u(z)}}{z^{N-1}}\quad \mbox{and such that }\,\xi(0)\geq C,
$$
for some positive constant $C>0$. Then, for $\eps\in (0,\eps_0)$ we consider the pair
$$
u_\eps := u +  \varphi_\eps\colon B_{R_0}\to \R\qquad\mbox{and}\qquad A_\eps := A + (\mathbb{I}-A)\eta_\eps \in \mathcal{A},
$$
where $\varphi_\eps$ is the approximating sequences associated to $P_u$ constructed in Section \ref{s:appr} and satisfying $S(u_\eps) \cap B_{R_0} = \{0\}$. Without loss of generality we can assume that the $2N$ nodal lines of $u$ intersect the boundary $\partial B_{R_0}$ transversally in $2N$ distinct points (otherwise we may decrease the radius of the ball). In this way, u has exactly $2N$ nodal regions in $B_{R_0}$ and the same holds for $u_\eps$ for $\eps$ sufficiently small. Le us consider a nodal component $\O\subset B_{R_0}$ of $u$  and a family of nodal components $\O_{\eps}\subset B_{R_0}$ of $u_\eps$ converging to $\O$ with respect to the Hausdorff distance.
As in the previous Subsection \ref{sec:hodo}, we map $\O$ and $\O_\eps$ to bounded regions of the upper half-plane through the quasiconformal hodograph transformations defined as
$$
\Theta(x,y):=(\overline{u}(x,y),u(x,y))\quad \mbox{and}\quad \Theta_\eps(x,y):=(\overline{u}_\eps(x,y),u_\eps(x,y)).
$$
and the corresponding matrices
$$
B(x,y):=
\begin{pmatrix}
(\mathrm{det}A)(\Theta^{-1}(x,y)) & 0 \\
0 & 1
\end{pmatrix}\;\quad \mbox{and}\quad B_\eps(x,y):=
\begin{pmatrix}
(\mathrm{det}A_\eps)(\Theta_\eps^{-1}(x,y)) & 0 \\
0 & 1
\end{pmatrix}.
$$
We refer to \eqref{e:A-ham-conjugate} for the general construction of the $A$-harmonic conjugate. Note also that there is a gain of H\"older continuity: indeed, if we set $\tilde{\O}:=\Theta(\O)\subset \{y>0\},$ then each $B_\eps\in C^{0,\alpha}(\tilde\O)$ as it is the identity in a neighbourhood of zero, where $\Theta_\eps^{-1}$ is not more than $C^{0,\bar\alpha}$;  while $B\in C^{0,\bar\alpha}(\tilde\O)$ and not more. On the other hand, obviously $B_\eps$ converges uniformly to $B$. We denote $W:=w\circ\Theta^{-1}$, which is a $C^{1,\bar\alpha}(\tilde\Omega)$ solution to
$\mathrm{div}\left(|y|^a B\nabla W\right)=0$ in $\tilde{\O}$, in the sense of Definition \ref{definition.energy.a}. More precisely,
$$
\int_{\tilde{\O}}|y|^a B\nabla W\cdot\nabla\phi\,dxdy=0,\quad \text{for every } \phi\in C^\infty_c(\{y>0\}).
$$
Now we solve the mixed boundary value problems
$$
\mathrm{div}\left(|y|^a B_\eps\nabla W_\eps\right)=0 \quad \mbox{in }\tilde\O,\quad
|y|^a \partial_y W_\eps = 0\quad \mbox{on } \partial{\tilde \O}\cap\{y=0\},\quad
W_\eps=W\quad \mbox{on }\partial{\tilde \O}\cap\{y>0\},
$$
whose solution must be intended as the unique minimizer $W_\eps$ of
$$
\min\left\{\int_{\tilde\O}|y|^a B_\eps\nabla f\cdot \nabla f\,dxdy \colon \, f=W+V ,\,\, V\in \tilde H^1_0(\tilde\O,|y|^a\,)\right\}.
$$
Here $\tilde H^1_0(B_1,|y|^a)$ is the defined as
the completion in $H^1(\tilde\O,|y|^a)$ of the space
$$\{f\in C^\infty(\overline{\tilde \O})\colon f\equiv 0\; \mbox{in a neighbourhood of}\; \partial{\tilde \O}\cap\{y>0\}\}.$$
Note that, in light of the  Poincar\'e inequality in \cite[Lemma 3.2]{JeonVita},
\[
(f,g)_\eps :=\int_{\tilde\O}|y|^a B_\eps\nabla f\cdot \nabla g\,dxdy
\]
defines a family of equivalent scalar products in $\tilde H^1_0(B_1,|y|^a)$. We write $V_\eps:=W_\eps-W$ and we wish to prove that $V_\eps\to 0$ so that $W_\eps$ is really and approximating family. To this aim, we first observe that there is a uniform bound on their norms in $\tilde H^1_0(B_1,|y|^a)$. Next we test both equations by $V_\eps$  to infer
$$
\int_{\tilde\O} |y|^aB_\eps \nabla V_\eps\cdot \nabla V_\eps\,dxdy=\int_{\tilde\O} |y|^a(B-B_\eps) \nabla W\cdot \nabla V_\eps\,dxdy\;,
$$
where the right hand side converges to zero as $\eps\to 0$, since $B_\eps\to B$ uniformly. By Schauder estimates, the convergence holds in the $C^{1,\alpha}$ topology on compact subsets of $\tilde\Omega$.

Now consider $w_\eps:=W_\eps\circ\Theta_\eps$, defined in $\Theta_\eps^{-1}(\tilde\Omega)=\Theta_\eps^{-1}(\Theta(\Omega))$ which are solutions to
$$
\mathrm{div}\left(|u_\eps|^a A_\eps\nabla w_\eps\right)=0\qquad\mbox{in }\Theta_\eps^{-1}(\tilde{\O})
$$
in the sense of Definition \ref{definition.energy.a}.  On the other hand, thanks to the previous results in \cite[Theorem 1.1]{TerTorVit1}, as each $B_\eps$ is bounded in $C^{0,\alpha}(\omega)$ (not uniformly-in-$\eps$), we have bounds of the $W_\eps$ in $C^{1,\alpha}(\omega)$, for all $\omega$ whose relative closure with respect to $\{y>0\}$ is contained in $\tilde\O$. As $\Theta_\eps$ is also  $C^{1,\alpha}$, we infer the same regularity for the solution $w_\eps$.

Unfortunately, at this stage, we still can not say that such $C^{1,\alpha}$ bounds in the nodal component are uniform-in-$\eps$. The reason is that the $C^{0,\alpha}$-bounds on the matrices $B_\eps$ can not be uniform-in-$\eps$ in a neighbourhood of the singular point, as can be seen from the local behavior of the inverse of the hodograph map $\Theta$ at the origin. Here, Theorem \ref{uniformwa_sector} comes to our rescue by ensuring that the $C^{1,\alpha}$-bounds in a nodal component are uniform-in-$\eps$, for $\eps<\eps_0$. Thanks to this uniformity, as $w$ is the limit of the $w_\eps$, it inherits the regularity properties of its approximants.
\end{proof}

\subsection{Proof of Theorem \ref{effective2}}\label{s:fin-effective} The aim of this last section is to show uniformity-in-$\mathcal S_{N_0}$ and \emph{a posteriori} $C^{1,1-}$-regularity for the solution $w$ to \eqref{eqwa}. We split the proof into two steps.\\

Step 1: covering. Any point $x$ in the ball $B_{1}$ either
$$
\mathrm{(i)}\,\,\, x \in B_1\setminus Z(u),\qquad
\mathrm{(ii)}\,\,\, x \in B_1\cap R(u),\qquad
\mathrm{(iii)}\,\,\, x \in B_1\cap S(u).
$$
Then, there exists a small radius $r_x>0$ such that either $B_{r_x}(x)\cap Z(u)=\emptyset$ in case (i), $B_{r_x}(x)\cap S(u)=\emptyset$ in case (ii), or $B_{r_x}(x)\cap S(u)=\{x\}$ in case (iii). Being $w$ a continuous solution to \eqref{eqwa} with $A\in\mathcal A$, and $a\geq0$, we know that in any case $w\in C^{1,1-}_\loc(B_{r_x}(x))$. The latter regularity relies in classical Schauder theory for second order uniformly elliptic equations in case (i), \cite[Theorem 1.1 and Lemma 2.12]{TerTorVit1} in case (ii), and Theorem \ref{t:effective} in case (iii). Hence, by compactness one can extract a finite covering of $\overline{B_{r}}$ (for any $0<r<1$) with the same property, obtaining the desired regularity together with an estimate depending on $u,Z(u)$. In particular, for any $\alpha\in(0,1)$ there exists a constant $C>0$ depending on $\alpha,a,\lambda,\Lambda,L,u$ and $Z(u)$ such that
\begin{equation*}
\left\|w\right\|_{C^{1,\alpha}(B_{1/2})}\leq C\left\|w\right\|_{L^\infty(B_1)}.
\end{equation*}

Step 2: uniformity-in-$\mathcal S_{N_0}$.
    Once solutions are $C^{1,1-}$ regular, the uniformity of the regularity estimate above in $\mathcal S_{N_0}$ follows from Theorem \ref{uniformwa} part (ii).

\section{Liouville theorem}\label{sec:liouville}
This section is devoted to the proof of the Liouville theorem for entire solutions to \eqref{e:a-entire}.

\subsection{Strong unique continuation principle for degenerate or singular equations}\label{sec:unique}
The aim of this section is the proof of a strong unique continuation property for weak solutions  $w \in H^{1,a}(B_1^+)$ to
\begin{equation}\label{eqa}
\mathrm{div}\left(x_n^{a}A\nabla w\right)=0 \quad\mathrm{in \ } B_1^+,\qquad
x_n^{a}A\nabla w\cdot  e_{n}=0 \quad\mathrm{in \ } B_1',
\end{equation}
with $a>-1$ and $A(x',x_n)$ symmetric, uniformly elliptic, with Lipschitz continuous entries and such that
\be\label{e:matrice-dopo}
A(0,0)=\mathbb I,\qquad A(x',x_n)=\left(\begin{array}{c|c}
B(x',x_n)&{\mathbf 0}\\\hline
\mathbf 0 &m(x',x_n)
\end{array}\right)
\ee
for $m$ strictly positive. 
\begin{Proposition}
Let $a>-1$ and $w$ be a nontrivial weak solution to \eqref{eqa} with $w(0)=0$. Then, up to consider the symmetric extension of $w$ across $\Sigma=\{x_n=0\}$, there exist $k\in\mathbb N\setminus\{0\}$ and a homogeneous entire $L_a$-harmonic polynomial $P$ in $\R^n$ of degree $k$ which is symmetric with respect to $\Sigma$ such that
$$
w_r(x)=\frac{w(r x)}{r^k}\to P(x)\qquad\mathrm{as \ }r\to0^+
$$
in $H^{1,a}_\loc(\R^n)$ and $C^{1,\alpha}_\loc(\R^n)$ for any $\alpha\in(0,1)$.
\end{Proposition}
\proof
Since the proof follows by extending known results of the case $a\in (-1,1)$ to the more general $a>-1$, we simply sketch the main ideas. First, by exploiting the local $C^{1,1-}$-regularity results of weak solutions to \eqref{eqa} in \cite{SirTerVit1,TerTorVit1}, one can provide an Almgren type monotonicity formula with center in the origin. In fact, one can argue as in \cite{DelFelVit}, and deduce the monotonicity from a Pohozaev type identity in the full range $a>-1$ and with coefficients $A$ depending also on the vertical variable $x_n$. Then, the blow-up sequence
$$\tilde w_r(x)=\frac{w(r x)}{\sqrt{H(r)}},\qquad\mbox{with}\quad H(r)=\frac{1}{r^{n-1+a}}\int_{\partial B_r}|x_n|^a\mu w^2\,d\sigma,\qquad \mu(x)=\frac{A(x)x\cdot x}{|x|^2}$$
converges, up to a subsequence, in $H^{1,a}_\loc(\R^n)\cap C^{1,1-}_\loc(\R^n)$  to a $\gamma$-homogeneous function $W\in H^{1,a}_\loc(\R^n)$
solving
$$
\mathrm{div}\left(x_n^{a}\nabla W\right)=0 \quad \mathrm{in \ } \R^n_+,\qquad
x_n^{a}\partial_n W=0 \quad \mathrm{on \ } \Sigma.
$$
in a weak-sense. Then, the $\gamma$-homogeneity implies the growth condition
$$|W(x)|\leq C(1+|x|)^\gamma\qquad\mathrm{in \ }\R^n_+$$
and then the Liouville theorem \cite[Lemma B.2]{DelFelVit} implies that $W$ is a homogeneous polynomial and hence $\gamma=\mathrm{deg}(W) \in 1+ \mathbb N$. Finally, by exploiting the validity of a Weiss type monotonicity formula, one can deduce the convergence of the blow-up sequeunce
$$w_r(x)=\frac{w(r x)}{r^k}\to P(x),$$
where $P$ is a homogeneous polynomial satisfying the thesis of the Proposition.
\endproof
\begin{Lemma}\label{tracepolynomial}
Let $P$ be a nontrivial $L_a$-harmonic polynomial which is symmetric with respect to $\Sigma=\{x_n=0\}$. Then $\mathrm{deg}(P)=\mathrm{deg}(p)$ where $p(x')=P(x',0)$.
\end{Lemma}
\proof
Let us prove the result for a homogeneous nontrivial $L_a$-harmonic polynomial $P$ which is symmetric with respect to $\Sigma=\{x_n=0\}$. We proceed by induction on $\mathrm{deg}(P)=k\in\mathbb N$. If $k=0$ then $P$ is constant and $p$ too. Let us suppose the result true for $k$ and prove it for $k+1$. Let us first remark that $P$ must depend on the variable $x_i$ for some $i\in\{1,\dots,n-1\}$. Otherwise fixed $a>-1$, up to multiplicative constants, the equation $L_a W=0$ has only two linearly independent solutions which depend only on $x_n$; that is
$$
W_1(x_n):=1,\quad\mbox{and}\quad W_2(x_n):=x_n|x_n|^{-a}\qquad\text{($W_2(x_n):=\log |x_n|$ when $a=1$)}.
$$
Clearly, since $W_2$ is not an admissible solution,
it follows that $W$ is a constant, in contradiction with the assumption $\mathrm{deg}(P)=k+1\geq1$. Hence, since $P$ depends on $x_i$ for some $i\in\{1,...,n-1\}$, then $Q=\partial_i P$ is still a nontrivial $L_a$-harmonic polynomial which is symmetric with respect to $\Sigma=\{x_n=0\}$ with $\mathrm{deg}(Q)=k$ and this is due to the homogeneity of $P$. Then by inductive hypothesis we know that the trace $q(x')=Q(x',0)$ is still a nontrivial polynomial with $\mathrm{deg}(Q)=\mathrm{deg}(q)$. But $q=\partial_ip$ which implies that $\mathrm{deg}(p)=k+1$.\\
Ultimately, the case of non-homogeneous polynomials follows by a blow-down argument.
\endproof
\begin{Corollary}\label{tracepolynomial2}
Let $P$ be a $L_a$-harmonic polynomial which is symmetric with respect to $\Sigma=\{x_n=0\}$. If $P\equiv 0$ on $\Sigma$ then $P$ is trivial.
\end{Corollary}

We remark that in two dimensions the full classification of $L_a$-harmonic polynomials which are symmetric with respect to $\Sigma$ is given in \cite[Proposition 4.13]{SirTerTor}.
The following is a direct consequence of the previous result.
\begin{Proposition}[Strong trace unique continuation property]\label{traceuniquecontinuation}
Let $a>-1$ and $w$ be a weak solution to \eqref{eqa}. If $w(x',0)=O(|x'|^k)$ as $|x'|\to0^+$ for any $k\in\mathbb N$, then, $u$ is trivial in $B_1^+$.
\end{Proposition}

\subsection{A Liouville theorem for general exponents}
The aim of this section, is the proof of Theorem \ref{liou_a>-1}.

\begin{proof}[Proof of Theorem \ref{liou_a>-1} part i)]
   If $\mathrm{deg}(u)=1$ then $u$ is an affine function with nodal set consisting of a hyperplane and up to translations and rotations we may assume that $u(x)=x_n$ and $Z(u)=\Sigma=\{x_n=0\}$. Let us consider $w_1,w_2$ the restrictions of $w$ to $\{x_n>0\}$ and $\{x_n<0\}$ and consider their symmetric extensions across $\Sigma$. Then \cite[Lemma B.2]{DelFelVit} implies that if there exist constants $C>0$ and $\gamma\geq0$ such that
\begin{equation*}
|w_i(x)|\leq C(1+|x|)^{\gamma}\qquad\mathrm{in \ } \R^n,
\end{equation*}
then $w_i$ are both polynomials $P_i$ of degree $k_i$ at most $\lfloor \gamma \rfloor$. Then, Lemma \ref{tracepolynomial} implies that $p_i(x')=P_i(x',0)$ are polynomials of degree $k_i$, but continuity of $w$ implies that $p_1=p_2=p$ and consequently $k_1=k_2=k\leq \lfloor \gamma \rfloor$. Moreover, $P_2(x',x_n)\equiv P_1(x',x_n)$. In fact if this is not the case, one would have $P_1-P_2$ is still a nontrivial $L_a$-harmonic polynomial of degree at most $k$ with zero trace and symmetric with respect to $\Sigma$, in contradiction with Corollary \ref{tracepolynomial2}.
\end{proof}

In the same spirit of Lemma \ref{l:quasiconformal-solution}, we start by showing a remarkable feature of solutions to \eqref{e:a-entire} in the case of weights that are harmonic polynomials.
\begin{Lemma}\label{c:conformal-polynomial}
Let $u:\R^2\to\R$ be an harmonic polynomial of degree $d$ and $\O$ be one connected component of $\R^2\setminus Z(u)$ such that $0\in\partial\O$ and its boundary is connected. Then, for $a>-\min\{1,2/d\}$, if $w\in H^1_\loc(\O,|u|^a)$ satisfies
$$
\int_{\O} |u|^a \nabla w \cdot \nabla \phi \,dx dy= 0,\qquad \mbox{for every }\phi \in C^\infty_c(\R^2),
$$
then $w(x,y) = \overline{w}\left(\overline{u}(x,y),u(x,y)\right)$ in $\overline{\O}$, where:
\begin{enumerate}
    \item[\rm{(i)}] $\overline{u}$ is an harmonic conjugate of $u$ in $B_1$ such that $\overline{u}(0)=0$;
    \item[\rm{(ii)}] $\overline{w} \in H^1_\loc(\{y>0\}, |y|^a)$ satisfies$$
\int_{\{y>0\}} |y|^a \nabla \overline{w} \cdot \nabla \varphi \, dx dy= 0,\qquad \mbox{for every }\varphi \in C^\infty_c(\R^2).
$$
\end{enumerate}
\end{Lemma}
\begin{proof}
Since the proof coincides, except for minor changes, with that of Lemma \ref{l:quasiconformal-solution}, we omit some details. First, we stress that if $\Omega$ is an open connected component of $\{u\neq 0\}$ with a connected boundary, it implies that there cannot be critical points in $\Omega$. Indeed, the existence of a possible critical point \( z_0 \in \Omega \) implies the existence of an unbounded sublevel associated to the critical value \( u(z_0) \), where the polynomial \( u \) takes values in \( (0,u(z_0)) \). On the other hand, since the closure of such sublevel is strictly contained in \( \Omega \), this would imply the existence of an additional connected component of $\partial \O$, contradicting the initial assumptions.\\

\noindent Given $u$ a harmonic polynomial of degree $d\geq 2$, consider $\overline{u}$ to be its harmonic conjugate satisfying
\be\label{e:daintegrare}
\overline{u}(0)=0\qquad\mathrm{and}\qquad \nabla \overline{u} = J \nabla u.
\ee
Thus, in view of the Cauchy-Riemann equations, if we set
$\Theta(x,y) := \left(\overline{u}(x,y),u(x,y)\right)$,
by direct computation we get that
$$
|\nabla u|^2 \equiv  |\nabla \overline{u}|^2 \equiv  |\mathrm{det} D\Theta|,\qquad \nabla u\cdot \nabla \overline{u}\equiv 0.
$$
Therefore, the set of critical points of \( u \) coincides with that of \( \overline{u} \), and it follows that
$$
\overline{u}(z_0)\neq 0,\qquad \mbox{for every }z_0 \in (S(u)\cap \partial\O)\setminus \{0\}.
$$
Indeed, let $\gamma \colon I \to \R$ be a parametrization of the nodal line connecting $z_0\in (S(u)\cap \partial\O)\setminus \{0\}$ to the origin. Then, since $\nabla u(\gamma(t))\cdot \gamma'(t)=0$ for every $t \in I$, by integrating \eqref{e:daintegrare} we get
$$
\overline{u}(z_0)
= \int_{I} (J\nabla u)(\gamma(t))\cdot\gamma'(t) dt \neq 0.
$$
Similarly, one can see that at every singular point in $S(u)\cap \partial\O$, the function $\overline{u}$ attains distinct values, that is the restriction of $\Theta$ on $S(u)\cap \partial\O$ is injective. Therefore, although $S(u)\cap \partial\Omega$ is the critical set of the map \( \Theta \colon \overline{\Omega} \to \{ y \geq 0 \} \), it is possible to construct an inverse map \( \Theta^{-1} \colon \{ y \geq 0 \} \to \overline{\Omega} \) by associating each element of $\Theta(S(u)\cap \partial\Omega)$ with its corresponding singular point of $u$. Therefore, we have shown that $
\Theta\colon \overline{\O} \to \{y\geq 0\}$ is conformal and finally the proof follows as in Lemma \ref{l:quasiconformal-solution}.
\end{proof}

\begin{proof}[Proof of Theorem \ref{liou_a>-1} part ii)]
We split the proof into two steps.\\

Step 1: working on a nodal domain with connected boundary. Let $u$ be a harmonic polynomial of degree $d\geq 2$ and $\Omega$ be a connected component of $\R^2 \setminus Z(u)$ whose boundary is connected with $0\in\partial \O$. Then, by applying Lemma \ref{c:conformal-polynomial} to $w$, we can construct  $\overline{w}\in H^1_\loc(\{y>0\},|y|^a)$ satisfying
$$
\int_{\{y>0\}} |y|^a \nabla \overline w \cdot \nabla \varphi \, dx dy= 0,\qquad \mbox{for every }\varphi \in C^\infty_c(\R^2),
$$
and
$$
|\overline{w}(x,y)|\leq \tilde{C}(1+|(x,y)|)^{\overline{\gamma}} \quad \mbox{in }\{y>0\},\qquad \mbox{with }\overline{\gamma}=\frac{\gamma}{d}.
$$
Notice that the growth condition above follows by exploiting the behaviour of $u$ and its harmonic conjugate at infinity. Therefore, by applying part i) of Theorem \ref{liou_a>-1}, we deduce that $\overline w$ coincides with a polynomial $P$ of degree at most $\lfloor \gamma/d\rfloor$ and symmetric with respect to $y$-variable.
Finally, in view of Lemma \ref{c:conformal-polynomial}, we infer that
$$
w(x,y)  =  P(\overline{u}(x,y),u(x,y)), \quad \mbox{in }\Omega,
$$
with $\overline{u}$ the harmonic conjugate of $u$ such that $\overline{u}(0)=0$.\\

Step 2: extending the Liouville theorem by unique continuation principle. Up to translation, it is not restrictive to suppose that $0 \in R(u)\cap \partial \Omega$ and localize the problem near $\partial\Omega$, which is connected and piece-wise real analytic curve. Now, consider $$V(x,y):=w(x,y)-P(\overline{u}(x,y),u(x,y))$$ which still solves \eqref{e:a-entire} in the sense of Definition \ref{definition.energy.a}. Let us consider Fermi coordinates around $0$ as in \cite[Section 1.1]{SirTerVit2}. Being the curve locally real analytic, after composition $W=V\circ\Phi$ with the change of variable, we end up with a solution in $H^{1,a}(B_1)$ to $\mathrm{div}(|y|^a A\nabla W)=0$ in $B_1$, with $a>-1$ and $A(x,y)$ symmetric, uniformly elliptic, with real analytic entries of the form \eqref{e:matrice-dopo}. Denoting by $B_1^-$ the half ball where $W\equiv 0$, then on the complementary half ball $B_1^+$
\begin{equation*}
\mathrm{div}\left(y^{a}A\nabla W\right)=0 \quad\mathrm{in \ } B_1^+,\qquad
y^{a}A\nabla W\cdot  e_{y}=0 \quad\mathrm{in \ } B_1',\qquad
W=0 \quad\mathrm{in \ } B_1'.
\end{equation*}
Then Proposition \ref{traceuniquecontinuation} implies $W\equiv 0$ in $B_1^+$ too.
\end{proof}

\begin{remark}\label{rem:conn_bdry}
Disregarding the  unique continuation principle and going back to Step 1, we infer the validity of the Liouville type Theorem \ref{liou_a>-1} for solutions on a nodal component of $u$ having a connected boundary and satisfying the growth assumptions there.
\end{remark}

\end{document}